\documentclass[11p]{article}
\usepackage[a4paper, total={6in, 9in}]{geometry}

\usepackage[reqno]{amsmath}
\makeatletter
\newcommand{\leqnomode}{\tagsleft@true}
\newcommand{\reqnomode}{\tagsleft@false}
\makeatother

\usepackage{amssymb}
\usepackage{mathtools}
\usepackage{bbm}
\usepackage{tikz}
\usetikzlibrary{patterns}
\usepackage{hyperref}
\usepackage{amsthm}
\usepackage{wrapfig}
\usepackage{float}
\usepackage{booktabs}
\usepackage{tablefootnote}
\usepackage{comment}
\usepackage{nccmath}
\usepackage{pdflscape}
\usepackage{subcaption}
\usepackage{todonotes}
\usepackage{siunitx}

\usepackage{qcircuit}

\newtheorem{theorem}{Theorem}[section] 
\newtheorem{definition}[theorem]{Definition} 
\newtheorem{lemma}[theorem]{Lemma}

\newtheorem{corollary}[theorem]{Corollary}

\theoremstyle{definition}

\newtheorem{remark}[theorem]{Remark}

\allowdisplaybreaks
\usepackage{algpseudocode}
\usepackage{algorithm}
\usepackage{url}

\newcommand{\sub}{\text{sub}}

\newcommand{\Span}{\text{Span}}

\newcommand{\Aut}{\normalfont \text{Aut}}

\newcommand{\Stab}{\normalfont \normalfont \text{Stab}}

\newcommand{\Cay}{\normalfont \text{Cay}}

\newcommand{\Sym}{\normalfont \text{Sym}}

\newcommand{\Orb}{\normalfont \text{Orb}}
\newcommand{\id}{\normalfont \text{id}}
\newcommand{\Coup}{\normalfont \text{Coup}}

\newcommand{\Autbold}{\textbf{Aut}}
\newcommand{\Caybold}{\textbf{Cay}}
\newcommand{\Coupbold}{\textbf{Coup}}

\begin{document} 

\title{Exploiting Symmetries in Optimal Quantum Circuit Design}
\author{Frank de Meijer \thanks{Corresponding author. Delft Institute of Applied Mathematics, Delft University of Technology, The Netherlands, {\tt f.j.j.demeijer@tudelft.nl}} \and {Dion Gijswijt} \thanks{Delft Institute of Applied Mathematics, Delft University of Technology, The Netherlands, {\tt d.c.gijswijt@tudelft.nl}}
	\and {Renata Sotirov}  \thanks{CentER, Department of Econometrics and OR, Tilburg University, The Netherlands, {\tt r.sotirov@uvt.nl}}}
\date{}
\maketitle

\begin{abstract}
A physical limitation in quantum circuit design is the fact that gates in a quantum system can only act on qubits that are physically adjacent in the architecture. To overcome this problem, SWAP gates need to be inserted to make the circuit physically realizable. The nearest neighbour compliance problem (NNCP) asks for an optimal embedding of qubits in a given architecture such that the total number of SWAP gates to be inserted is minimized. 
In this paper we study the NNCP on general quantum architectures. Building upon an existing linear programming formulation, we show how the model can be reduced by exploiting the symmetries of the graph underlying the formulation. 
The resulting model is equivalent to a generalized network flow problem and follows from an in-depth analysis of the automorphism group of specific Cayley graphs. As a byproduct of our approach, we show that the NNCP is polynomial time solvable for several classes of symmetric quantum architectures. Numerical tests on various architectures indicate that the reduction in the number of variables and constraints is on average at least~90\%. In particular, NNCP instances on the star architecture can be solved for quantum circuits up to 100 qubits and more than 1000 quantum gates within a very short computation time. These results are far beyond the computational capacity when solving the instances without the exploitation of symmetries. 
\end{abstract}
\hfill \break

\section{Introduction}
The most commonly used model for quantum computation is that of the gated quantum computer, where a calculation is performed by executing so-called quantum circuits. A quantum circuit acts on multiple quantum bits, i.e., qubits, 
\textcolor{black}{which are logical units of operation.}
Whereas classical bits exclusively take the Boolean values zero or one, qubits can be in a superposition state, which upon measurement are displayed as zero or one with a certain probability. A quantum circuit sequentially acts on the qubits via quantum gates, which are unitary transformations that sequentially adjust the state of one or more qubits to perform an operation. 
Quantum circuits extend on the gate model for classical computing, and hence, a quantum computer can perform any computation that a classical computer can perform~\cite{NielsenChuang}. However, based on quantum phenomena such as superposition and entanglement, a quantum system is able to perform a much broader spectrum of operations. For an extensive overview of the advances and applications of quantum computing, see e.g.,~\cite{Montanaro}.

\textcolor{black}{Although there do exist quantum hardwares that act on more than two qubits simultaneously~\cite{rajakumar2022}}, many physical implementations of quantum gates operate on only one or two qubits at a time~\cite{HaffnerEtAl,NielsenChuang,Qiskit}. In this setting, gates that act on more than two qubits therefore need to be realized as a sequence of gates of size at most two, which, fortunately, is possible for any quantum gate~\cite{NielsenChuang}. For instance, the set of one-qubit gates and two-qubit controlled-NOT gates is universal~\cite{BarencoEtAl}, meaning that this set is sufficient to perform any quantum computation. 

The qubits in a quantum system are physically embedded in a certain design, i.e., the quantum architecture. This architecture is commonly represented as a coupling graph, where the vertices represent the qubits and an edge is drawn between two qubits whenever the qubits can communicate in the quantum system. With ``communicate", we refer to the possibility to apply a gate to the two qubits and consequently affect their simultaneous state.
Among the special coupling graphs considered in the literature are the linear array, see e.g.,~\cite{Bhattacharjee, ChengEtAl, HirataEtAl, KoleEtAl, MulderijEtAl}, the two-dimensional grid, see e.g.,~\cite{AlfailakawiEtAl,Bhattacharjee2, ChoiVanMeter}, the three-dimensional grid \cite{FarghadanMohammadzadeh}, the IBM QX architecture, see e.g.,~\cite{WilleEtAl}, but also general coupling graphs~\cite{BhattacharjeeGeneral,SiraichiEtAl,VenturelliEtAl,LiEtAl, CowtanEtAl, ItokoEtAl}. 

A physical limitation of the architecture is that two-qubit gates can only be applied when the qubits are physically adjacent to each other in the coupling graph. These restrictions are known as nearest neighbour constraints and have been subject of interest in the design of quantum realizations of specific circuits, see e.g.,~\cite{FowlerEtAl}, or the design of quantum architectures itself, see~\cite{MulderijEtAl} and the references therein. Instead of research on quantum realizations that comply with the nearest neighbour constraints, we can also disregard these constraints at first and alter existing quantum circuits to make them feasible, which will be the followed approach in this paper. 

A quantum circuit can be made compliant with respect to the nearest neighbour constraints by the insertion of SWAP gates. A SWAP gate acts on two adjacent qubits by interchanging their location in the coupling graph\footnote{Strictly speaking, a SWAP gate does only interchange the state of the involved qubits, while the actual hardware entities remain unchanged in the architecture.}. If the coupling graph is connected, any quantum circuit can be made compliant by the insertion of a finite number of SWAP gates and there are often many ways to do so. However, due to a qubit's interaction with its environment \cite{DiVincenzo},  quantum systems currently still suffer from physical instability of qubit states after some period of time. This raises the desire for quantum circuits with as few gates as possible. We therefore prefer to add the minimum number of SWAP gates in order to make a circuit compliant. 

Given a quantum circuit and a coupling graph, the nearest neighbour compliance problem (NNCP) asks for an optimal sequential allocation of the qubits over the quantum architecture such that the total number of SWAP gates to be inserted is minimized. With ``sequential", we refer to the decision variables to not only concern the initial allocation, but also the actual SWAP operations that take place over time. The NNCP was proven to be $\mathcal{NP}$-hard via a reduction from the token swapping problem~\cite{SiraichiEtAl}. 

Most research on the NNCP has been on heuristic methods, such as greedy methods~\cite{AlfailakawiEtAl, HirataEtAl}, harmony search~\cite{AlfailakawiEtAl}, optimal linear arrangement~\cite{PedramShafaei} and receding horizon methods~\cite{HirataEtAl, KoleEtAl, ShafaeiEtAl, WilleEtAl2}. Exact approaches to tackle the NNCP include exhaustive search~\cite{DingYamashita,HirataEtAl}, explicit cost enumeration~\cite{WilleEtAl3} and linear programming (LP) based methods on the adjacent transposition graph~\cite{MatsuoYamashita,Mulderij}. All these methods embrace an implicit factorial scaling in the number of qubits, due to the inherited total number of possible assignments of the qubits. Recently, also polynomial sized models have been considered that are based on mixed-integer linear programming~\cite{MulderijEtAl, VanHouteEtAl}. The construction considered in~\cite{MulderijEtAl} is based on the linear array coupling graph, while the models in~\cite{VanHouteEtAl} consider ordering problems for distributed quantum computing. An integer programming approach for a very related, but slightly more general problem called the qubit routing problem, has been considered in~\cite{NanniciniEtAl}.
Other research focuses on a related version of the NNCP, where an initial qubit ordering has to be realized that minimizes the (approximated) number of SWAP operations, without actually considering the exact insertions into the quantum circuit, see~\cite{ShafaeiEtAl2, KoleEtAl2,KoleEtAl3}. 

\medskip

Building upon the shortest path formulation considered in~\cite{MatsuoYamashita,Mulderij}, a main feature of our approach concerns the exploitation of symmetries in the model. The literature on symmetry reduction methods in mathematical optimization is extensive, and we refer the reader to~\cite{Liberti, MargotOverview} for comprehensive overviews in this direction. It is well-known that symmetries in integer linear programming (ILP) problems lead to poor behaviour of numerical algorithms, due to the costly duplication of computational effort in branching approaches. To reduce this negative effect, symmetries need to be broken, e.g., by perturbation, symmetry-breaking inequalities (e.g.,~\cite{MargotInequalities2003}) or specialized branching techniques (e.g.,~\cite{SheraliSmith}). The literature on symmetry reduction for integer linear programs (ILPs) can be distinguished between problem-based approaches, whose symmetry groups are known a priori (see e.g.,~\cite{LeeMargot}), or generic techniques. The latter class on one hand contains methods based on branching tree reductions, such as isomorphism pruning~\cite{MargotPruning2002, MargotOrbits2003} and orbital branching~\cite{OstrowskiEtAl}. Alternative methods mainly consider symmetry-handling constraints to restrict the feasible region of an optimization problem by eliminating symmetric solutions. Two well-known streams in this direction are the utilization of  orbitopes~\cite{KaibelPfetsch} and fundamental domains~\cite{Friedman}. Branching tree reductions and symmetry-handling constraints can also be combined, see e.g.,~\cite{VanDoornmalenHojny}.  

When considering symmetry reduction methods for linear programs, a major research line considers the study of symmetric polyhedra, see~\cite[Section~6]{MargotOverview} and the references therein. Another research line considers the exploitation of symmetries in the simplex \mbox{algorithm~\cite{TuckerMatrixGame,TuckerCombinatorial}}. B\"odi et al.~\cite{BodiEtAl} consider the exploitation of symmetries in linear programs by restricting to the subspace of fixed points under a linear map induced by the symmetries in the program. This approach can be generalized to convex programs and is closely related to the invariant-based symmetry reduction approaches applied to conic and semidefinite programs, see e.g., Gatermann and Parrilo~\cite{GatermannParrilo}, to which our reduction method also belongs.

\subsection*{Main results and outline}
In this paper we consider the nearest neighbour compliance problem on general coupling graphs. Following the linear programming (LP) formulation derived in~\cite{MatsuoYamashita}, we analyse the group symmetry of the underlying graph, which is a sequence of connected Cayley graphs. By exploiting these symmetries, we reduce the LP model in the number of variables and constraints, leading to a symmetry-reduced formulation for solving the NNCP. We show the theoretical and practical strength of our approach for several classes of symmetric coupling graphs for which the reduction is most significant, namely the graphs that embrace a large automorphism group. 

The LP formulation of~\cite{MatsuoYamashita} can be viewed as a single-pair shortest path problem on a directed graph that we refer to as the graph $X = (V,A)$. \textcolor{black}{The graph $X$ is composed of layers.} As a first step in our approach, we consider the automorphism group of the subgraphs of $X$ \textcolor{black}{induced by each layer}. Each subgraph is a Cayley graph of the symmetric group~$\mathbb{S}_n$ generated by the edges in the coupling graph of the quantum architecture. We derive the full automorphism group of such Cayley graphs, after which we extend these automorphism results to derive the automorphism group $G_X$ of $X$. We also study the orbit and orbital structure of the group action of $G_X$ on $X$. The results on the group structure of these Cayley graphs are in itself interesting, as such graphs are of main importance in interconnection networks~\cite{GanesanOverview,Heydemann}.

By averaging over each orbital of the action of $G_X$ on $X$ via the Reynolds operator, we show how the LP formulation can be reduced following the approach of~\cite{BodiEtAl}. We show that the resulting reduced LP formulation is  equivalent to a generalized network flow problem on an auxiliary graph following from our construction. For symmetric coupling graphs, this reduced LP formulation is significantly smaller in size. As a byproduct of our approach, we show that the NNCP is polynomial time solvable for coupling graph whose automorphism group scales factorially in the number of qubits, e.g., the star graph or complete bipartite graphs with one of the sizes fixed. The construction of the reduced LP formulation follows completely from the algebraic analysis of $X$ and does not rely on the use of any external algebraic software.

Although the ingredients of our approach are presented generally, we explicitly show how the reduced LP can be constructed for three special graph types: the cycle graph, the star graph and the biclique graph. For each of these classes, we show how the orbital structure unfolds by analyzing a specific subgroup of the automorphism group of the coupling graphs.

Finally, we test our symmetry-reduced formulation on real and randomly generated quantum circuits defined on the above-mentioned coupling graphs. Our numerical tests confirm that the effort spent in the algebraic analysis pays off, as computation times to solve an instance are several orders of magnitude smaller compared to the nonreduced model. Whereas the model from~\cite{MatsuoYamashita} can only solve instances up to 8 qubits, the largest instances we solve contain up to 100 (resp.~40) qubits and several hundreds of quantum gates on the star (resp.~biclique) coupling graph. Observe that such instances are far out of reach for the nonreduced model, as this would require the use of at least $100! \approx 9.33 \cdot 10^{157}$ constraints and even more variables.

\medskip

This paper is structured as follows. Section~\ref{Sec:NNCP} formally introduces the NNCP and reviews the shortest path formulation of~\cite{MatsuoYamashita}. In Section~\ref{Sec:ExploitSym} we analyse the automorphism group of the graph underlying the formulation, as well as its orbit and orbital structure. These algebraic properties are exploited in Section~\ref{Sec:SymmetryReduction}, where we present our symmetry-reduced NNCP formulation. In Section~\ref{Sec:SpecialCoupling} we apply our approach to several specific types of coupling graphs. Computational results are discussed in Section~\ref{Sec:CompResults}. 

\subsection*{Notation}
A directed graph is given by a pair $(V,A)$, where $V$ is a vertex set and $A \subseteq V \times V$ an arc set. For $i \in V$ and $A' \subseteq A$, we let $\delta^+(i, A')$ (resp.~$\delta^-(i, A')$) denote the set of arcs in $A'$ that leave (resp.~enter) vertex $i$. In case $A' = A$, we just write $\delta^+(i)$ (resp.~$\delta^-(i)$). {\color{black} An undirected graph is given by a pair $(V,E)$, where $V$ is a vertex set and $E \subseteq V^{(2)}$ is an edge set, where $V^{(2)}$ consists of all two-element subsets of $V$.}

The set of integers $\{1, \ldots, k\}$ is denoted by $[k]$. For a subset $S$ of a  finite set $T$, we denote by~$\mathbbm{1}_S \in \{0,1\}^T$ the indicator vector of $S$ in $T$. 

For a group $G$, we denote by $\id_G$ (or simply $\id$) its identity element. 
When $G$ acts on a set $X$, we denote by $\Orb(x) := \{g(x) \, : \, \, g \in G\} \subseteq X$ the orbit of $x \in X$ under the action of $G$. The set of orbits of $X$ under $G$ is denoted by the quotient $X / G$. For any $g \in G$, we let $X^g := \{ x \, : \, \, g(x) = x\}$ be the set of fixed points of $g$.

The symmetric group on a finite set $Y$ is denoted by $\Sym(Y)$. When $Y = [n]$, its symmetric group is shortened to $\mathbb{S}_n$. A permutation $\tau \in \mathbb{S}_n$ can be written in one-line notation, i.e., as an ordered array~$(\tau(1), \tau(2), \ldots, \tau(n))$ of images of $\{1, \ldots, n\}$ under $\tau$. Alternatively, $\tau$ can be written in the usual cycle notation of permutations. Permutations that only consist of a 2-cycle are called transpositions. For $S \subseteq [n]$ and $\tau \in \mathbb{S}_n$, we define the sets $\tau(S) := \{\tau(s) \, : \, \, s \in S\}$ and~$\tau^{-1}(S) := \{\tau^{-1}(s) \, : \, \, s \in S\}$. Moreover, we let $\mathbb{S}_n(S) := \{\tau \in \mathbb{S}_n \, : \, \, \tau(S) = S\}$ denote the setwise stabilizer of $S$ under $\mathbb{S}_n$.

\section{Nearest neighbour compliance problem} \label{Sec:NNCP}
A given quantum circuit can be made feasible with respect to the adjacent interaction constraints by inserting SWAP gates. Although these do not interfere with the functionality of the quantum circuit, the total number of gates is favoured to be as small as possible {\color{black}for several reasons. 
Executing a quantum gate is often computationally expensive and introduces some error due to environmental noise~\cite{ge2024}. Moreover, it is well-known that qubits have a limited coherence time, inducing the need for quantum circuit of small depth.}
The nearest neighbour compliance problem (NNCP) aims at finding an embedding of the qubits over a given architecture such that the number of SWAP gates needed to make the final circuit feasible with respect to the adjacent interaction constraints is minimized.

In this section we formally introduce the nearest neighbour compliance problem as a shortest path problem.  

\subsection{Mathematical formulation of the NNCP}
\label{Subsec:NNCP}
We make two model assumptions about the quantum circuits under consideration. First, quantum gates that act on a single qubit always comply with the adjacent interaction constraints and are therefore not taken into consideration. Second, it only makes sense to talk about adjacency in the context of two-qubit quantum gates. If a quantum gate acts on more than two qubits, we first decompose it into two-qubit gates. This is always possible \cite{NielsenChuang} and there exist a large variety of ways for doing this. Throughout this paper, we assume without loss of generality that quantum circuits consist of a sequence of two-qubit gates.

\medskip 

Let $Q = {\color{black}\{q_1, q_2, \ldots, q_n\}}$ denote the set of qubits of the quantum system. The qubits need to be embedded in a certain topology, that we refer to as the architecture of the quantum system. This architecture is fixed and can be modeled as a graph $(L, E)$. Here $L = [n]$ denotes a set of physical locations and $E \subseteq L^{(2)}$ is the adjacency structure of the architecture. That is, if~${\{i,j\} \in E}$, then locations~$i$ and~$j$ are physically adjacent to each another and can therefore directly share information. The graph is denoted as the {coupling graph} of the quantum system and denoted by $\Coup(E) := (L,E)$. We assume that $(L,E)$ is connected, which implies that all pairs of locations can indirectly share information.

Each qubit in $Q$ needs to be assigned to a physical location in $L$. A bijection $\tau : L \rightarrow Q$ is called a {qubit order}. To present a qubit order, we use one-line notation with respect to the images in $Q$. For example, the order
\begin{align*} 
\tau = (\tau(1), \tau(2), \tau(3), \tau(4)) = ({\color{black}q_2, q_3, q_1, q_4})
\end{align*} 
corresponds to the assignment where qubit ${\color{black}q_2}$ is on location 1, qubit ${\color{black}q_3}$ on location~2, qubit~${\color{black}q_1}$ on location~3 and qubit~${\color{black}q_4}$ on location~4. The set of all qubit orders on $n$ qubits is equal to~$\mathbb{S}_n$. 

A SWAP gate interchanges the qubits on two locations in the embedding. It can also be modeled as an element $\sigma \in \mathbb{S}_n$, where $\sigma$ is a {transposition}. Using cycle notation, the SWAP gate $\sigma = (i~j)$ \textcolor{black}{applied to} the qubit order $\tau$ interchanges the qubits $\tau(i)$ and $\tau(j)$. Applying this SWAP gate can be seen as a right action of $\sigma$ on $\mathbb{S}_n$, i.e.,
\begin{align*}
\tau \circ \sigma & = (\tau(1), \tau(2), \ldots, \tau(i), \ldots, \tau(j), \ldots , \tau(n))\circ(i~j) \\
& = (\tau(1), \tau(2), \ldots, \tau(j), \ldots, \tau(i), \ldots, \tau(n)),
\end{align*}
for all $\tau \in \mathbb{S}_n$. To simplify notation, we omit the $\circ$ in group actions and just write $\tau \sigma$ in the sequel.
\begin{remark}
Although both elements of $\mathbb{S}_n$, $\tau$ represents a qubit order, while $\sigma$ represents a SWAP gate. To discriminate between these objects, we always use one-line notation for qubit orders and cycle notation for SWAP gates throughout the paper. 
\end{remark}

A SWAP gate can only be applied to qubits on locations that are adjacent in $\Coup(E)$. Whenever there is an edge $\{i,j\} \in E$, the SWAP gate $(i~j)$ acts on adjacent locations. Let
\begin{align} \label{Def:T}
T := \left\{ (i~j) \in \mathbb{S}_n \, : \, \, \{i,j\} \in E \right\}
\end{align}
denote the set of transpositions that correspond to a \textcolor{black}{feasible} SWAP gate in the quantum system. 

Given two qubit orders $\tau_1, \tau_2 \in \mathbb{S}_n$, we are interested in the minimum number of SWAP gates that need to be applied to $\tau_1$ to obtain $\tau_2$ by only using SWAP gates from $T$. Let~${J_T: \mathbb{S}_n \times \mathbb{S}_n \to \mathbb{Z}_+}$ be defined as
\begin{align} \label{def:metricJ}
J_T(\tau_1, \tau_2) := \min \{ k \, : \, \, \tau_2 = \tau_1 \sigma_1  \sigma_2  \dots  \sigma_k, \, \sigma_1, \ldots, \sigma_k \in T\}, 
\end{align}
which forms a metric on all qubit orders and depends on the quantum architecture $T$. 
Observe that this metric is left-invariant, i.e., $J_T (\tau_1, \tau_2) = J_T (\pi \tau_1, \pi \tau_2)$ for all $\pi \in \mathbb{S}_n$, implying that $J_T(\tau_1, \tau_2)$  equals the length of the shortest sequence of transpositions of $T$ needed to generate $\tau_2^{-1}\tau_1$. It is known that finding such minimum-length sequence is in general $\mathit{PSPACE}$-complete~\cite{Jerrum}. For special types of coupling graphs, however, the metric $J_T$ is computationally tractable, e.g., when $\Coup(E)$ is a path or the complete graph. For these cases, $J_T$ coincides with the Kendall tau distance and the Cayley distance, respectively.

Let ${\color{black} q_1, q_2} \in Q$ be two qubits such that {\color{black}$q_1 \neq q_2$}. Then the unordered pair {\color{black} $g_{q_1q_2} =\{q_1,q_2\}$} is a two-qubit quantum gate that acts on qubits {\color{black}$q_1$ and $q_2$}. Whenever the specific qubits on which the gate acts are irrelevant, we sometimes omit the subscripts. A finite sequence~${C = (g^1, \ldots, g^m)}$ of gates~${g^1, \ldots, g^m}$ is called a gate sequence of size $m$. 
Given a set of qubits $Q$ and a gate sequence $C$, the tuple~${\Gamma = (Q,C)}$ is called a {quantum circuit}. 

We say that a qubit order $\tau$ complies with a gate {\color{black}$g_{q_1q_2}$} if qubits {\color{black}$q_1$ and $q_2$} are adjacent in $\tau$ with respect to the coupling graph $\Coup(E)$, i.e., if $\tau^{-1}({\color{black}g_{q_1q_2}}) = \{\tau^{-1}({\color{black}q_1}), \tau^{-1}({\color{black}q_2})\} \in E$. We now formulate the NNCP.
\begin{definition}[NNCP] \label{Def:NNC}
Let $\Gamma = (Q,C)$ be a quantum circuit with $n$ qubits and $m$ gates, and let $\Coup(E) = (L,E)$ be the coupling graph of the underlying architecture. Then, the nearest neighbour compliance problem asks for a sequence of qubit orders $\tau^k$, $k \in [m]$, each one corresponding to an order prior to applying a gate of $C$, such that $\sum_{k = 1}^{m-1}J_T(\tau^k, \tau^{k+1})$ is minimized and such that $\tau^k$ complies with $g^k$ for all $k \in [m]$.
\end{definition}
The NNCP as presented in Definition~\ref{Def:NNC} is known to be $\mathcal{NP}$-hard in general \cite{SiraichiEtAl}. 

We end this section by introducing the notion of the so-called gate graph, which captures the underlying qubit dependencies imposed by the gates in the circuit.
\begin{definition} \label{Def:GateGraph}
Let $\Gamma = (Q,C)$ be a quantum circuit. The gate graph $(Q,U)$ is an undirected graph that has vertex set $Q$ and edge set $U = \{g \, : \, \, g \in C\}$. 
\end{definition}
The gate graph $(Q,U)$ will be exploited in Section~\ref{Subsec:automorphismX}. 

\subsection{The NNCP as a shortest path problem}
\label{Subsec:ShortestPath}
In this section we show how the NNCP can be modeled as a shortest path problem in a directed graph following the construction of \cite{MatsuoYamashita, Mulderij}.

Let an instance of the NNCP as defined in Section~\ref{Subsec:NNCP} be given. Of key importance in the reduction to a shortest problem is the notion of a Cayley graph.

\begin{definition}[Cayley graph] \label{Def:Cayley}
Let $G$ be a finite group and let $S$ be a subset of $G$ such that~${\id_G \notin S}$ and~$S = S^{-1} := \{s^{-1} \, : \, \,s \in S\}$. The Cayley graph $\Cay(G,S)$ on $G$ with respect to $S$ is defined as the (directed) graph with vertex set $G$ and arc set $\{(g, gs) \, : \, \, g \in G, s \in S\}$. 
\end{definition}
Observe that $\Cay(G,S)$ as in Definition~\ref{Def:Cayley} contains an arc if and only if it also contains the reversed arc. Although this suggests that any $\Cay(G,S)$ is undirected, we stick to the setting of two reversed directed arcs, since we will employ the Cayley graphs as subgraphs of a larger directed graph.

Let $H := \Cay(\mathbb{S}_n, T)$, where $T$ is given by~\eqref{Def:T}. More precisely, the vertex and arc set of $H$ are given by $V(H) := \mathbb{S}_n$ and $A(H) := \left\{(\tau, \tau  \sigma) \, : \, \, \tau \in \mathbb{S}_n, \sigma \in T\right\}$,
respectively. Each vertex in $V(H)$ represents a qubit order, while an arc in $A(H)$ represents a SWAP gate that translates a qubit order into another qubit order with respect to the coupling graph. Now, we define the subgraphs $H^k$ for~$k \in [m]$ as disjoint copies of $H$, one for each gate in the circuit. 

The $m$ subgraphs $H^k$ are merged to obtain a graph $X = (V,A)$.
The vertex set~$V$ of~$X$ consists of the union of all $V^k$, $k \in [m]$, as well as a source $s$ and sink $t$, i.e., $V = \{s\} \cup V^1 \cup \dots \cup V^m \cup \{t\}$. Since the subgraphs $H^1, \ldots, H^m$ are identical, we use superscripts to indicate to which subgraph a vertex belongs. For example, $\tau^k$ and $\tau^{k+1}$ correspond to the same qubit order in subgraph $k$ and~$k + 1$, respectively.  

The arc set $A$ of $X$ contains the union of all $A^k$, $k \in [m]$. Moreover, the arcs between different subgraphs are introduced by the following sets: 
\begin{align}
\begin{aligned} D^0 & := \{(s,\tau^1) \, : \, \, \tau^1 \in V^1\} \\
D^k & := \{(\tau^k, \tau^{k+1}) \, : \, \, \tau^k \in V^k, \tau^{k+1} \in V^{k+1},~ (\tau^k)^{-1}(g^k) \in E\}, \quad k \in [m-1] \\
D^m & := \{(\tau^m, t) \, : \, \, \tau^m \in V^m,~(\tau^m)^{-1}(g^m) \in E\}. \end{aligned} \label{Def:Dk}
\end{align}
These sets can be interpreted as follows. The set $D^0$ contains an arc from $s$ to all nodes in~$H^1$. For all $k \in [m-1]$, $D^k$ contains the connecting arcs from $H^k$ to $H^{k+1}$. Suppose the gate~$g^k$ acts on qubits {\color{black}$q_1$ and $q_2$}. Then we include an arc from a qubit order~$\tau^k$ in $H^k$ to the same qubit order $\tau^{k+1}$ in $H^{k+1}$ if and only if {\color{black}$q_1$ and $q_2$} are adjacent in $\tau^k$ with respect to $\Coup(E)$. That is, whenever~$(\tau^k)^{-1}(g^k) = \{(\tau^k)^{-1}({\color{black}q_1}),~ (\tau^k)^{-1}({\color{black}q_2})\} \in E$. Similarly, $D^m$ contains all arcs from $\tau^m$ with this property to the sink node $t$. Now, the arc set $A$ of $X$ is given by
\begin{align*}
A = A^1 \cup \cdots \cup A^m \cup D^0 \cup D^1 \cup \cdots \cup D^m. 
\end{align*}
We set the cost of each arc in $A^k, k \in [m]$, equal to one, as traversing these arcs corresponds to applying one SWAP gate. The cost of the arcs in $D^k, k = 0, \ldots, m$, is equal to zero, as no SWAP gates are applied when moving from a subgraph to the next.

This construction implies the following result.

\begin{theorem}[\cite{Mulderij}]
Any $(s,t)$-path in $X$ {\color{black}induces} a sequence $(\tau^1, \ldots, \tau^m)$ of qubit orders that all comply with the adjacent interaction constraints. A shortest $(s,t)$-path in $X$ corresponds to an optimal solution of the NNCP.
\end{theorem}

There are many algorithms in the literature for solving the shortest path instance, e.g., Dijkstra's algorithm with Fibonacci heaps \cite{FredmanTarjan}. Alternatively, we can solve it as a linear programming (LP) problem. For all $k \in [m]$ and $e \in A^k$, let $x_{e}$ denote a variable that is one if arc $e$ is used on a path, and zero otherwise. Similarly, for all $k \in \{0\} \cup [m]$ and $e \in D^k$, let $y_e$ denote a variable that is one if arc $e$ is used on a path, and zero otherwise. Then the shortest $(s,t)$-path in $X$ can be found by solving the following LP:
\leqnomode
\begin{align} \tag{SPP} \label{Prob:SPP}
\hspace{0.5cm} \begin{aligned}
\quad~~\min \, & \sum_{k =1}^m \sum_{e \in A^k}x_e \\
\text{s.t.}\, & \sum_{e \in D^0}y_e = 1, \quad \sum_{e \in D^m} y_e = 1 \\
& \sum_{e \in \delta^-(\tau,D^{k-1})} \hspace{-0.35cm} y_e + \sum_{e \in \delta^-(\tau, A^k)}\hspace{-0.35cm}x_e = \sum_{e \in \delta^+(\tau,D^k)}\hspace{-0.35cm}y_e + \sum_{e \in \delta^+(\tau, A^k)}\hspace{-0.35cm}x_e ~~\forall \tau \in V^k, k\in [m]\\
& 0 \leq x_e \leq 1 \quad \forall e \in A^k,~k \in [m], \\
 & 0 \leq y_e \leq 1 \quad \forall e \in D^k,~k \in \{0\} \cup [m].  
\end{aligned} 
\end{align}

\reqnomode

\section[Symmetries in $X = (V,A)$]{Symmetries in \boldmath $X = (V,A)$} \label{Sec:ExploitSym}
The graph $X$ constructed in Section~\ref{Subsec:ShortestPath} contains $\Theta(m n!)$ vertices and $\Theta(|E| mn!)$ arcs. The bottleneck in solving the NNCP to optimality is clearly the factorial scaling in the number of qubits. Fortunately, for many structured quantum system architectures, the problem can be reduced by exploiting the symmetries in $X$. In this section we study these symmetries in terms of the automorphism group of $X$.

In Section~\ref{Subsec:automorphismCay} and~\ref{Subsec:automorphismX} we study the automorphism group of Cayley graphs generated by transpositions and the automorphism group of $X$, respectively. In Section~\ref{Subsec:Orbit} we study the orbit and orbital structure induced by this group action on $X$. The results in this section are the key ingredients of the symmetry reduction explained in Section~\ref{Sec:SymmetryReduction}.

\subsection[Automorphism group of $\Aut(\Cay(\mathbb{S}_n, T))$]{Automorphism group of \boldmath $\Autbold(\Caybold(\mathbb{S}_n, T))$}
\label{Subsec:automorphismCay}
For a directed graph $X$ with vertex set $V$ and arc set $A$, a permutation $\rho \in \Sym(V)$ is called an automorphism of $X$ if $(\rho(i), \rho(j)) \in A$ if and only if $(i,j) \in A$. We also say that such $\rho$ acts on $X$. The automorphism group of $X$ is the group of all automorphisms of $X$ and is denoted by $\Aut(X)$.

In order to determine the automorphism group of the graph $X$ introduced in Section~\ref{Subsec:ShortestPath}, we start by considering the automorphism group of the subgraphs $H^k$, $k \in [m]$. Recall that all $H^k$ are identical and equal to $\Cay(\mathbb{S}_n, T)$, where $T$ is a set of transpositions, see~\eqref{Def:T}. Hence, the goal of this subsection is to study $\Aut(\Cay(\mathbb{S}_n, T))$. 

\medskip

There exist several works in the literature on the automorphism group of Cayley graphs generated by transpositions. As indicated by Feng~\cite{Feng}, we can show that $\mathbb{S}_n$ acts on $\Cay(\mathbb{S}_n, T)$~by left multiplication. That is, for any $a \in \mathbb{S}_n$ the mapping $\tau \mapsto a \tau$ defines an automorphism of $\Cay(\mathbb{S}_n, T)$. All such automorphisms form a subgroup of $\Aut(\Cay(\mathbb{S}_n, T))$. We can also show that the group $\Aut(\Coup(E))$ acts on $\Cay(\mathbb{S}_n, T)$ by right multiplication via the mapping $\tau \mapsto \tau b^{-1}$, which is an automorphism of $\Cay(\mathbb{S}_n, T)$ for all $b \in \Aut(\Coup(E))$. To verify this, let $(\tau_1, \tau_2)$ be an arc in $\Cay(\mathbb{S}_n, T)$. Then $\tau_2 = \tau_1 \sigma_1$ for some $\sigma_1 \in T$. The image of this arc under the action of an element~$b \in \Aut(\Coup(E))$ is
\begin{align*}
    (\tau_1 b^{-1},\, \tau_2 b^{-1}) = (\tau_1 b^{-1}, \, \tau_1 \sigma_1 b^{-1}) = (\tau_1 b^{-1}, \, \tau_1 b^{-1} b\sigma_1 b^{-1}).
\end{align*}
It is well-known that if a permutation maps $i$ to $j$, then the conjugate of this permutation by~$b$ maps~$b(i)$ to~$b(j)$. Therefore, if $\sigma_1 = (i~j)$, then $\sigma_2 := b \sigma_1 b^{-1} = (b(i)~b(j))$. Since $b$ is an automorphism of $\Coup(E)$, $\sigma_2 \in T$, which implies that $(\tau_1 b^{-1}, \, \tau_2 b^{-1})$ is again an arc of $\Cay(\mathbb{S}_n, T)$. Since~$\tau \mapsto \tau b^{-1}$ is bijective, it follows that $\Aut(\Coup(E))$ indeed acts on $\Cay(\mathbb{S}_n, T)$ by right multiplication.  

We now show how both group actions are combined in order to obtain a subgroup of $\Aut(\Cay(\mathbb{S}_n, T))$. Let us define the mapping $\theta: \mathbb{S}_n \times \Aut(\Coup(E)) \rightarrow \Aut(\Cay(\mathbb{S}_n,T))$ given by
\begin{align} \label{Eq:theta}
      \theta(a,b) := (\tau \mapsto a \tau b^{-1}). 
\end{align}
Indeed, $\theta(a,b)$ is the composition of an action by left multiplication by an element $a \in \mathbb{S}_n$ and a right multiplication by an element $b \in \Aut(\Coup(E))$ (in arbitrary order). So, for all~$(a,b)$ in its domain, $\theta(a,b)$ is indeed an automorphism of $\Cay(\mathbb{S}_n,T)$. We can show that the map $\theta$ is a group homomorphism that is injective.

\begin{theorem} \label{Thm:DirectSubgroup}
For $n \geq 3$, the mapping $\theta$ is a group homomorphism from $\mathbb{S}_n \times \Aut(\Coup(E))$ to $\Aut(\Cay(\mathbb{S}_n, T))$ that is injective. 
\end{theorem}
\begin{proof}
We start by showing that $\theta$ is indeed a group homomorphism.
Let $(a_1, b_1), (a_2, b_2) \in \mathbb{S}_n \times \Aut(\Coup(E))$. Then, for all $\tau \in \mathbb{S}_n$: 
\begin{align*}
\theta\left((a_1, b_1)(a_2, b_2) \right)(\tau) & = \theta \left( (a_1a_2, b_1b_2) \right)(\tau) =  a_1 a_2 \tau (b_1 b_2)^{-1}  =  a_1 a_2 \tau b_2^{-1} b_1^{-1}  \\
\theta((a_1, b_1)) \theta((a_2, b_2))(\tau) & =  \theta(a_1, b_1)(a_2 \tau b_2^{-1}) =  a_1a_2 \tau b_2^{-1}b_1^{-1}.
\end{align*}
Hence, $\theta$ is a group homomorphism. To prove injectivity, assume that $(a_1, b_1), (a_2, b_2) \in \mathbb{S}_n \times \Aut(\Coup(E))$ are such that $\theta((a_1, b_1)) = \theta((a_2, b_2))$. Then, $a_1 \tau b_1^{-1} = a_2 \tau b_2^{-1}$ for all~${\tau \in \mathbb{S}_n}$. In particular, this must hold for $\tau = \id$, from which it follows that $a_1 b_1^{-1} = a_2b_2^{-1}$, and hence,~${a_2 = a_1 b_1^{-1}b_2}$. Substituting this into~${a_1 \tau b_1^{-1} = a_2 \tau b_2^{-1}}$, yields
\begin{align*}
a_1 \tau b_1^{-1} = a_1 b_1^{-1} b_2 \tau b_2^{-1}~~\forall \tau \in \mathbb{S}_n, \quad  \text{ or equivalently,} \quad   \tau b_1^{-1} b_2 = b_1^{-1}b_2 \tau~~\forall \tau \in \mathbb{S}_n. 
\end{align*}
This implies that $b_1^{-1}b_2 \in Z(\mathbb{S}_n) := \{ g \in \mathbb{S}_n \, : \, \, gh = hg~~\forall h \in \mathbb{S}_n \}$. It is well-known that the center~$Z(\mathbb{S}_n)$ is trivial for $n \geq 3$, hence $b_1 = b_2$. From this, it simply follows that also $a_1 = a_2$, hence~$\theta$ is injective.
\end{proof}

Theorem~\ref{Thm:DirectSubgroup} shows that the image of $\mathbb{S}_n \times \Aut(\Coup(E))$ under $\theta$ is a subgroup of $\Aut(\Cay(\mathbb{S}_n, T))$, which is isomorphic to $\mathbb{S}_n \times \Aut(\Coup(E))$ by the injectivity of $\theta$.

\medskip

The group $\mathbb{S}_n \times \Aut(\Coup(E))$ is the full automorphism group if the mapping $\theta$ is a bijection. A sufficient condition for this to be true is the normality of $\Cay(\mathbb{S}_n, T)$. We call the Cayley graph $\Cay(\mathbb{S}_n, T)$ {normal} if the subgroup of all automorphisms by left multiplication by elements of $\mathbb{S}_n$, i.e., $\{(\tau \mapsto a \tau ) \, : \, \, a \in \mathbb{S}_n \}$, is a normal subgroup of $\Aut(\Cay(\mathbb{S}_n, T))$. It has been shown by Ganesan~\cite{Ganesan2016} that $\Cay(\mathbb{S}_n, T)$ is normal if and only if $\Aut(\Cay(\mathbb{S}_n, T)) \cong \mathbb{S}_n \times \Aut(\Coup(E))$. Moreover, several sufficient conditions for normality have been derived in the literature~\cite{Ganesan2013, Ganesan2015}. It was recently shown by Gijswijt and De Meijer~\cite{GijswijtDeMeijer} that almost all graphs of the form $\Cay(\mathbb{S}_n,T)$ are normal, which immediately implies the following result.

\begin{theorem}[\cite{GijswijtDeMeijer}] \label{Thm:AllNormal}
    Let $\Coup(E)$ be connected and not equal to the 4-cycle $C_4$ or the complete graph $K_n$. Then, $\Aut(\Cay(\mathbb{S}_n,T)) \cong \mathbb{S}_n \times \Aut(\Coup(E))$.
\end{theorem}
In case $\Coup(E)$ is $C_4$ or $K_n$, the automorphism group of $\Cay(\mathbb{S}_n, T)$ is also known, see~\cite[Section 3]{Ganesan2013} and \mbox{\cite[Theorem 1.1]{Ganesan2015}}, respectively. 

In the settings that are interesting for our application, $\Coup(E)$ is not equal to $C_4$ or $K_n$ (indeed, this would imply a very small number of qubits or a trivial NNCP instance, respectively). Therefore, Theorem~\ref{Thm:AllNormal} applies in the sequel.

\subsection[Automorphism group of $X$]{Automorphism group of \boldmath $X$}
\label{Subsec:automorphismX}

Now that we established the full automorphism group of $\Cay(\mathbb{S}_n, T)$, we focus on the automorphism group of the entire graph $X$. Indeed, we need to take the arc structure in-between the subgraphs $H^k$ into account. We start by showing how these arcs restrict the automorphism group of a single subgraph, after which we combine these results to obtain $\Aut(X)$. 

\medskip

Each $H^k$ corresponds to a gate $g^k$ acting on two qubits in $Q$. The set of outgoing arcs $D^k$ consists of arcs leaving qubit orders $\tau$ where $\tau^{-1}(g^k) \in E$, see~\eqref{Def:Dk}. Since this arc structure needs to be preserved, the automorphisms of interest must setwise fix the qubit orders with this property. For all $k \in [m]$, let
\begin{align} \label{Def:Fk}
F^k := \{ \tau \in \mathbb{S}_n \, : \, \, \tau^{-1}(g^k) \in E \}. 
\end{align}
Instead of the automorphism group of $\Cay(\mathbb{S}_n, T)$, we are only interested in its subgroup that setwise fixes $F^k$. That is,
\begin{align*}
\Aut(\Cay(\mathbb{S}_n, T), F^k) := \left\{ \rho \in \Aut(\Cay(\mathbb{S}_n, T)) \, : \, \, \rho(F^k) = F^k \right\}. 
\end{align*}
For each $S \subseteq [n]$, let $\mathbb{S}_n(S) = \{ \tau \in \mathbb{S}_n \, : \, \, \tau(S) = S\}$, which is clearly a subgroup of $\mathbb{S}_n$. Now, if $\Coup(E) = K_n$, it follows that $F^k = \mathbb{S}_n$ and $\Aut(\Cay(\mathbb{S}_n, T), F^k) = \Aut(\Cay(\mathbb{S}_n,T))$. The following results establish a characterization of $\Aut(\Cay(\mathbb{S}_n, T), F^k)$ when $\Coup(E) \neq K_n$. 

\begin{theorem} \label{Thm:DirectSubgroupFk} 
Let $\Coup(E)$ be connected and not equal to $C_4$ or $K_n$. Then, $\Aut(\Cay(\mathbb{S}_n, T), F^k)$ is isomorphic to $\mathbb{S}_n(g^k) \times \Aut(\Coup(E))$.
\end{theorem}
\begin{proof}
Let $\theta$ be the group homomorphism defined in~\eqref{Eq:theta}. We now consider its restriction to the subgroup $\mathbb{S}_n(g^k) \times \Aut(\Coup(E))$, which we denote by $\theta_r$. Then its image $\theta_r(\mathbb{S}_n(g^k) \times \Aut(\Coup(E)))$ is clearly a subgroup of $\Aut(\Cay(\mathbb{S}_n, T))$. Since $\theta$ is injective by Theorem~\ref{Thm:DirectSubgroup}, so is $\theta_r$, and thus~$\theta_r(\mathbb{S}_n(g^k) \times \Aut(\Coup(E)))$ is isomorphic to $\mathbb{S}_n(g^k) \times \Aut(\Coup(E))$. 

We now prove that the set $\theta_r(\mathbb{S}_n(g^k) \times \Aut(\Coup(E)))$ is a subgroup of $\Aut(\Cay(\mathbb{S}_n, T), F^k)$. Let~${a \in \mathbb{S}_n(g^k)}$ and $b \in \Aut(\Coup(E))$. Then $\theta_r(a, b)$ is the mapping $\tau \mapsto a \tau b^{-1}$. Now, let~$\tau \in F^k$, i.e.,~$\tau^{-1}(g^k) \in E$. Using the fact that $a(g^k) = g^k$ and $b$ maps pairs in~$E$ to pairs in $E$, we obtain
\begin{align*}
(a\tau b^{-1})^{-1}(g^k) = (b \tau^{-1} a^{-1})(g^k) \in E,
\end{align*}
which implies $a \tau b^{-1} \in F^k$. So, $\theta_r(a,b) \in \Aut(\Cay(\mathbb{S}_n, T), F^k)$, from where it follows that~$\theta_r(\mathbb{S}_n(g^k) \times \Aut(\Coup(E)))$ is a subgroup of $\Aut(\Cay(\mathbb{S}_n, T), F^k)$. 

Next, we show that 
it is actually the full automorphism group. It suffices to show that any element in $\Aut(\Cay(\mathbb{S}_n, T), F^k)$ is of the form~$\theta_r(a,b)$ for some $a \in \mathbb{S}_n(g^k)$ and $b \in \Aut(\Coup(E))$. Let $\rho \in \Aut(\Cay(\mathbb{S}_n, T), F^k)$. By Theorem~\ref{Thm:DirectSubgroup}, we know that $\rho: \tau \mapsto a \tau b^{-1}$ for some $a \in \mathbb{S}_n, b \in \Aut(\Coup(E))$.  Suppose~$a \notin \mathbb{S}_n(g^k)$. Let $g^k$ be the pair $\{{\color{black}q_1,q_2}\}$. Then there exist~$k_1, k_2$ such that $a(k_1) = {\color{black}q_1}$ and~$a(k_2) = {\color{black}q_2}$, with~${\{k_1, k_2\} \neq \{{\color{black}q_1,q_2}\}}$. 
Now, we select two pairs of vertices $e \in E$ and $f \notin E$ as follows. If~${|\{k_1, k_2, {\color{black}q_1, q_2}\}| = 3}$, take~$e$ and~$f$ such that they share one vertex, otherwise take~$e$ and~$f$ disjoint. The only cases in which such selection is not possible, is when the subgraph induced by any three distinct vertices is a clique or for each edge in $E$ the graph resulting from deleting the edge is a clique. The only connected coupling graphs that satisfy either of these properties are $C_4$ and $K_n$. However, these graph structures are forbidden by the theorem statement.

Now, take any $\hat{\tau} \in \mathbb{S}_n$ such that
\begin{align*}
\hat{\tau}(e) = \{{\color{black}q_1, q_2}\} \quad \text{and} \quad  \hat{\tau}(f) = \{k_1, k_2\}. 
\end{align*}
As $\hat{\tau}^{-1}(\{{\color{black}q_1,q_2}\}) = e \in E$, it follows that $\hat{\tau} \in F^k$. However,
\begin{align*}
\rho(\hat{\tau})^{-1}(\{{\color{black}q_1,q_2}\}) = (a \hat{\tau} b^{-1})^{-1}(\{{\color{black}q_1,q_2}\}) = b \hat{\tau}^{-1} a^{-1} (\{{\color{black}q_1, q_2}\}) = b \hat{\tau}^{-1} (\{k_1, k_2\}) = b(f) \notin E,
\end{align*}
since $b$ maps non-edges to non-edges in $\Coup(E)$. We conclude that~$\rho(\hat{\tau}) \notin F^k$, which implies that~$\rho \notin \Aut(\Cay(\mathbb{S}_n, T), F^k)$. Since this is a contradiction, each automorphism in $\Aut(\Cay(\mathbb{S}_n, T), F^k)$ is in~$\theta_r(\mathbb{S}_n(g^k) \times \Aut(\Coup(E)))$. 
\end{proof}

Let $G^k_{\sub}$ denote $\Aut(\Cay(\mathbb{S}_n, T), F^k)$. If $X$ consists of only one subgraph, $X$ has vertex set ${\{s\} \cup V^{1} \cup \{t\}}$ and one can verify that in that case ${\color{black}\{}\id_{\{s\}}{\color{black}\}} \times G^1_{\sub} \times {\color{black}\{}\id_{\{t\}}{\color{black}\}}$ is $\Aut(X)$. Now, suppose $X$ has two subgraphs. Then, $H^1$ corresponds to gate~$g^1$ and $H^2$ corresponds to a possibly different gate $g^2$. In the sequel, we study how this affects the automorphism group of $X$. 

To that end, we need two intermediate results. For a set $S \subseteq [n]$, let $C(\mathbb{S}_n(S))$ denote the centralizer subgroup of $\mathbb{S}_n(S)$ which is defined as
\begin{align}
    C(\mathbb{S}_n(S)) = \left\{\tau \in \mathbb{S}_n \, : \, \, \tau \pi = \pi \tau \text{ for all } \pi \in \mathbb{S}_n(S) \right\}.
\end{align}
When $n \leq 2$, we know that $\mathbb{S}_n$ is abelian and thus $C(\mathbb{S}_n(S)) = \mathbb{S}_n$. Otherwise, we show that the centralizer subgroup is contained in $\mathbb{S}_n(S)$. 

\begin{lemma} \label{Lem:Centralizer} Let $n \geq 3$. Then, we have $C(\mathbb{S}_n(S)) \subseteq \mathbb{S}_n(S)$ for all $S \subseteq [n]$. 
\end{lemma}
\begin{proof}
Since $\mathbb{S}_n(S) = \mathbb{S}_n([n] \setminus S)$, we may assume that $|S| \geq 2$. Now, let $\tau \in C(\mathbb{S}_n(S))$ and assume for the sake of contradiction that $\tau \notin \mathbb{S}_n(S)$. Then there exist distinct $i, j \in S$ such that $\tau(i) \notin S$. Now, consider the transposition $(i~j)$. We have $(i~j) \tau(i) = \tau(i)$, while~${\tau (i~j) (i) = \tau(j)}$. Hence, $\tau$ and~$(i~j)$ do not commute, while $(i~j) \in \mathbb{S}_n(S)$. Therefore,~$\tau \notin C(\mathbb{S}_n(S))$, which is a contradiction.
\end{proof}
Exploiting Lemma~\ref{Lem:Centralizer}, we can show the following result for general sets $F$ of the form~\eqref{Def:Fk}.

\begin{theorem} \label{Thm:multiplesubgraphs}
Let $i, j \in [n]$, $n \geq 3$, and let $F = \{\tau \in \mathbb{S}_n \, : \, \, \{\tau^{-1}(i), \tau^{-1}(j)\} \in E\}$. Let~$a, b \in \mathbb{S}_n$ and suppose that $a \tau b^{-1} = \tau$ for all $\tau \in F$. Then $a = b = \id$. 
\end{theorem}
\begin{proof}
Observe that for all $\tau_1, \tau_2 \in F$ we have:
\begin{align*}
    \tau_1 b \tau_1^{-1} = a = \tau_2 b \tau_2^{-1}.
\end{align*}
Now, let us fix an edge $e \in E$. We can write any element $\pi \in \mathbb{S}_n(e)$ in the form $\pi = \tau^{-1}\tau'$ for some~$\tau, \tau' \in F$. To verify this, observe that since $e \in E$ there exist elements in $F$ that map $e$ to~$\{i, j\}$. By combining two such elements $\tau$ and $\tau'$, the composition $\tau^{-1}\tau'$ always maps $e$ back to~$e$. On the complement $[n] \setminus e$ we find all possible permutations in $F$, so we can always find 
$\tau, \tau' \in F$ such that $\tau^{-1}\tau'$ acts like $\pi$ on the set $[n] \setminus e$. 

Let $\tau_1, \tau_2 \in F$ be such that $\pi = \tau_1^{-1}\tau_2$. Then we know $\tau_1 b \tau_1^{-1} = \tau_2 b \tau_2^{-1}$, which can be rewritten as~${\pi^{-1} b \pi = b}$. As $\pi \in \mathbb{S}_n(e)$ was chosen arbitrarily, it follows that $\pi^{-1}b \pi = b$ for all $\pi \in \mathbb{S}_n(e)$, and thus $b \in C(\mathbb{S}_n(e))$. We now apply Lemma~\ref{Lem:Centralizer} with $S = e$. Since $n \geq 3$, it follows that $b \in \mathbb{S}_n(e)$. 

By repeating this argument for all $e \in E$, it follows that $b \in \bigcap_{e \in E} \mathbb{S}_n(e)$. As $\Coup(E)$ is connected, we conclude that $b = \id$, from which it immediately follows that $a = \id$ as well. 
\end{proof}

Let $\rho \in \Aut(X)$ where $X$ consists of two subgraphs. The case $n \leq 2$ leads to a trivial NNCP instance. Therefore, we may assume that $n \geq 3$. Then the restriction of $\rho$ to $H^1$ is an element of~$G^1_\sub$. In particular, each $\tau^1 \in F^1$ is mapped to $\rho(\tau^1) { \color{black} \, \in F^1}$ (here the superscript $1$ is added to indicate that $\tau^1$ is a vertex of $H^1$). In order to maintain the arc structure of $D^1$, it follows that the restriction of $\rho$ to $H^2$ should not only be an element of $G^2_\sub$, it should also pointwise fix the elements $\rho(\tau^2)$ for all $\tau^2 \in F^1$. Applying the result of Theorem~\ref{Thm:multiplesubgraphs}, the restriction of $\rho$ to $H^2$ should be the same automorphism as the restriction to $H^1$. On top of that, this restriction must also be in $G^2_\sub$. Thus,~$\rho$ is of the form $(\id_{\{s\}}, \pi, \pi, \id_{\{t\}})$ with $\pi \in G^1_{\sub} \cap G^2_{\sub}$. Extending this argument to larger $k$, let us define the following groups: 

\begin{align}
G_\sub & := \bigcap_{k = 1}^m G^k_{\sub} \cong \bigcap_{k = 1}^m \mathbb{S}_n(g^k) \times \Aut(\Coup(E)),  \label{Def:G_sub} \\
G_X & := \left\{ (\id_{\{s\}}, \rho, \ldots, \rho,  \id_{\{t\}}) \in {\color{black}\{}\id_{\{s\}}{\color{black}\}} \times \prod_{k=1}^m \Aut(H^k) \times {\color{black}\{}\id_{\{t\}}{\color{black}\}} \, : \, \, \rho \in G_{\sub} \right\} \label{Def:G_X}. 
\end{align}
By construction, $G_X$ is the full automorphism group $\Aut(X)$ in case $\Coup(E)$ is connected and not equal to $C_4$ or $K_n$. Observe that if these conditions are not met, $G_X$ is still a subgroup of $\Aut(X)$.

To get rid of of the intersection in the definition of $G_\sub$, we exploit the notion of the gate graph~$(Q,U)$ of a quantum circuit $\Gamma$, see Definition~\ref{Def:GateGraph}.
If $g^{k_1}$ is in $C$ with $g^{k_1} = \{{\color{black}q_1,q_2}\}$, this implies that the set $\{{\color{black}q_1,q_2}\}$ must be setwise fixed by all permutations in the group $\mathbb{S}_n(g^{k_1})$. If also~$g^{k_2} \in C$ with~$g^{k_2} = \{{\color{black}q_2, q_3}\}$, there is no other option than fixing ${\color{black}q_1}, {\color{black}q_2}$ and ${\color{black}q_3}$ elementwise in the group intersection~$\mathbb{S}_n(g^{k_1}) \cap \mathbb{S}_n(g^{k_2})$. From this observation, we can partition all qubits in $Q$ based on whether they belong to a connected component of size one, two or at least three in the gate graph $(Q,U)$. This leads to the introduction of the fixing pattern of $\Gamma$. 

\begin{definition} \label{Def:fixingpattern}
Let $\Gamma = (Q,C)$ be a quantum circuit on $n$ qubits. We define the fixing pattern of~$\Gamma$ as the partition $\mathcal{F} := \{S_1, \ldots, S_l\}$ of $Q$ such that each $S_i$ is either:
\begin{itemize}
\item a single qubit contained in a connected component of the gate graph $(Q,U)$ of size at least 3; 
\item a pair of qubits $\{{\color{black}q_1,q_2}\}$ that forms a connected component in the gate graph $(Q,U)$; 
\item the set of all {\color{black}isolated vertices} in the gate graph $(Q,U)$, which we denote by the free set in $\mathcal{F}$.
\end{itemize} 
Moreover, we define $f$ as the size of the free set, $p$ as the number of pairs and~${c~(=~n~-2p -f)}$ to be the number of qubits in a connected component of size at least~3 in~$(Q,U)$.
\end{definition}

Observe that $\mathcal{F}$ can be easily constructed by a scan of the connected components of~$(Q,U)$. The extreme cases are $\mathcal{F} = \{Q\}$ if $\Gamma$ contains no gates, whereas $\mathcal{F} = \{ \{{\color{black}q_1}\}, \ldots, \{{\color{black}q_2}\}\}$ if~$(Q,U)$ is connected. The group $\cap_{k=1}^m\mathbb{S}_n(g^k)$ consists of all permutations that setwise fix the elements in $\mathcal{F}$. To simplify notation, we define
\begin{align*}
\mathbb{S}_n(\mathcal{F}) := \{ a \in \mathbb{S}_n \, : \, \, a(S_i) = S_i \text{ for all } i \in [l] \}. 
\end{align*}
We know that $G_\sub \cong \mathbb{S}_n(\mathcal{F}) \times \Aut(\Coup(E))$, which implies that
\begin{align} \label{Eq:GXreformulated}
G_X \cong \mathbb{S}_n(\mathcal{F}) \times \Aut(\Coup(E)).
\end{align}
It follows from a simple counting argument that $|\mathbb{S}_n(\mathcal{F})| = 2^p f!$.

\subsection[Orbit and orbital structure of group action on $X$]{Orbit and orbital structure of group action on \boldmath $X$} \label{Subsec:Orbit} 
The elements of $G_X$ act on the vertices and arcs of $X$. In this section we study this group action in terms of its induced orbit and orbital structure, which will become of key importance in the symmetry reduction explained in Section~\ref{Sec:SymmetryReduction}.

Each automorphism in $G_X$ maps the vertex set of $X$ to itself. Given a vertex $\tau \in V$, the orbit of $\tau$ is the set of vertices to which $\tau$ is mapped to by the elements in $G_X$, i.e., all vertices~$\rho(\tau)$ with~$\rho \in G_X$.
The set of orbits forms a partition of $V$, which is written as the quotient~$V/G_X$.

Similarly, $G_X$ acts on the arc set $A$ by $\rho((\tau_1, \tau_2))=(\rho(\tau_1),\rho(\tau_2))$ for all $\rho \in G_X$. We denote the set of orbitals by $A/G_X$. Note that arcs in the same orbital have their initial vertices in the same orbit.
It is therefore natural to first understand the orbit structure of the action of $G_X$ on $V$. 

\medskip

Let $\Orb(\tau)$ denote the orbit of vertex $\tau \in V$. It follows from the construction of $G_X$ that $\Orb(s) = \{s\}$ and $\Orb(t) = \{t\}$. Moreover, the subgraphs $H^k$, $k \in [m]$, are invariant under the action of $G_X$ on $X$. For that reason, we can restrict ourselves to identifying the orbits within each subgraph $H^k$ under the action of $G_{\sub}$. Since all subgraphs are identical, this provides the orbit description for the entire graph $G_X$. 

Similar as before, we use $\tau$ to denote a vertex, as each vertex represents a qubit order in~$\mathbb{S}_n$. For all $k \in [m]$ and all $\tau \in V^k$, we obtain 
\begin{align}
& \,\,\,\, \begin{aligned}\Orb(\tau) & = \left\{\rho(\tau) \, : \, \, \rho \in G_\sub \right\} = \left\{ a \tau b^{-1} \, : \, \, a \in \mathbb{S}_n(\mathcal{F}),~ b \in \Aut(\Coup(E)) \right\}. \end{aligned} \label{Def:Orb}
\intertext{We also define the stabilizer subgroup with respect to $\tau$ under the action of $G_\sub$ as}
& \begin{aligned}\Stab(\tau) & := \left \{ \rho \in G_\sub \, : \, \, \rho(\tau) = \tau \right\}  \\
& \,\, \cong \left \{ (a,b) \in \mathbb{S}_n(\mathcal{F}) \times \Aut(\Coup(E)) \, : \, \, a \tau b^{-1} = \tau \right\}. \end{aligned} \label{Def:Stab}
\end{align}
The condition given in~\eqref{Def:Stab} for $(a,b)$ to act as a stabilizer can be rewritten as~${a = \tau b \tau^{-1}}$.
Thus, a pair~${(a,b) \in \mathbb{S}_n(\mathcal{F}) \times \Aut(\Coup(E))}$ corresponds to an element in $\Stab(\tau)$ if and only if the permutation $\tau b \tau^{-1}$ is in $\mathbb{S}_n(\mathcal{F})$ and $a = \tau b \tau^{-1}$. This implies that for all $S_i \in \mathcal{F}$ we must have $\tau b \tau^{-1} (S_i) = S_i$, or equivalently, $b(\tau^{-1}(S_i)) = \tau^{-1}(S_i)$. Hence,~$b$ setwise fixes the {inverse fixing pattern} in $\mathcal{F}$ with respect to $\tau$. Let us define the subgroup~$B_\tau$ of $\Aut(\Coup(E))$ that consists of all such elements, i.e.,
\begin{align}
B_\tau := \left\{ b \in \Aut(\Coup(E)) \, : \, \, b \left( \tau^{-1}(S_i) \right) = \tau^{-1}(S_i) \quad \forall i \in [l] \right\}. 
\end{align} 
Since for each $b \in B_\tau$, there exists exactly one element $a \in \mathbb{S}_n(\mathcal{F})$ such that~\mbox{$a \tau b^{-1} = \tau$}, we know
\begin{align} \label{Eq:stab_tau_rewritten}
    \Stab(\tau) \cong \{ (a, b) \, : \, \, b \in B_\tau, \, a  = \tau b \tau^{-1}\},
\end{align}
in particular, we have $|\Stab(\tau)| = |B_\tau|$. 

As $B_\tau$ is a subgroup of $\Aut(\Coup(E))$, it acts on the edge set of $\Coup(E)$. The orbital of an edge $\{i,j\} \in E$ under this group action is the set of all edges $\{b(i), b(j)\}$ with $b \in B_\tau$. We denote by the quotient $E / B_\tau$ the set of orbitals under this group action. 

We can show that if $\tau_1$ and $\tau_2$ belong to the same orbit, then the subgroups~$B_{\tau_1}$ and~$B_{\tau_2}$ are conjugate subgroups. Moreover, the quotients of their actions on $E$ have the same cardinality.
\begin{lemma} \label{Lem:conjugatesubgroups}
Let $\tau_1$ and $\tau_2$ be two qubit orders with $\tau_2 = a \tau_1 b^{-1}$ for some $a \in \mathbb{S}_n(\mathcal{F})$ and~$b \in \Aut(\Coup(E))$. Then, 
\begin{enumerate}
\item[(i)] $B_{\tau_2} = b B_{\tau_1} b^{-1}$; 
\item[(ii)] there exists a bijection from $E / B_{\tau_1}$ to $E / B_{\tau_{2}}$ given by left multiplication with $b$.  
\end{enumerate}
\end{lemma}
\begin{proof} 
\begin{enumerate}
    \item[$(i)$] Exploiting the fact that $a^{-1}(S_i) = S_i$ for all $i \in [l]$, we obtain
\begin{align*}
    B_{\tau_2} & = \left\{ b_2 \in \Aut(\Coup(E)) \, : \, \, b_2 \left( \tau_2^{-1}(S_i) \right) = \tau_2^{-1}(S_i) \quad \forall i \in [l] \right\} \\
    & = \left\{ b_2 \in \Aut(\Coup(E)) \, : \, \, b_2 \left( (a \tau_1 b^{-1})^{-1}(S_i) \right) = (a \tau_1 b^{-1})^{-1}(S_i) \quad \forall i \in [l] \right\} \\
    & = \left\{ b_2 \in \Aut(\Coup(E)) \, : \, \, b_2 b \tau_1^{-1} a^{-1} (S_i)  = b \tau_1^{-1} a^{-1}(S_i) \quad \forall i \in [l] \right\} \\
    & = \left\{ b_2 \in \Aut(\Coup(E)) \, : \, \, b^{-1} b_2 b \left( \tau_1^{-1} (S_i) \right)  =  \tau_1^{-1} (S_i) \quad \forall i \in [l] \right\} \\
    & = \left\{b b_1 b^{-1} \in \Aut(\Coup(E)) \, : \, \, b_1 \left( \tau_1^{-1} (S_i) \right)  =  \tau_1^{-1} (S_i) \quad \forall i \in [l] \right\} \\
    & = b B_{\tau_1} b^{-1}. 
\end{align*}

 \item[$(ii)$] This fact follows directly from $(i)$, by observing that
 \begin{align*}
     b~\Orb_{B_{\tau_1}}(i) = \left\{ b b_1b^{-1} (b(i)) \, : \, \, b_1 \in B_{\tau_1} \right\} = \left\{ b_2(b(i)) \, : \, \, b_2 \in B_{\tau_2} \right\} = \Orb_{B_{\tau_2}}(b(i)). 
 \end{align*}
One easily verifies that left multiplication by $b$ gives a bijection. 

\end{enumerate}

\end{proof}
\noindent As a consequence of the well-known orbit-stabilizer theorem, we establish the following relation between $\Orb(\tau)$ and $\Stab(\tau)$: 
\begin{align} \label{Eq:OrbitStabTheorem}
|\Orb(\tau)| = \frac{|G_\sub |}{|\Stab(\tau)|} = \frac{2^p f! \cdot |\Aut(\Coup(E))|}{|B_\tau|} .
\end{align}
Of course,  $\Orb(\tau)$ does not depend on the particular choice of the representative $\tau$ in the orbit. 

\medskip

To increase our understanding of $\Orb(\tau)$, we rewrite~\eqref{Def:Orb} as follows: 
\begin{align} \label{Eq:OrbitCosets} 
    \Orb(\tau) = \mathbb{S}_n(\mathcal{F})\tau \Aut(\Coup(E)) = \bigcup_{\tilde{\tau} \in \mathbb{S}_n(\mathcal{F})\tau} \tilde{\tau} \Aut(\Coup(E)). 
\end{align}
In other words, if $\mathbb{S}_n(\mathcal{F})$ is trivial, then the orbit partition of $V^k$ is given by the left cosets of $\Aut(\Coup(E))$ in $G_\sub$. Otherwise, each orbit is the union of several left cosets of $\Aut(\Coup(E))$ in~$G_\sub$, where the union is determined by the elements in the right cosets of $\mathbb{S}_n(\mathcal{F})$ in $G_\sub$. 

Of particular importance in the symmetry reduction is the number of orbits in each subgraph. We let $V^k / G_X$ denote the set of orbits of vertices in $V^k$ under the action of $G_X$, although we formally refer to the action of $G_X$ restricted to $V^k$. We allow for this slight abuse of notation, in order to simplify the terminology in Section~\ref{Sec:SymmetryReduction}.

\begin{theorem} \label{Thm:QuotientVk} The number of orbits of $V^k$ under $G_X$ is 
$|V^k / G_X | = \mfrac{\sum_{\tau \in \mathbb{S}_n} |B_\tau|}{2^p f! \cdot |\Aut(\Coup(E))|}$. 
\end{theorem}
\begin{proof}
Let $(V^k)^\rho$ denote the set of vertices in $V^k$ that are (pointwise) fixed by $\rho \in G_X$. Then, Burnside's lemma implies that $|V^k / G_X | = \frac{\sum_{\rho \in G_X}|(V^k)^\rho |}{|G_X|}$.
The sum in the numerator counts for every group element the number of vertices that are fixed. Alternatively, we can also sum over all vertices and count the number of group elements that stabilize the vertex. This leads to
\begin{align*}
    |V^k / G_X | = \frac{\sum_{\tau \in \mathbb{S}_n} |\Stab(\tau)| }{|\mathbb{S}_n(\mathcal{F}) \times \Aut(\Coup(E))|} = \frac{\sum_{\tau \in \mathbb{S}_n} |B_\tau|}{2^p f! \cdot |\Aut(\Coup(E))|}.
\end{align*}
\end{proof}

We now shift our focus to the analysis of the orbital structure of the arcs of $X$ under the action of $G_X$. Recall that $A$ consists of two types of arcs: arcs within a subgraph (the sets~$A^k$,~$k \in [m]$) and the arcs between the subgraphs (the sets $D^k, k \in \{0\}\cup [m]$). Since the sets~${A^1, \ldots, A^k}$ are identical and each set is invariant under the group action $G_X$, we can restrict our focus to the action of $G_\sub$ on a single subgraph. The orbital of an arc~${(\tau, \tau \sigma) \in A^k}$ corresponding to transposition~$\sigma = (i~j) \in T$ is given by
\begin{align*}
\Orb((\tau, \tau\sigma)) & := \left\{ (\rho(\tau), \rho(\tau \sigma )) \, : \, \, \rho \in G_\sub \right\}  \\
& \,\,= \left\{ (a \tau b^{-1}, a \tau \sigma b^{-1} \, : \, \, a \in \mathbb{S}_n(\mathcal{F}),~ b \in \Aut(\Coup(E)) \right\} \\
& \,\,= \left\{ (a \tau b^{-1}, a \tau b^{-1} (b(i)~b(j)) \, : \, \, a \in \mathbb{S}_n(\mathcal{F}),~ b \in \Aut(\Coup(E)) \right\},
\end{align*}
where the last line follows from the fact that $b(i~j)b^{-1} = (b(i)~b(j))$. 
This expression of $\Orb((\tau, \tau\sigma))$ implies that all arcs within the same orbital start at vertices within the same orbit and end at vertices within the same orbit (where the start- and end-orbits can differ). Moreover, the transpositions to which the arcs in $\Orb((\tau, \tau\sigma))$ correspond are related via $\Aut(\Coup(E))$, as the following lemma illustrates.
\begin{lemma} \label{Lem:BijectionOrbital}
    Let $\tau \in \mathbb{S}_n$. There exists a bijection between the orbitals starting from $\Orb(\tau)$ and the orbitals in $E / B_{\tau}$.
\end{lemma}
\begin{proof}
    It suffices to consider the orbital partition of the arcs leaving $\tau$, i.e., $\delta^+(\tau, A^k)$. If $B_\tau$ is trivial, the stabilizer subgroup of $\tau$ in $G_\sub$ is trivial, implying that no two arcs in $\delta^+(\tau, A^k)$ belong to the same orbital. In that case, $E / B_\tau$ is just the partition of $E$ under the identity map. If $B_\tau$ is nontrivial and $b \in B_\tau$ maps the edge corresponding to $\sigma_1$ to a different edge corresponding to $\sigma_2$, then the distinct arcs $(\tau, \tau\sigma_1)$ and $(\tau,\tau \sigma_2)$ belong to the same orbital under $G_\sub$. If $b \in B_\tau$ maps the edge corresponding to $\sigma_1$ to itself, then the orbital containing $(\tau, \tau\sigma_1)$ has a smaller cardinality. These three cases are depicted in Figure~\ref{Fig:OverviewOrbital}. We conclude that the arcs $(\tau, \tau \sigma_1)$ and $(\tau, \tau \sigma_2)$ belong to the same orbital if and only if the edges corresponding to $\sigma_1$ and $\sigma_2$ in $\Coup(E)$ belong to the same orbital in $E / B_\tau$. 
\end{proof}

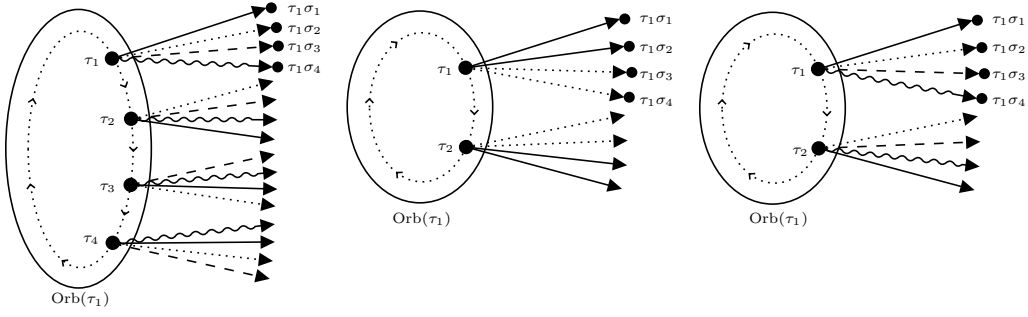
\begin{figure}
    
\centering

\tikzset{every picture/.style={line width=0.75pt}} 
\scalebox{0.8}{
\begin{tikzpicture}[x=0.75pt,y=0.75pt,yscale=-1,xscale=1]

\draw   (6.33,679.38) .. controls (6.33,631.01) and (26.66,591.81) .. (51.73,591.81) .. controls (76.81,591.81) and (97.13,631.01) .. (97.13,679.38) .. controls (97.13,727.74) and (76.81,766.94) .. (51.73,766.94) .. controls (26.66,766.94) and (6.33,727.74) .. (6.33,679.38) -- cycle ;
\draw  [fill={rgb, 255:red, 0; green, 0; blue, 0 }  ,fill opacity=1 ] (68.6,622.68) .. controls (68.6,620.47) and (70.39,618.68) .. (72.6,618.68) .. controls (74.81,618.68) and (76.6,620.47) .. (76.6,622.68) .. controls (76.6,624.88) and (74.81,626.68) .. (72.6,626.68) .. controls (70.39,626.68) and (68.6,624.88) .. (68.6,622.68) -- cycle ;
\draw  [fill={rgb, 255:red, 0; green, 0; blue, 0 }  ,fill opacity=1 ] (80.53,660.54) .. controls (80.53,658.33) and (82.32,656.54) .. (84.53,656.54) .. controls (86.74,656.54) and (88.53,658.33) .. (88.53,660.54) .. controls (88.53,662.75) and (86.74,664.54) .. (84.53,664.54) .. controls (82.32,664.54) and (80.53,662.75) .. (80.53,660.54) -- cycle ;
\draw  [fill={rgb, 255:red, 0; green, 0; blue, 0 }  ,fill opacity=1 ] (80.27,702.14) .. controls (80.27,699.93) and (82.06,698.14) .. (84.27,698.14) .. controls (86.48,698.14) and (88.27,699.93) .. (88.27,702.14) .. controls (88.27,704.35) and (86.48,706.14) .. (84.27,706.14) .. controls (82.06,706.14) and (80.27,704.35) .. (80.27,702.14) -- cycle ;
\draw  [fill={rgb, 255:red, 0; green, 0; blue, 0 }  ,fill opacity=1 ] (69.07,738.54) .. controls (69.07,736.33) and (70.86,734.54) .. (73.07,734.54) .. controls (75.28,734.54) and (77.07,736.33) .. (77.07,738.54) .. controls (77.07,740.75) and (75.28,742.54) .. (73.07,742.54) .. controls (70.86,742.54) and (69.07,740.75) .. (69.07,738.54) -- cycle ;
\draw  [dash pattern={on 0.84pt off 2.51pt}] (20.27,679.81) .. controls (20.27,638.76) and (34.92,605.48) .. (53,605.48) .. controls (71.08,605.48) and (85.73,638.76) .. (85.73,679.81) .. controls (85.73,720.86) and (71.08,754.14) .. (53,754.14) .. controls (34.92,754.14) and (20.27,720.86) .. (20.27,679.81) -- cycle ;
\draw   (19.55,704.76) -- (21.64,701.82) -- (24.5,704.02) ;
\draw   (82.23,637.15) -- (81.04,640.56) -- (77.69,639.23) ;
\draw   (88.43,678.34) -- (85.87,680.87) -- (83.43,678.21) ;
\draw   (83.3,718.87) -- (80.27,720.84) -- (78.42,717.75) ;
\draw   (42.24,753.42) -- (41.66,749.86) -- (45.24,749.42) ;
\draw  [dash pattern={on 0.84pt off 2.51pt}]  (73,622.68) -- (169.3,603.77) ;
\draw [shift={(172.24,603.2)}, rotate = 168.89] [fill={rgb, 255:red, 0; green, 0; blue, 0 }  ][line width=0.08]  [draw opacity=0] (8.93,-4.29) -- (0,0) -- (8.93,4.29) -- cycle    ;
\draw  [dash pattern={on 4.5pt off 4.5pt}]  (73,622.68) -- (170.45,615.03) ;
\draw [shift={(173.44,614.8)}, rotate = 175.51] [fill={rgb, 255:red, 0; green, 0; blue, 0 }  ][line width=0.08]  [draw opacity=0] (8.93,-4.29) -- (0,0) -- (8.93,4.29) -- cycle    ;
\draw    (73,622.68) .. controls (74.75,621.1) and (76.42,621.19) .. (77.99,622.94) .. controls (79.57,624.69) and (81.24,624.78) .. (82.99,623.21) .. controls (84.74,621.63) and (86.41,621.72) .. (87.98,623.47) .. controls (89.55,625.22) and (91.22,625.31) .. (92.97,623.74) .. controls (94.72,622.16) and (96.39,622.25) .. (97.96,624) .. controls (99.54,625.75) and (101.21,625.84) .. (102.96,624.27) .. controls (104.71,622.69) and (106.38,622.78) .. (107.95,624.53) .. controls (109.52,626.28) and (111.19,626.37) .. (112.94,624.8) .. controls (114.69,623.22) and (116.36,623.31) .. (117.94,625.06) .. controls (119.51,626.81) and (121.18,626.9) .. (122.93,625.33) .. controls (124.68,623.76) and (126.35,623.85) .. (127.92,625.6) .. controls (129.5,627.35) and (131.17,627.44) .. (132.92,625.86) .. controls (134.67,624.29) and (136.34,624.38) .. (137.91,626.13) .. controls (139.48,627.88) and (141.15,627.97) .. (142.9,626.39) .. controls (144.65,624.82) and (146.32,624.91) .. (147.89,626.66) .. controls (149.47,628.41) and (151.14,628.5) .. (152.89,626.92) .. controls (154.64,625.35) and (156.31,625.44) .. (157.88,627.19) -- (162.06,627.41) -- (170.04,627.84) ;
\draw [shift={(173.04,628)}, rotate = 183.04] [fill={rgb, 255:red, 0; green, 0; blue, 0 }  ][line width=0.08]  [draw opacity=0] (8.93,-4.29) -- (0,0) -- (8.93,4.29) -- cycle    ;
\draw    (73,622.68) -- (166.19,591.74) ;
\draw [shift={(169.04,590.8)}, rotate = 161.64] [fill={rgb, 255:red, 0; green, 0; blue, 0 }  ][line width=0.08]  [draw opacity=0] (8.93,-4.29) -- (0,0) -- (8.93,4.29) -- cycle    ;
\draw  [dash pattern={on 4.5pt off 4.5pt}]  (84.93,660.88) -- (172.47,648.8) ;
\draw [shift={(175.44,648.4)}, rotate = 172.15] [fill={rgb, 255:red, 0; green, 0; blue, 0 }  ][line width=0.08]  [draw opacity=0] (8.93,-4.29) -- (0,0) -- (8.93,4.29) -- cycle    ;
\draw    (84.93,660.88) .. controls (86.6,659.21) and (88.27,659.22) .. (89.93,660.89) .. controls (91.59,662.56) and (93.26,662.57) .. (94.93,660.91) .. controls (96.6,659.25) and (98.27,659.26) .. (99.93,660.93) .. controls (101.59,662.6) and (103.26,662.61) .. (104.93,660.95) .. controls (106.6,659.29) and (108.26,659.29) .. (109.93,660.96) .. controls (111.59,662.63) and (113.26,662.64) .. (114.93,660.98) .. controls (116.6,659.32) and (118.27,659.33) .. (119.93,661) .. controls (121.59,662.67) and (123.26,662.68) .. (124.93,661.02) .. controls (126.6,659.36) and (128.26,659.36) .. (129.93,661.03) .. controls (131.59,662.7) and (133.26,662.71) .. (134.93,661.05) .. controls (136.6,659.39) and (138.27,659.4) .. (139.93,661.07) .. controls (141.59,662.74) and (143.26,662.75) .. (144.93,661.09) .. controls (146.6,659.43) and (148.26,659.43) .. (149.93,661.1) .. controls (151.59,662.77) and (153.26,662.78) .. (154.93,661.12) .. controls (156.6,659.46) and (158.27,659.47) .. (159.93,661.14) -- (164.84,661.16) -- (172.84,661.18) ;
\draw [shift={(175.84,661.2)}, rotate = 180.2] [fill={rgb, 255:red, 0; green, 0; blue, 0 }  ][line width=0.08]  [draw opacity=0] (8.93,-4.29) -- (0,0) -- (8.93,4.29) -- cycle    ;
\draw    (84.93,660.88) -- (171.27,673.17) ;
\draw [shift={(174.24,673.6)}, rotate = 188.11] [fill={rgb, 255:red, 0; green, 0; blue, 0 }  ][line width=0.08]  [draw opacity=0] (8.93,-4.29) -- (0,0) -- (8.93,4.29) -- cycle    ;
\draw  [dash pattern={on 0.84pt off 2.51pt}]  (84.93,660.88) -- (170.55,637.58) ;
\draw [shift={(173.44,636.8)}, rotate = 164.78] [fill={rgb, 255:red, 0; green, 0; blue, 0 }  ][line width=0.08]  [draw opacity=0] (8.93,-4.29) -- (0,0) -- (8.93,4.29) -- cycle    ;
\draw  [dash pattern={on 0.84pt off 2.51pt}]  (84.27,702.14) -- (171.67,715.15) ;
\draw [shift={(174.64,715.6)}, rotate = 188.47] [fill={rgb, 255:red, 0; green, 0; blue, 0 }  ][line width=0.08]  [draw opacity=0] (8.93,-4.29) -- (0,0) -- (8.93,4.29) -- cycle    ;
\draw  [dash pattern={on 4.5pt off 4.5pt}]  (83.87,702.14) -- (171.31,683.42) ;
\draw [shift={(174.24,682.8)}, rotate = 167.92] [fill={rgb, 255:red, 0; green, 0; blue, 0 }  ][line width=0.08]  [draw opacity=0] (8.93,-4.29) -- (0,0) -- (8.93,4.29) -- cycle    ;
\draw    (83.87,702.14) .. controls (85.38,700.33) and (87.04,700.19) .. (88.85,701.7) .. controls (90.66,703.21) and (92.32,703.07) .. (93.83,701.26) .. controls (95.34,699.45) and (97,699.31) .. (98.81,700.82) .. controls (100.62,702.33) and (102.28,702.19) .. (103.79,700.38) .. controls (105.3,698.57) and (106.96,698.43) .. (108.77,699.94) .. controls (110.58,701.45) and (112.24,701.3) .. (113.75,699.49) .. controls (115.26,697.68) and (116.92,697.54) .. (118.73,699.05) .. controls (120.54,700.56) and (122.2,700.42) .. (123.71,698.61) .. controls (125.22,696.8) and (126.88,696.66) .. (128.69,698.17) .. controls (130.5,699.68) and (132.16,699.54) .. (133.67,697.73) .. controls (135.18,695.92) and (136.84,695.78) .. (138.65,697.29) .. controls (140.46,698.8) and (142.12,698.66) .. (143.63,696.85) .. controls (145.14,695.04) and (146.8,694.9) .. (148.61,696.41) .. controls (150.42,697.92) and (152.08,697.78) .. (153.59,695.97) .. controls (155.1,694.16) and (156.76,694.01) .. (158.57,695.52) .. controls (160.38,697.03) and (162.04,696.89) .. (163.55,695.08) -- (164.88,694.97) -- (172.85,694.26) ;
\draw [shift={(175.84,694)}, rotate = 174.94] [fill={rgb, 255:red, 0; green, 0; blue, 0 }  ][line width=0.08]  [draw opacity=0] (8.93,-4.29) -- (0,0) -- (8.93,4.29) -- cycle    ;
\draw    (83.87,702.14) -- (172.84,704.71) ;
\draw [shift={(175.84,704.8)}, rotate = 181.65] [fill={rgb, 255:red, 0; green, 0; blue, 0 }  ][line width=0.08]  [draw opacity=0] (8.93,-4.29) -- (0,0) -- (8.93,4.29) -- cycle    ;
\draw    (72.53,738.88) .. controls (73.99,737.03) and (75.65,736.83) .. (77.5,738.29) .. controls (79.35,739.75) and (81.01,739.55) .. (82.46,737.7) .. controls (83.92,735.85) and (85.58,735.65) .. (87.43,737.11) .. controls (89.28,738.57) and (90.94,738.37) .. (92.39,736.52) .. controls (93.85,734.67) and (95.51,734.47) .. (97.36,735.93) .. controls (99.21,737.39) and (100.87,737.19) .. (102.32,735.34) .. controls (103.78,733.49) and (105.44,733.29) .. (107.29,734.75) .. controls (109.14,736.21) and (110.8,736.01) .. (112.25,734.16) .. controls (113.71,732.31) and (115.37,732.11) .. (117.22,733.57) .. controls (119.07,735.03) and (120.73,734.83) .. (122.18,732.98) .. controls (123.64,731.13) and (125.3,730.93) .. (127.15,732.39) .. controls (129,733.85) and (130.66,733.65) .. (132.11,731.8) .. controls (133.57,729.95) and (135.23,729.75) .. (137.08,731.21) .. controls (138.93,732.67) and (140.59,732.47) .. (142.04,730.62) .. controls (143.5,728.77) and (145.16,728.57) .. (147.01,730.03) .. controls (148.86,731.49) and (150.52,731.29) .. (151.97,729.44) .. controls (153.43,727.59) and (155.09,727.39) .. (156.94,728.85) .. controls (158.79,730.31) and (160.45,730.11) .. (161.91,728.26) -- (163.32,728.09) -- (171.26,727.15) ;
\draw [shift={(174.24,726.8)}, rotate = 173.23] [fill={rgb, 255:red, 0; green, 0; blue, 0 }  ][line width=0.08]  [draw opacity=0] (8.93,-4.29) -- (0,0) -- (8.93,4.29) -- cycle    ;
\draw    (72.53,738.88) -- (171.24,738.02) ;
\draw [shift={(174.24,738)}, rotate = 179.5] [fill={rgb, 255:red, 0; green, 0; blue, 0 }  ][line width=0.08]  [draw opacity=0] (8.93,-4.29) -- (0,0) -- (8.93,4.29) -- cycle    ;
\draw  [dash pattern={on 0.84pt off 2.51pt}]  (72.53,738.88) -- (170.06,749.67) ;
\draw [shift={(173.04,750)}, rotate = 186.31] [fill={rgb, 255:red, 0; green, 0; blue, 0 }  ][line width=0.08]  [draw opacity=0] (8.93,-4.29) -- (0,0) -- (8.93,4.29) -- cycle    ;
\draw  [dash pattern={on 4.5pt off 4.5pt}]  (72.53,738.88) -- (168.11,760.53) ;
\draw [shift={(171.04,761.2)}, rotate = 192.77] [fill={rgb, 255:red, 0; green, 0; blue, 0 }  ][line width=0.08]  [draw opacity=0] (8.93,-4.29) -- (0,0) -- (8.93,4.29) -- cycle    ;
\draw   (20.3,650.19) -- (23.34,648.24) -- (25.16,651.36) ;
\draw  [fill={rgb, 255:red, 0; green, 0; blue, 0 }  ,fill opacity=1 ] (169.04,590.8) .. controls (169.04,589.17) and (170.36,587.86) .. (171.98,587.86) .. controls (173.6,587.86) and (174.92,589.17) .. (174.92,590.8) .. controls (174.92,592.42) and (173.6,593.74) .. (171.98,593.74) .. controls (170.36,593.74) and (169.04,592.42) .. (169.04,590.8) -- cycle ;
\draw   (219.33,653.1) .. controls (219.33,619.8) and (239.66,592.81) .. (264.73,592.81) .. controls (289.81,592.81) and (310.13,619.8) .. (310.13,653.1) .. controls (310.13,686.4) and (289.81,713.4) .. (264.73,713.4) .. controls (239.66,713.4) and (219.33,686.4) .. (219.33,653.1) -- cycle ;
\draw  [fill={rgb, 255:red, 0; green, 0; blue, 0 }  ,fill opacity=1 ] (289.2,628.48) .. controls (289.2,626.27) and (290.99,624.48) .. (293.2,624.48) .. controls (295.41,624.48) and (297.2,626.27) .. (297.2,628.48) .. controls (297.2,630.68) and (295.41,632.48) .. (293.2,632.48) .. controls (290.99,632.48) and (289.2,630.68) .. (289.2,628.48) -- cycle ;
\draw  [fill={rgb, 255:red, 0; green, 0; blue, 0 }  ,fill opacity=1 ] (289.53,678.34) .. controls (289.53,676.13) and (291.32,674.34) .. (293.53,674.34) .. controls (295.74,674.34) and (297.53,676.13) .. (297.53,678.34) .. controls (297.53,680.55) and (295.74,682.34) .. (293.53,682.34) .. controls (291.32,682.34) and (289.53,680.55) .. (289.53,678.34) -- cycle ;
\draw  [dash pattern={on 0.84pt off 2.51pt}] (233.27,653.34) .. controls (233.27,627.45) and (247.92,606.48) .. (266,606.48) .. controls (284.08,606.48) and (298.73,627.45) .. (298.73,653.34) .. controls (298.73,679.22) and (284.08,700.2) .. (266,700.2) .. controls (247.92,700.2) and (233.27,679.22) .. (233.27,653.34) -- cycle ;
\draw   (301.18,655.04) -- (298.72,657.67) -- (296.18,655.11) ;
\draw   (246.62,612.18) -- (250.2,612.59) -- (249.65,616.15) ;
\draw   (250.62,697.98) -- (250.32,694.38) -- (253.93,694.23) ;
\draw    (293.6,628.48) -- (389.43,615.4) ;
\draw [shift={(392.4,615)}, rotate = 172.23] [fill={rgb, 255:red, 0; green, 0; blue, 0 }  ][line width=0.08]  [draw opacity=0] (8.93,-4.29) -- (0,0) -- (8.93,4.29) -- cycle    ;
\draw  [dash pattern={on 0.84pt off 2.51pt}]  (293.6,628.48) -- (391,631.31) ;
\draw [shift={(394,631.4)}, rotate = 181.67] [fill={rgb, 255:red, 0; green, 0; blue, 0 }  ][line width=0.08]  [draw opacity=0] (8.93,-4.29) -- (0,0) -- (8.93,4.29) -- cycle    ;
\draw  [dash pattern={on 0.84pt off 2.51pt}]  (293.6,628.48) -- (389.45,646.44) ;
\draw [shift={(392.4,647)}, rotate = 190.62] [fill={rgb, 255:red, 0; green, 0; blue, 0 }  ][line width=0.08]  [draw opacity=0] (8.93,-4.29) -- (0,0) -- (8.93,4.29) -- cycle    ;
\draw    (293.6,628.48) -- (386.74,598.71) ;
\draw [shift={(389.6,597.8)}, rotate = 162.28] [fill={rgb, 255:red, 0; green, 0; blue, 0 }  ][line width=0.08]  [draw opacity=0] (8.93,-4.29) -- (0,0) -- (8.93,4.29) -- cycle    ;
\draw  [dash pattern={on 0.84pt off 2.51pt}]  (293.53,678.34) -- (391.4,674.32) ;
\draw [shift={(394.4,674.2)}, rotate = 177.65] [fill={rgb, 255:red, 0; green, 0; blue, 0 }  ][line width=0.08]  [draw opacity=0] (8.93,-4.29) -- (0,0) -- (8.93,4.29) -- cycle    ;
\draw    (293.53,678.34) -- (391.02,689.46) ;
\draw [shift={(394,689.8)}, rotate = 186.5] [fill={rgb, 255:red, 0; green, 0; blue, 0 }  ][line width=0.08]  [draw opacity=0] (8.93,-4.29) -- (0,0) -- (8.93,4.29) -- cycle    ;
\draw    (293.53,678.34) -- (387.9,703.81) ;
\draw [shift={(390.8,704.6)}, rotate = 195.1] [fill={rgb, 255:red, 0; green, 0; blue, 0 }  ][line width=0.08]  [draw opacity=0] (8.93,-4.29) -- (0,0) -- (8.93,4.29) -- cycle    ;
\draw  [dash pattern={on 0.84pt off 2.51pt}]  (293.53,678.34) -- (389.86,658.79) ;
\draw [shift={(392.8,658.2)}, rotate = 168.53] [fill={rgb, 255:red, 0; green, 0; blue, 0 }  ][line width=0.08]  [draw opacity=0] (8.93,-4.29) -- (0,0) -- (8.93,4.29) -- cycle    ;
\draw   (230.71,650.56) -- (233.26,648.01) -- (235.71,650.65) ;
\draw  [fill={rgb, 255:red, 0; green, 0; blue, 0 }  ,fill opacity=1 ] (172.24,603.2) .. controls (172.24,601.57) and (173.56,600.26) .. (175.18,600.26) .. controls (176.8,600.26) and (178.12,601.57) .. (178.12,603.2) .. controls (178.12,604.82) and (176.8,606.14) .. (175.18,606.14) .. controls (173.56,606.14) and (172.24,604.82) .. (172.24,603.2) -- cycle ;
\draw  [fill={rgb, 255:red, 0; green, 0; blue, 0 }  ,fill opacity=1 ] (173.44,614.8) .. controls (173.44,613.17) and (174.76,611.86) .. (176.38,611.86) .. controls (178,611.86) and (179.32,613.17) .. (179.32,614.8) .. controls (179.32,616.42) and (178,617.74) .. (176.38,617.74) .. controls (174.76,617.74) and (173.44,616.42) .. (173.44,614.8) -- cycle ;
\draw  [fill={rgb, 255:red, 0; green, 0; blue, 0 }  ,fill opacity=1 ] (173.04,628) .. controls (173.04,626.37) and (174.36,625.06) .. (175.98,625.06) .. controls (177.6,625.06) and (178.92,626.37) .. (178.92,628) .. controls (178.92,629.62) and (177.6,630.94) .. (175.98,630.94) .. controls (174.36,630.94) and (173.04,629.62) .. (173.04,628) -- cycle ;
\draw  [fill={rgb, 255:red, 0; green, 0; blue, 0 }  ,fill opacity=1 ] (389.6,597.8) .. controls (389.6,596.17) and (390.92,594.86) .. (392.54,594.86) .. controls (394.16,594.86) and (395.48,596.17) .. (395.48,597.8) .. controls (395.48,599.42) and (394.16,600.74) .. (392.54,600.74) .. controls (390.92,600.74) and (389.6,599.42) .. (389.6,597.8) -- cycle ;
\draw  [fill={rgb, 255:red, 0; green, 0; blue, 0 }  ,fill opacity=1 ] (392.4,615) .. controls (392.4,613.37) and (393.72,612.06) .. (395.34,612.06) .. controls (396.96,612.06) and (398.28,613.37) .. (398.28,615) .. controls (398.28,616.62) and (396.96,617.94) .. (395.34,617.94) .. controls (393.72,617.94) and (392.4,616.62) .. (392.4,615) -- cycle ;
\draw  [fill={rgb, 255:red, 0; green, 0; blue, 0 }  ,fill opacity=1 ] (394,631.4) .. controls (394,629.77) and (395.32,628.46) .. (396.94,628.46) .. controls (398.56,628.46) and (399.88,629.77) .. (399.88,631.4) .. controls (399.88,633.02) and (398.56,634.34) .. (396.94,634.34) .. controls (395.32,634.34) and (394,633.02) .. (394,631.4) -- cycle ;
\draw  [fill={rgb, 255:red, 0; green, 0; blue, 0 }  ,fill opacity=1 ] (392.4,647) .. controls (392.4,645.37) and (393.72,644.06) .. (395.34,644.06) .. controls (396.96,644.06) and (398.28,645.37) .. (398.28,647) .. controls (398.28,648.62) and (396.96,649.94) .. (395.34,649.94) .. controls (393.72,649.94) and (392.4,648.62) .. (392.4,647) -- cycle ;
\draw   (439.33,653.9) .. controls (439.33,620.6) and (459.66,593.61) .. (484.73,593.61) .. controls (509.81,593.61) and (530.13,620.6) .. (530.13,653.9) .. controls (530.13,687.2) and (509.81,714.2) .. (484.73,714.2) .. controls (459.66,714.2) and (439.33,687.2) .. (439.33,653.9) -- cycle ;
\draw  [fill={rgb, 255:red, 0; green, 0; blue, 0 }  ,fill opacity=1 ] (509.2,629.28) .. controls (509.2,627.07) and (510.99,625.28) .. (513.2,625.28) .. controls (515.41,625.28) and (517.2,627.07) .. (517.2,629.28) .. controls (517.2,631.48) and (515.41,633.28) .. (513.2,633.28) .. controls (510.99,633.28) and (509.2,631.48) .. (509.2,629.28) -- cycle ;
\draw  [fill={rgb, 255:red, 0; green, 0; blue, 0 }  ,fill opacity=1 ] (509.53,679.14) .. controls (509.53,676.93) and (511.32,675.14) .. (513.53,675.14) .. controls (515.74,675.14) and (517.53,676.93) .. (517.53,679.14) .. controls (517.53,681.35) and (515.74,683.14) .. (513.53,683.14) .. controls (511.32,683.14) and (509.53,681.35) .. (509.53,679.14) -- cycle ;
\draw  [dash pattern={on 0.84pt off 2.51pt}] (453.27,654.14) .. controls (453.27,628.25) and (467.92,607.28) .. (486,607.28) .. controls (504.08,607.28) and (518.73,628.25) .. (518.73,654.14) .. controls (518.73,680.02) and (504.08,701) .. (486,701) .. controls (467.92,701) and (453.27,680.02) .. (453.27,654.14) -- cycle ;
\draw   (521.18,655.84) -- (518.72,658.47) -- (516.18,655.91) ;
\draw   (466.62,612.98) -- (470.2,613.39) -- (469.65,616.95) ;
\draw   (470.62,698.78) -- (470.32,695.18) -- (473.93,695.03) ;
\draw  [dash pattern={on 0.84pt off 2.51pt}]  (513.6,629.28) -- (609.43,616.2) ;
\draw [shift={(612.4,615.8)}, rotate = 172.23] [fill={rgb, 255:red, 0; green, 0; blue, 0 }  ][line width=0.08]  [draw opacity=0] (8.93,-4.29) -- (0,0) -- (8.93,4.29) -- cycle    ;
\draw  [dash pattern={on 4.5pt off 4.5pt}]  (513.6,629.28) -- (611,632.11) ;
\draw [shift={(614,632.2)}, rotate = 181.67] [fill={rgb, 255:red, 0; green, 0; blue, 0 }  ][line width=0.08]  [draw opacity=0] (8.93,-4.29) -- (0,0) -- (8.93,4.29) -- cycle    ;
\draw    (513.6,629.28) .. controls (515.55,627.95) and (517.18,628.25) .. (518.51,630.2) .. controls (519.84,632.15) and (521.48,632.45) .. (523.43,631.12) .. controls (525.38,629.79) and (527.01,630.09) .. (528.34,632.04) .. controls (529.67,633.99) and (531.31,634.29) .. (533.26,632.96) .. controls (535.21,631.63) and (536.84,631.93) .. (538.17,633.88) .. controls (539.5,635.83) and (541.14,636.13) .. (543.09,634.8) .. controls (545.04,633.47) and (546.67,633.77) .. (548,635.72) .. controls (549.33,637.67) and (550.97,637.97) .. (552.92,636.64) .. controls (554.87,635.31) and (556.5,635.62) .. (557.83,637.57) .. controls (559.16,639.52) and (560.79,639.82) .. (562.74,638.49) .. controls (564.69,637.16) and (566.33,637.46) .. (567.66,639.41) .. controls (568.99,641.36) and (570.62,641.66) .. (572.57,640.33) .. controls (574.52,639) and (576.16,639.3) .. (577.49,641.25) .. controls (578.82,643.2) and (580.45,643.5) .. (582.4,642.17) .. controls (584.35,640.84) and (585.99,641.14) .. (587.32,643.09) .. controls (588.65,645.04) and (590.28,645.34) .. (592.23,644.01) .. controls (594.18,642.68) and (595.81,642.99) .. (597.14,644.94) -- (601.59,645.77) -- (609.45,647.24) ;
\draw [shift={(612.4,647.8)}, rotate = 190.62] [fill={rgb, 255:red, 0; green, 0; blue, 0 }  ][line width=0.08]  [draw opacity=0] (8.93,-4.29) -- (0,0) -- (8.93,4.29) -- cycle    ;
\draw    (513.6,629.28) -- (606.74,599.51) ;
\draw [shift={(609.6,598.6)}, rotate = 162.28] [fill={rgb, 255:red, 0; green, 0; blue, 0 }  ][line width=0.08]  [draw opacity=0] (8.93,-4.29) -- (0,0) -- (8.93,4.29) -- cycle    ;
\draw  [dash pattern={on 4.5pt off 4.5pt}]  (513.53,679.14) -- (611.4,675.12) ;
\draw [shift={(614.4,675)}, rotate = 177.65] [fill={rgb, 255:red, 0; green, 0; blue, 0 }  ][line width=0.08]  [draw opacity=0] (8.93,-4.29) -- (0,0) -- (8.93,4.29) -- cycle    ;
\draw    (513.53,679.14) .. controls (515.38,677.67) and (517.03,677.86) .. (518.5,679.71) .. controls (519.97,681.55) and (521.63,681.74) .. (523.47,680.27) .. controls (525.32,678.8) and (526.97,678.99) .. (528.44,680.84) .. controls (529.9,682.69) and (531.55,682.88) .. (533.4,681.41) .. controls (535.24,679.94) and (536.9,680.13) .. (538.37,681.97) .. controls (539.84,683.82) and (541.49,684.01) .. (543.34,682.54) .. controls (545.19,681.07) and (546.84,681.26) .. (548.31,683.11) .. controls (549.78,684.95) and (551.44,685.14) .. (553.28,683.67) .. controls (555.13,682.2) and (556.78,682.39) .. (558.24,684.24) .. controls (559.71,686.09) and (561.36,686.28) .. (563.21,684.81) .. controls (565.05,683.34) and (566.71,683.53) .. (568.18,685.37) .. controls (569.65,687.22) and (571.3,687.41) .. (573.15,685.94) .. controls (574.99,684.47) and (576.65,684.66) .. (578.12,686.5) .. controls (579.58,688.35) and (581.23,688.54) .. (583.08,687.07) .. controls (584.93,685.6) and (586.58,685.79) .. (588.05,687.64) .. controls (589.52,689.48) and (591.18,689.67) .. (593.02,688.2) .. controls (594.87,686.73) and (596.52,686.92) .. (597.99,688.77) .. controls (599.45,690.62) and (601.1,690.81) .. (602.95,689.34) -- (603.07,689.35) -- (611.02,690.26) ;
\draw [shift={(614,690.6)}, rotate = 186.5] [fill={rgb, 255:red, 0; green, 0; blue, 0 }  ][line width=0.08]  [draw opacity=0] (8.93,-4.29) -- (0,0) -- (8.93,4.29) -- cycle    ;
\draw    (513.53,679.14) -- (607.9,704.61) ;
\draw [shift={(610.8,705.4)}, rotate = 195.1] [fill={rgb, 255:red, 0; green, 0; blue, 0 }  ][line width=0.08]  [draw opacity=0] (8.93,-4.29) -- (0,0) -- (8.93,4.29) -- cycle    ;
\draw  [dash pattern={on 0.84pt off 2.51pt}]  (513.53,679.14) -- (609.86,659.59) ;
\draw [shift={(612.8,659)}, rotate = 168.53] [fill={rgb, 255:red, 0; green, 0; blue, 0 }  ][line width=0.08]  [draw opacity=0] (8.93,-4.29) -- (0,0) -- (8.93,4.29) -- cycle    ;
\draw   (450.71,651.36) -- (453.26,648.81) -- (455.71,651.45) ;
\draw  [fill={rgb, 255:red, 0; green, 0; blue, 0 }  ,fill opacity=1 ] (609.6,598.6) .. controls (609.6,596.97) and (610.92,595.66) .. (612.54,595.66) .. controls (614.16,595.66) and (615.48,596.97) .. (615.48,598.6) .. controls (615.48,600.22) and (614.16,601.54) .. (612.54,601.54) .. controls (610.92,601.54) and (609.6,600.22) .. (609.6,598.6) -- cycle ;
\draw  [fill={rgb, 255:red, 0; green, 0; blue, 0 }  ,fill opacity=1 ] (612.4,615.8) .. controls (612.4,614.17) and (613.72,612.86) .. (615.34,612.86) .. controls (616.96,612.86) and (618.28,614.17) .. (618.28,615.8) .. controls (618.28,617.42) and (616.96,618.74) .. (615.34,618.74) .. controls (613.72,618.74) and (612.4,617.42) .. (612.4,615.8) -- cycle ;
\draw  [fill={rgb, 255:red, 0; green, 0; blue, 0 }  ,fill opacity=1 ] (614,632.2) .. controls (614,630.57) and (615.32,629.26) .. (616.94,629.26) .. controls (618.56,629.26) and (619.88,630.57) .. (619.88,632.2) .. controls (619.88,633.82) and (618.56,635.14) .. (616.94,635.14) .. controls (615.32,635.14) and (614,633.82) .. (614,632.2) -- cycle ;
\draw  [fill={rgb, 255:red, 0; green, 0; blue, 0 }  ,fill opacity=1 ] (612.4,647.8) .. controls (612.4,646.17) and (613.72,644.86) .. (615.34,644.86) .. controls (616.96,644.86) and (618.28,646.17) .. (618.28,647.8) .. controls (618.28,649.42) and (616.96,650.74) .. (615.34,650.74) .. controls (613.72,650.74) and (612.4,649.42) .. (612.4,647.8) -- cycle ;

\draw (53.8,620.6) node [anchor=north west][inner sep=0.75pt]  [font=\scriptsize] [align=left] {$\displaystyle \tau _{1}$};
\draw (63.4,657.4) node [anchor=north west][inner sep=0.75pt]  [font=\scriptsize] [align=left] {$\displaystyle \tau _{2}$};
\draw (63.4,700) node [anchor=north west][inner sep=0.75pt]  [font=\scriptsize] [align=left] {$\displaystyle \tau _{3}$};
\draw (53,732.6) node [anchor=north west][inner sep=0.75pt]  [font=\scriptsize] [align=left] {$\displaystyle \tau _{4}$};
\draw (180.6,588.2) node [anchor=north west][inner sep=0.75pt]  [font=\scriptsize] [align=left] {$\displaystyle \tau _{1} \sigma _{1}$};
\draw (180.6,600.2) node [anchor=north west][inner sep=0.75pt]  [font=\scriptsize] [align=left] {$\displaystyle \tau _{1} \sigma _{2}$};
\draw (180.6,611.8) node [anchor=north west][inner sep=0.75pt]  [font=\scriptsize] [align=left] {$\displaystyle \tau _{1} \sigma _{3}$};
\draw (180.6,625) node [anchor=north west][inner sep=0.75pt]  [font=\scriptsize] [align=left] {$\displaystyle \tau _{1} \sigma _{4}$};
\draw (274,626) node [anchor=north west][inner sep=0.75pt]  [font=\scriptsize] [align=left] {$\displaystyle \tau _{1}$};
\draw (274.4,674.4) node [anchor=north west][inner sep=0.75pt]  [font=\scriptsize] [align=left] {$\displaystyle \tau _{2}$};
\draw (400.2,594) node [anchor=north west][inner sep=0.75pt]  [font=\scriptsize] [align=left] {$\displaystyle \tau _{1} \sigma _{1}$};
\draw (400.2,612.6) node [anchor=north west][inner sep=0.75pt]  [font=\scriptsize] [align=left] {$\displaystyle \tau _{1} \sigma _{2}$};
\draw (400.2,629.4) node [anchor=north west][inner sep=0.75pt]  [font=\scriptsize] [align=left] {$\displaystyle \tau _{1} \sigma _{3}$};
\draw (400.2,645.8) node [anchor=north west][inner sep=0.75pt]  [font=\scriptsize] [align=left] {$\displaystyle \tau _{1} \sigma _{4}$};
\draw (494.8,627) node [anchor=north west][inner sep=0.75pt]  [font=\scriptsize] [align=left] {$\displaystyle \tau _{1}$};
\draw (495.6,677.2) node [anchor=north west][inner sep=0.75pt]  [font=\scriptsize] [align=left] {$\displaystyle \tau _{2}$};
\draw (620.2,594.8) node [anchor=north west][inner sep=0.75pt]  [font=\scriptsize] [align=left] {$\displaystyle \tau _{1} \sigma _{1}$};
\draw (620.2,612.4) node [anchor=north west][inner sep=0.75pt]  [font=\scriptsize] [align=left] {$\displaystyle \tau _{1} \sigma _{2}$};
\draw (620.2,629.2) node [anchor=north west][inner sep=0.75pt]  [font=\scriptsize] [align=left] {$\displaystyle \tau _{1} \sigma _{3}$};
\draw (620.2,645.6) node [anchor=north west][inner sep=0.75pt]  [font=\scriptsize] [align=left] {$\displaystyle \tau _{1} \sigma _{4}$};
\draw (468.68,717.8) node [anchor=north west][inner sep=0.75pt]  [font=\scriptsize] [align=left] {Orb$\displaystyle ( \tau _{1})$};
\draw (246.08,717.8) node [anchor=north west][inner sep=0.75pt]  [font=\scriptsize] [align=left] {Orb$\displaystyle ( \tau _{1})$};
\draw (33.04,768.2) node [anchor=north west][inner sep=0.75pt]  [font=\scriptsize] [align=left] {$\displaystyle \text{Orb}( \tau _{1})$};

\end{tikzpicture}}
\caption{Graphical overview of orbital structure within a subgraph $H^k$. Each line type (solid, dotted, dashed and curled) corresponds to another orbital. Case I (left): $B_{\tau_1}$ is trivial. Case II (middle): $B_{\tau_1}$ is nontrivial and the orbital of $\sigma_1$ under $B_{\tau_1}$ contains $\sigma_2$. Case III (right): $B_{\tau_1}$ is nontrivial, but the orbital of $\sigma_1$ under $B_{\tau_1}$ only consists of $\sigma_1$. \label{Fig:OverviewOrbital}}
\end{figure}

\medskip

The following result regards the cardinality of the set of orbitals of $A^k$ under the action of $G_X$ restricted to $A^k$. By slight abuse of notation, we again denote this set by the quotient $A^k / G_X$. 
\begin{theorem} \label{Thm:QuotientAk}
The number of orbitals of $A^k$ under $G_X$ is 
$|A^k / G_X | = \mfrac{\sum_{\tau \in \mathbb{S}_n} |B_\tau| \cdot |E / B_\tau|}{2^p f! \cdot |\Aut(\Coup(E))|}$. 
\end{theorem}
\begin{proof}
Since the arcs belonging to an orbital all start from vertices in the same orbit, it suffices to enumerate over all orbits and count the number of orbitals starting from that orbit. It follows from Lemma~\ref{Lem:BijectionOrbital} that the number of distinct orbitals starting from $\Orb(\tau)$ is $|E / B_\tau|$, where the choice of $\tau$ to represent $\Orb(\tau)$ does not affect this quantity, see Lemma~\ref{Lem:conjugatesubgroups}.
We now sum over all $\Orb(\tau) \in V^k / G_X$: 
\begin{align*}
    |A^k / G_X| & = \sum_{\Orb(\tau) \in V^k / G_X} |E/ B_\tau | \\
    & = \sum_{\Orb(\tau) \in V^k / G_X} \frac{2^p f! \cdot |\Aut(\Coup(E))|}{|B_\tau|} \cdot  \frac{|B_\tau| \cdot |E/ B_\tau |}{2^p f! \cdot |\Aut(\Coup(E))|}   \\
    & = \sum_{\Orb(\tau) \in V^k / G_X} |\Orb(\tau)| \cdot \frac{|B_\tau| \cdot |E/ B_\tau |}{2^p f! \cdot |\Aut(\Coup(E))|}   \\
    & =  \frac{\sum_{\tau \in \mathbb{S}_n} |B_\tau| \cdot |E / B_\tau|}{2^p f! \cdot |\Aut(\Coup(E))|} .
\end{align*}
In the third equality we used \eqref{Eq:OrbitStabTheorem}, as well as the fact that the sum of $|\Orb(\tau)| \cdot |B_\tau| \cdot |E / B_\tau|$ over all orbits equals the sum of $|B_\tau| \cdot |E / B_\tau|$ over all vertices, since $|B_\tau| \cdot |E / B_\tau|$ is constant for all $\tau$ within an orbit, see Lemma~\ref{Lem:conjugatesubgroups}. 
\end{proof}

To study the orbital representation of $D^k$ under the action of $G_X$, we distinguish between the case $k = 0$ and $k \in [m]$. For $k = 0$, $D^k$ contains all arcs between $s$ and $V^1$. Therefore, each orbital of $D^0$ under $G_X$ consists of all arcs starting from $s$ and ending at vertices in an orbit of $V^1$. The arcs in $D^k$, $k \in [m]$, correspond to ordered pairs $(\tau^k, \tau^{k+1})$, where~$\tau$ represents the same qubit order in $H^k$ and $H^{k+1}$. Such an arc exists in $D^k$ whenever~$\tau^k \in F^k$, see \eqref{Def:Fk}. The orbital of $(\tau^k, \tau^{k+1})$ is the set
\begin{align*}
    \Orb((\tau^k, \tau^{k+1})) & = \{ (\rho(\tau^k), \rho(\tau^{k+1}) \, : \, \, \rho \in G_\sub \} \\
    & = \{ (a \tau^k b^{-1}, a \tau^{k+1} b^{-1} ) \, : \, \, a \in \mathbb{S}_n(\mathcal{F}), b \in \Aut(\Coup(E)) \}. 
\end{align*}
Let $D^k / G_X$ denote the set of orbitals of the group action of $G_X$ restricted to $D^k$. Since~$\tau^k$ and $\tau^{k+1}$ represent the same qubit orders in $H^k$ and $H^{k+1}$, respectively, all arcs within $\Orb((\tau^k, \tau^{k+1}))$ start and end at vertices in the same orbit. This leads to the following result.

\begin{theorem} \label{Thm:QuotientDk} The number of orbitals of $D^0$ under $G_X$ is $|D^0 / G_X |$ $= \mfrac{\sum_{\tau \in \mathbb{S}_n} |B_\tau|}{2^p f! \cdot |\Aut(\Coup(E))|}$. For $k \neq 0$, the number of orbitals of $D^k$ under $G_X$ is
$|D^k / G_X |$ $= \mfrac{\sum_{\tau \in F^k} |B_\tau|}{2^p f! \cdot |\Aut(\Coup(E))|}$. 
\end{theorem}
\begin{proof}
The first part follows directly from Theorem~\ref{Thm:QuotientVk}. For the second part, observe that we have~$D^k = \{(\tau^k, \tau^{k+1}) \, : \, \, \tau^k \in V^k,~ \tau^{k+1} \in V^{k+1},~ \tau^k \in F^k\}$, where $F^k$ is defined in~\eqref{Def:Fk}. The cardinality of $D^k / G_X$ is equal to the number of orbits of $F^k$ under the action of $G_X$ restricted to the vertices in $F^k$. The cardinality of $F^k / G_X$ can be derived similarly as in the proof of Theorem~\ref{Thm:QuotientVk}, leading to
\begin{align*}
   |D^k / G_X| = |F^k / G_X| = \frac{\sum_{\tau \in F^k} |B_\tau|}{2^p f! \cdot |\Aut(\Coup(E))|}. 
\end{align*}
\end{proof}

The results of Theorems~\ref{Thm:QuotientVk},~\ref{Thm:QuotientAk} and~\ref{Thm:QuotientDk} are summarized in Table~\ref{Table:quotients}. Moreover, we simplify the cardinalities of the quotients for the special case where $B_\tau$ is trivial for all~$\tau \in \mathbb{S}_n$.      

\begin{table}[H]
\footnotesize
\centering
\setlength{\tabcolsep}{20pt}
\begin{tabular}{@{}ccc@{}}
\toprule
\textbf{Quotient} & \textbf{Order}                                                                                   & \textbf{\begin{tabular}[c]{@{}c@{}}Order when \boldmath $B_\tau$ is \\ trivial for all \boldmath $\tau \in \mathbb{S}_n$\end{tabular}} \\ \midrule
$V^k / G_X$, $k \in [m]$             & $\dfrac{\sum_{\tau \in \mathbb{S}_n} |B_\tau|}{2^p f! \cdot |\Aut(\Coup(E))|}$                    &     $\dfrac{n!}{2^p f! \cdot |\Aut(\Coup(E))|}$                          \\[1.5em]
$A^k / G_X$, $k \in [m]$             & $\dfrac{\sum_{\tau \in \mathbb{S}_n} |B_\tau| \cdot |E / B_\tau|}{2^p f! \cdot |\Aut(\Coup(E))|}$ &    $\dfrac{n! \cdot |E|}{2^p f! \cdot |\Aut(\Coup(E))|}$                                                                                                                 \\[1.5em]
$D^0 / G_X$             & $\dfrac{\sum_{\tau \in \mathbb{S}_n} |B_\tau|}{2^p f! \cdot |\Aut(\Coup(E))|}$                             &              $\dfrac{n!}{2^p f! \cdot |\Aut(\Coup(E))|}$                                                                                                      \\[1.5em]

$D^k / G_X$, $k \in [m]$             & $\dfrac{\sum_{\tau \in F^k} |B_\tau|}{2^p f! \cdot |\Aut(\Coup(E))|}$                             &              $\dfrac{2 |E| (n-2)!}{2^p f! \cdot |\Aut(\Coup(E))|}$                                                                                                      \\[1.5em] \bottomrule
\end{tabular}
\caption{Overview of the orders of quotients $V^k / G_X$, $A^k / G_X$ and $D^k / G_X$ in terms of the cardinality of $B_\tau$. \label{Table:quotients}}
\end{table} 

\noindent In practical situations, it is often appropriate to possess an orbit (resp.\ orbital) representation of some set under a group action. Such representation contains exactly one element from each orbit (resp.\ orbital). In the sequel, we let $\mathcal{R}(V^k / G_X) \subseteq V^k$ denote an orbit representation of $\mathbb{S}_n$ under the group action $G_\sub$. We can obtain $\mathcal{R}(V^k / G_X)$ by exploiting~\eqref{Eq:OrbitCosets}. First, one can efficiently obtain a representation of left cosets of $\Aut(\Coup(E))$ in $\mathbb{S}_n$, see e.g., Dixon and Majeed~\cite{DixonMajeed}. This coset representation can be compressed to an orbit representation by a merge operation of multiple left cosets. For each representative $\tau$, we enumerate the elements of the right coset of $\mathbb{S}_n(\mathcal{F})$ containing $\tau$. This provides the representatives of left cosets that belong to the same orbit. 

An orbital representation $\mathcal{R}(A^k / G_X)$ can be obtained by exploiting the proof of Theorem~\ref{Thm:QuotientAk}. We know that each orbital can be represented by the orbit from where the arcs in the orbital start, combined with a representative element from the quotient $E / B_\tau$, where $\tau$ belongs to the orbit. Hence, $\mathcal{R}(A^k / G_X) = \{(\tau, \sigma) \, : \, \, \tau \in \mathcal{R}(V^k / G_X),~ \sigma \in \mathcal{R}(E / B_\tau)\}$, where $\mathcal{R}(E / B_\tau)$ is an orbital representation of the edges in $E$ under the action of $B_\tau$. As the coupling graph is typically small, $\mathcal{R}(E / B_\tau)$ can be obtained by enumeration. 

Finally the orbital representation $\mathcal{R}(D^k / G_X)$ follows from the subset of $\mathcal{R}(V^k / G_X)$ associated with the orbits in the set $F^k$, i.e., $\mathcal{R}(D^k / G_X) = \{(\tau^k, \tau^{k+1}) \, :  \, \, \tau^k \in \mathcal{R}(V^k / G_X),~ \tau^k \in F^k \}$. 

The orbit and orbital representations can be found more easily when the underlying coupling graph is known, see Section~\ref{Sec:SpecialCoupling}.

\section{Symmetry reduction for the NNCP}
\label{Sec:SymmetryReduction}
In this section we show how the automorphism results derived in Section~\ref{Sec:ExploitSym} can be exploited to reduce the size of the NNCP introduced in Section~\ref{Subsec:ShortestPath}.

In Section~\ref{Subsec:reducedLP} we exploit the subgroup $G_X$, see~\eqref{Eq:GXreformulated}, in order to reduce the linear programming formulation \eqref{Prob:SPP} in terms of the number of variables and constraints. In Section~\ref{Subsec:reducedSPP} we show how this reduced LP can be rewritten as a generalized network flow problem. The backward reconstruction of optimal qubit orders from the reduced model is the topic of Section~\ref{Subsec:BackwardReconstruction}, where we also provide a pseudo-code on the construction of the symmetry-reduced NNCP formulation. {\color{black} The relationship with a dynamic programming algorithm is the topic of Section~\ref{Subsec:DP}.}

\subsection{Reduced LP formulation}
\label{Subsec:reducedLP}
The elements in $G_X$ act on the vertex and arc set of~$G$. For any arc $e \in A$ and any $\rho \in G_X$, let $\rho(e)$ denote the ordered pair to which $e$ is mapped to by $\rho$, which is again in $A$ since $\rho$ is an automorphism.
Now, let $x \in \prod_{k = 1}^m \mathbb{R}^{A^k}$ and $y \in \prod_{k = 0}^m \mathbb{R}^{D^k}$ be feasible for~\eqref{Prob:SPP}. We define the {Reynolds operator} $\psi$ that maps $x$ (resp.~$y$) to the average of the images of $x$ (resp.~$y$) under the action of $G_X$ on $A$. That is,
\begin{align} \label{Def:Reynolds}
\psi(x) := \frac{1}{|G_X|} \sum_{\rho \in G_X} x^\rho \quad \text{and} \quad \psi(y) := \frac{1}{|G_X|} \sum_{\rho \in G_X}y^\rho,
\end{align}
where $x^\rho$ and $y^\rho$ are defined as $x^\rho_e = x_{\rho(e)}$ and $y^\rho_e = y_{\rho(e)}$ for all arcs $e$.
As $A^k$ for all~${k \in [m]}$ and~$D^k$ for all~$k \in \{0\} \cup [m]$ are invariant under the action of $G_X$ on $A$, it follows that~${\psi(x) \in \prod_{k = 1}^m \mathbb{R}^{A^k}}$ and $\psi(y) \in \prod_{k = 0}^m \mathbb{R}^{D^k}$. We now prove the following result, {\color{black}which is a special case of a general result on symmetries of LPs due to B\"odi et al.~\cite{BodiEtAl}, see also Gatermann and Parrilo~\cite{GatermannParrilo}.}

\begin{theorem} \label{Thm:Reynolds}
Let $(x,y) \in \prod_{k = 1}^m \mathbb{R}^{A^k} \times \prod_{k = 0}^m \mathbb{R}^{D^k}$ be feasible (resp.~optimal) for \eqref{Prob:SPP}. Then, $(\psi(x), \psi(y))$ is also feasible (resp.~optimal) for~\eqref{Prob:SPP}. 
\end{theorem}

\begin{proof}
As the flow conservation constraints hold for $(x,y)$ and $\rho$ preserves the arc structure of $X$, the pair $(x^\rho, y^\rho)$ also satisfies these constraints for all $\rho \in G_X$. It follows that $(x^\rho, y^\rho)$ is feasible for \eqref{Prob:SPP} for all $\rho \in G_X$. Observe that the pair $(\psi(x),\psi(y))$ is a convex combination of $(x^\rho, y^\rho)$ over the elements of $G_X$. Because the feasible set of \eqref{Prob:SPP} is convex, it follows that $(\psi(x), \psi(y))$ is also feasible for \eqref{Prob:SPP}. 

The objective function of \eqref{Prob:SPP} can be written as $f(x,y) := \sum_{e \in A}x_e$. Since arcs are mapped to arcs by all $\rho \in G_X$, we have $f(x^\rho, y^\rho) = f(x,y)$. We then obtain:
\begin{align*}
    f(\psi(x),\psi(y)) = \sum_{e \in A}\psi(x)_e =  \frac{1}{|G_X|} \sum_{\rho \in G_X} \sum_{e \in A} x^\rho_e = \frac{1}{|G_X|} |G_X| \sum_{e \in A} x_e = f(x,y). 
\end{align*}
Thus, if $(x,y)$ is optimal for \eqref{Prob:SPP}, then so is $(\psi(x),\psi(y))$. 
\end{proof}

An implication of Theorem~\ref{Thm:Reynolds} is that we may restrict the feasible set of \eqref{Prob:SPP} to the subspace 
\begin{align} \label{Def:H_G_X}
\mathcal{H}_{G_X} := \left\{ (\psi(x),\psi(y)) \, : \, \, (x,y) \in \prod_{k = 1}^m \mathbb{R}^{A^k} \times \prod_{k = 0}^m \mathbb{R}^{D^k} \right\},
\end{align}
which is also denoted as the fixed point subspace in~\cite{BodiEtAl}. By construction of the Reynolds operator~\eqref{Def:Reynolds}, the entries in $\psi(x)$ belonging to the same orbital are equal. Therefore, the subspace $\mathcal{H}_{G_X}$ is spanned by the incidence vectors of orbitals of $X$. In Section~\ref{Subsec:Orbit} we derived the orbital structure of the action of $G_X$ on $X$. Recall that $A^k / G_X$ denotes (the index set of) the collection of orbitals of $A^k$ under the action of $G_X$. Now, if we denote the $i$th orbital of $A^k$ by $W^k_i$, we obtain
\begin{align}
A^k &= \bigsqcup_{i \in A^k / G_X} W^k_i \quad \text{for all } k \in [m],
\intertext{{\color{black}where $\sqcup$ denotes the disjoint union of sets.} In a similar fashion, the arc sets $D^k, k \in \{0\}\cup [m]$ can be partitioned into its collection of orbitals. If $Z^k_i$ denotes the $i$th orbital of $D^k$, then}
D^k &= \bigsqcup_{i \in D^k / G_X} Z^k_i \quad \text{for all } k \in \{0\} \cup [m].
\end{align}
Now, the subspace $\mathcal{H}_{G_X}$ can be rewritten as: 
\begin{align} 
\mathcal{H}_{G_X} =  \prod_{k = 1}^m \left( \Span\{ \mathbbm{1}_{W^k_i} \, : \, \, i \in A^k / G_X\} \right) \times \prod_{k = 0}^m \left(\Span\{ \mathbbm{1}_{Z^k_i} \, : \, \, i \in D^k / G_X\} \right),
\end{align}
which implies that the characteristic vectors of the orbitals form a basis for $\mathcal{H}_{G_X}$. 

Also the orbits of each of the vertex sets $V^k$ under the action of $G_X$ induce a partition of $V^k$. Let $V^k / G_X$ denote (the index set of) the collection of orbits of $V^k$ under $G_X$. The $u$th orbit of $V^k$ is denoted by $O^k_u$, with $u \in V^k / G_X$. Then,
\begin{align}
V^k = \bigsqcup_{u \in V^k / G_X} O^k_u \quad \forall k \in [m].
\end{align} 
To write the symmetry-reduced equivalent of~\eqref{Prob:SPP} explicitly, we need some further terminology. Let the out-degree $d^+(\tau, W^k_i)$ (resp.\ in-degree $d^-(\tau, W^k_i)$) denote the number of arcs in orbital $W^k_i$ that start (resp.\ end) at vertex $\tau$, i.e.,
\begin{align*}
d^+(\tau, W^k_i) :=  \left\lvert \left\{ (\tau, \tau \sigma) \in W^k_i \, : \, \, \sigma \in T \right\}\right\rvert \quad \text{and} \quad d^-(\tau, W^k_i) & := \left\lvert \left\{ (\tau \sigma, \tau) \in W^k_i \, : \, \,  \sigma \in T\right\}\right\rvert, 
\end{align*}
for all $i \in A^k / G_X$ and $k \in [m]$. 
Since $d^+(\tau_1, W_i^k) = d^+(\tau_2, W_i^k)$ for all orbitals $i$ when~$\tau_1$ and~$\tau_2$ belong to the same orbit, it makes sense to define $d^+(W^k_i)$  ($:= d^+(\tau, W_i^k )$ for any~$(\tau, \tau \sigma) \in W^i_k$) as the orbital out-degree in $W^k_i$. In a similar fashion we define $d^-(W^k_i)$.

From Lemma~\ref{Lem:BijectionOrbital} we know that there is a single case in which~$d^+(\tau, W^k_i) > 1$. Namely, two distinct arcs $(\tau, \tau\sigma_1)$ and $(\tau, \tau\sigma_2)$ with $\sigma_1 = (i~j)$ are both in the same orbital~$W^k_i$ if and only if there exists a $b \in B_\tau$ such that $\sigma_2 = (b(i)~b(j))$. This corresponds to case II in Figure~\ref{Fig:OverviewOrbital}. Hence, we have
\begin{align*}
    d^+(\tau, W^k_i) &= \left \lvert \left\{b(\{i,j\}) \, : \, \, b \in B_\tau\right \} \right\rvert \text{ for some } (\tau, \tau (i~j)) \in W^k_i, \\
    d^-(\tau, W^k_i) &= \left \lvert \left \{b(\{i,j\}) \, : \, \, b \in B_\tau\right \} \right\rvert \text{ for some } (\tau(i~j), \tau ) \in W^k_i. 
\end{align*}
Indeed, these equal the number of elements in an orbital of $\Coup(E)$ under the action of~$B_\tau$. Moreover, we also define $d^+(Z^0_i)$ (\mbox{resp.\ $d^-( Z^m_i)$}) as the number of arcs in orbital~$Z^0_i$ (\mbox{resp.\ $Z^m_i$}) starting from $s$ (resp.\ ending at $t$). For these degrees one can verify that~${d^+(Z^0_i) = |Z^0_i|}$ and $d^-(Z^m_i) = |Z^m_i|$.

For any vertex $\tau$, we let $\delta^+(\tau, A^k / G_X)$ (resp.\ $\delta^-(\tau, A^k / G_X)$) denote the set of orbitals that contain an arc starting (resp.\ ending) at vertex $\tau$. That is,
\begin{align*}
\delta^+(\tau, A^k / G_X) & := \left\{ i \in A^k / G_X \, : \, \, (\tau, \tau \sigma) \in W^k_i \text{ for some } \sigma \in T\right\}, \\
\delta^-(\tau, A^k / G_X) & := \left\{ i \in A^k / G_X \, : \, \, (\tau \sigma , \tau) \in W^k_i \text{ for some } \sigma \in T\right\}. 
\end{align*}
Similar definitions hold for $\delta^+(\tau, D^k / G_X)$ and $\delta^-(\tau, D^k / G_X)$. Again, observe that if~$\tau_1$ and~$\tau_2$ belong to the same orbit $O^k_u$, then $\delta^+(\tau_1, A^k/G_X) = \delta^+(\tau_2, A^k/G_X)$. 
For that reason, it makes sense to define $\delta^+(O^k_u , A^k/G_X)$, which is equal to $\delta^+(\tau, A^k/G_X)$ for any $\tau \in O^k_u$. 
In a similar fashion, we define $\delta^-(O^k_u , A^k/G_X)$, $\delta^+(O^k_u , D^k/G_X)$ and~$\delta^-(O^k_u , D^k/G_X)$ for all $u \in V^k / G_X$ and $k \in [m]$. 

\medskip

The symmetry reduced equivalent formulation of \eqref{Prob:SPP} is obtained by replacing every variable~$x_e$ in $H^k$ by a variable $\lambda^k_i$ corresponding to the orbital $W^k_i$ to which arc $e$ belongs. Similarly, we replace every variable $y_e$ in $D^k$ by a variable $\theta^k_i$ corresponding to the orbital~$Z^k_i$ to which arc $e$ belongs. As a consequence, the flow conservation constraint corresponding to vertices that belong to the same orbit becomes equivalent, hence we only keep one per orbit. The remaining linear programming problem we denote by \eqref{Prob:RSPP} and is given by
\leqnomode
\begin{align}  \tag{RSPP}\label{Prob:RSPP} \qquad\quad  \begin{aligned}
\min \quad & \sum_{k =1}^m \sum_{i \in A^k / G_X}|W^k_i|\lambda^k_i  \\
\text{s.t.}\quad & \sum_{i \in D^0 / G_X} d^+(Z^0_i)\theta^0_i = 1,~ \sum_{i \in D^m / G_X} d^-( Z^m_i)\theta^m_i = 1 \\
& \begin{aligned} & \sum_{\substack{i \in \delta^-(O^k_u, \\ D^{k-1}/G_X)}}\theta^{k-1}_i + \sum_{\substack{i \in \delta^-(O^k_u, \\ A^k /G_X)}}d^-(W^k_i)\lambda_i^k  =&& \\
 & \qquad \quad \sum_{\substack{i \in \delta^+(O^k_u, \\ D^k / G_X)}}\theta^{k}_i + \sum_{\substack{i \in \delta^+(O^k_u,\\ A^k / G_X)}}d^+( W^k_i)\lambda_i^k  \qquad \forall u \in V^k / G_X, ~k \in [m] \end{aligned}\\
& 0 \leq \lambda^k_i \leq 1 \quad \forall i \in A^k / G_X,~ k \in [m] \\
& 0 \leq \theta^k_i \leq 1 \quad \forall i \in D^k / G_X,~ k \in \{0\} \cup [m]. \end{aligned} \hspace{-1cm}
\end{align}
Observe that $|W^k_i|, d^+(Z^0_i)$ and $d^-(Z^m_i)$ for all appropriate $k$ and $i$ are proportional to the size of an orbit in one of the subgraphs, which is in turn proportional to $|\Aut(\Coup(E))|$, see~\eqref{Eq:OrbitStabTheorem}. For highly symmetric coupling graphs, the order of this automorphism group becomes very large, leading to extreme coefficient values in~\eqref{Prob:RSPP}. This may lead to numerical instability when solving such program.

To improve practical performance, we apply a scaling operation prior to solving the program. We first multiply both sides of the flow conservation constraints by $|\Aut(\Coup(E))|$ for all $u \in V^k / G_X$ and $k \in [m]$. After that, we apply the following substitution:
\begin{align*}
    \overline{\lambda}_i^k &:= |\Aut(\Coup(E))| \lambda_i^k \qquad \text{for all } i \in A^k / G_X, k \in [m], \\
    \overline{\theta}_i^k &:= |\Aut(\Coup(E))| \theta_i^k \qquad \text{for all } i \in D^k / G_X, k \in \{0\} \cup [m]. 
\end{align*}
This leads to the equivalent linear program~\eqref{Prob:RSPP_scale}. Observe that the new upper bounds on $\overline{\lambda}_i^k$ and $\overline{\theta}_i^k$ are omitted in this program. Indeed, the out-degree of $s$ and in-degree of $t$ needs to be 1, which implicitly enforces an upper bound of 1 on all $\overline{\theta}_i^0$ and $\overline{\theta}_i^m$. Since all variables and coefficient values are nonnegative and flow conservation holds throughout the program, 
we can without loss of generality omit the upper bounds on the variables. Hence, the coefficients of this program no longer depend on $|\Aut(\Coup(E))|$.
\begin{align}  \tag{RSPP$'$}\label{Prob:RSPP_scale} \qquad\quad  \begin{aligned}
\min \quad & \sum_{k =1}^m \sum_{i \in A^k / G_X}\mfrac{|W^k_i|}{|\Aut(\Coup(E))|}\overline{\lambda}^k_i  \\
\text{s.t.}\quad & \sum_{i \in D^0 / G_X} \mfrac{d^+(Z^0_i)}{|\Aut(\Coup(E))|}\overline{\theta}^0_i = 1,~ \sum_{i \in D^m / G_X} \mfrac{d^-( Z^m_i)}{|\Aut(\Coup(E))|}\overline{\theta}^m_i = 1 \\
& \begin{aligned} & \sum_{\substack{i \in \delta^-(O^k_u, \\ D^{k-1}/G_X)}}\overline{\theta}^{k-1}_i + \sum_{\substack{i \in \delta^-(O^k_u, \\ A^k /G_X)}}d^-(W^k_i)\overline{\lambda}_i^k  =&& \\
 & \qquad \quad \sum_{\substack{i \in \delta^+(O^k_u, \\ D^k / G_X)}}\overline{\theta}^{k}_i + \sum_{\substack{i \in \delta^+(O^k_u,\\ A^k / G_X)}}d^+( W^k_i)\overline{\lambda}_i^k  \qquad \forall u \in V^k / G_X,~k \in [m] \end{aligned}\\
& 0 \leq \overline{\lambda}^k_i  \quad \forall i \in A^k / G_X,~k \in [m] \\
& 0 \leq \overline{\theta}^k_i  \quad \forall i \in D^k / G_X,~k \in \{0\} \cup [m] \end{aligned}  \hspace{-1cm}
\end{align}

\reqnomode

Recall that the NNCP is in general $\mathcal{NP}$-hard \cite{SiraichiEtAl}.
Based on the LP formulation~\eqref{Prob:RSPP_scale}, we are able to unfold some special cases where the problem turns out to be polynomial time solvable. The condition that provides the key to this complexity result is the order of the automorphism group of the coupling graph.

Since all permutations in $B_\tau$ should setwise stabilize the sets $\tau^{-1}(S_i)$ for all $i \in [l]$, it follows that~$B_\tau$ is a subgroup of $\mathbb{S}_n(\mathcal{G})$, where $\mathcal{G} := \{\tau^{-1}(S_1), \ldots, \tau^{-1}(S_l)\}$. The order of~$\mathbb{S}_n(\mathcal{G})$ is~$2^p f!$, which implies that $|B_\tau| \leq 2^p f!$. This leads to the following complexity result.

\begin{theorem} \label{Thm:Complexity}
The NNCP is polynomial time solvable on coupling graphs with automorphism groups of order $\Omega((n-b)!)$, where $n$ is the number of vertices in the coupling graph and $b$ is a constant independent of $n$.
\end{theorem}
\begin{proof}
The number of variables in~\eqref{Prob:RSPP} equals $m|A^1 / G_X| + |D^0 / G_X| + m|D^1 / G_X|$. Based on Table~\ref{Table:quotients} and the inequalities $|B_\tau| \leq 2^p f!$, $|E / B_\tau| \leq |E|$ and~\mbox{$|F^k| \leq 2|E|(n-2)!$} for all~$\tau \in \mathbb{S}_n$ and~$k \in [m]$, we have
\begin{align*}
    & \quad \frac{m \sum_{\tau \in \mathbb{S}_n} |B_\tau| \cdot |E / B_\tau|}{2^p f! \cdot |\Aut(\Coup(E))|} + \frac{\sum_{\tau \in \mathbb{S}_n} |B_\tau|}{2^p f! \cdot |\Aut(\Coup(E))|}  + \frac{m \sum_{\tau \in F^1} |B_\tau|}{2^p f! \cdot |\Aut(\Coup(E))|}  \\
    \leq & \quad \frac{m~2^p f! |
    E| n!}{2^p f! \cdot |\Aut(\Coup(E))|} + \frac{2^p f! n!}{2^p f! \cdot |\Aut(\Coup(E))|}  + \frac{m~2 |E| (n-2)! 2^p f!}{2^p f! \cdot |\Aut(\Coup(E))|} \\
    = &\quad  O \left( \frac{m~|E|~n!}{|\Aut(\Coup(E))|}\right). 
\end{align*}
Whenever $|\Aut(\Coup(E))| = \Omega ((n-b)!)$, the number of variables in~\eqref{Prob:RSPP} is $O \left( m|E| n^b \right)$. Since~$b$ does not depend on the input, the number of variables in the reduced instance is polynomial in $n$,~$m$ and $|E|$. 
\end{proof}
The implication of Theorem~\ref{Thm:Complexity} does not solely restrict to trivial NNCP classes, such as the ones with a coupling graph that is complete. An example of a less trivial class of coupling graphs having a sufficiently large automorphism group are the bicliques, i.e., the complete bipartite graphs. 
\begin{corollary} \label{Cor:bicliquecomplexity}
The NNCP is polynomial time solvable on the biclique $K_{M,N}$ with $M$ of fixed size. In particular, the NNCP on the star $K_{1,N}$ is polynomial time solvable. 
\end{corollary}

\subsection{Reduced combinatorial formulation}
\label{Subsec:reducedSPP}
Similar to~\eqref{Prob:SPP} being an LP formulation of a shortest path problem, we show in this section that~\eqref{Prob:RSPP} and \eqref{Prob:RSPP_scale} also have a combinatorial interpretation. Such combinatorial approaches often have the potential to induce efficient algorithms that are favoured over solving their LP formulation. In order to simplify notation, we work with~\eqref{Prob:RSPP} in this section, although the construction for~\eqref{Prob:RSPP_scale} is similar.

To view~\eqref{Prob:RSPP} as a combinatorial problem, we consider the so-called quotient graph of~$X$ under the action of $G_X$. In its most general form, a quotient graph of a graph $X$ is induced by an equivalence relation on the vertices of $X$. We below provide the formal definition for the particular case where the equivalence relation is induced by an automorphism group of~$X$. 

\begin{definition}[Quotient graph implied by automorphisms]
Let $X = (V,A)$ be a directed graph and let $G$ be a subgroup of $\Aut(X)$. Then the quotient graph of $X$ under $G$ is the graph $\mathcal{X} = (\mathcal{V}, \mathcal{A})$ with $\mathcal{V} := V / G$ and $\mathcal{A} := A / G \subseteq \mathcal{V} \times \mathcal{V}$. 
\end{definition}

Since all arcs within an orbital of $X$ start at vertices in the same orbit and end at vertices in the same orbit, the quotient graph is well-defined. Observe that $\mathcal{X}$ can contain loops and multi-arcs, even if $X$ is simple.

Let $\mathcal{X} = (\mathcal{V},\mathcal{A})$ be the quotient graph of $X$ under $G_X$. Since the source vertex $s$ and the sink vertex $t$ are in isolated orbits, the vertices $s$ and $t$ are again in $\mathcal{V}$. By abuse of notation, we again denote these vertices as $s, t \in \mathcal{V}$. Since the constraints and variables in~\eqref{Prob:RSPP} correspond to orbits and orbitals of $X$ under $G_X$, respectively, the problem~\eqref{Prob:RSPP} is an optimization problem on the quotient graph $\mathcal{X}$. Now, for all $(j,\ell) \in \mathcal{A}$ we define the following flow variable $f_{j \ell}$: 
\begin{align} \label{Eq:flowf}
    f_{j \ell}&:= \begin{cases} 
    d^+(Z^0_i)\theta^{0}_i & \text{if $(j, \ell)$ corresponds to $Z^0_i$,} \\
    \theta_{i}^k & \text{if $(j, \ell)$ corresponds to $Z^k_i$, $\,\,k \in [m]$,} \\
    d^+(W^k_i)\lambda^k_i & \text{if $(j, \ell)$ corresponds to $W^k_i$, $k \in [m]$.}
\end{cases}
\intertext{Moreover, we define for all $(j, \ell) \in \mathcal{A}$ a cost vector}
    w_{j \ell}&:= \begin{cases} 
    \frac{|W^k_i|}{d^+(W^k_i)} \quad~& \text{if $(j, \ell)$ corresponds to $W^k_i$, $k \in [m]$,}\\
    0 & \text{otherwise,}
\end{cases}
\intertext{and an upper bound vector}
    u_{j \ell}&:= \begin{cases} 
    d^+(Z^0_i) \quad~& \text{if $(j, \ell)$ corresponds to $Z^0_i$,} \\
    1 & \text{if $(j, \ell)$ corresponds to $Z^k_i$, $\,\,k \in [m]$,} \\
    d^+(W^k_i) & \text{if $(j, \ell)$ corresponds to $W^k_i$, $k \in [m]$.}
\end{cases}
\intertext{Finally, for all $(j,\ell) \in \mathcal{A}$ we define a multiplier $p_{j \ell}$:}
p_{j \ell}&:= \begin{cases} 
    \frac{d^-(W^k_i)}{d^+(W^k_i)}\quad~& \text{if $(j, \ell)$ corresponds to $W^k_i$, $k \in [m]$,}\\
    d^-(Z^m_i) & \text{if $(j, \ell)$ corresponds to $Z^m_i$,} \\
    1 & \text{otherwise.}
\end{cases}
\end{align}
We now substitute $f_{j\ell}, w_{j \ell}$ and $p_{j\ell}$ for all orbitals $(j, \ell) \in \mathcal{A}$ into~\eqref{Prob:RSPP}. This yields an equivalent linear programming problem that has the structure of a minimum cost generalized network flow problem: 
\leqnomode
\begin{align}  \tag{GNFP} \label{Prob:GNFP} \qquad\qquad  \begin{aligned}
\min \quad &  \sum_{(j, \ell) \in \mathcal{A}} w_{j \ell} f_{ j \ell}\\
\text{s.t.} \quad & \sum_{(j,\ell) \in \delta^+(s)}f_{j \ell} = 1,~ \sum_{(j,\ell) \in \delta^-(t)}p_{j \ell}f_{j \ell} = 1 \\
& \sum_{(j,\ell) \in \delta^+(v)}f_{j \ell} = \sum_{(j,\ell) \in \delta^-(v)}p_{j \ell}f_{j \ell} \qquad \forall v \in \mathcal{V} \setminus \{s,t\} \\
& 0 \leq f_{j \ell} \leq u_{j \ell} \qquad \forall (i,j) \in \mathcal{A}. 
\end{aligned} 
\end{align}
\reqnomode
A generalized flow is a flow starting from a sink $s$, conserving the flow at each vertex and ending at a source $t$, where along each arc $(j, \ell)$ only a fraction of $p_{j\ell}$ of flow is moved from~$j$ to~$\ell$. This fraction, called the multiplier, can also be larger than one, which means that the flow is increased along the arc.
The problem~\eqref{Prob:GNFP} aims to send a generalized flow of one from $s$ to $t$ that has a minimal cost with respect to the cost vector $w$. The minimum cost generalized network flow problem is solvable in weakly polynomial time by the algorithm of Wayne~\cite{Wayne}. This is the only known combinatorial algorithm for this problem in the literature. 

\subsection{Symmetry-reduced NNCP algorithm} \label{Subsec:BackwardReconstruction}
In this section we show how an optimal solution to~\eqref{Prob:RSPP} or \eqref{Prob:GNFP} can be used to find an optimal sequence of qubit orders for the NNCP. Moreover, we briefly present an overview of the entire solution approach in terms of a pseudo-code. 

 By construction, solving~\eqref{Prob:RSPP} or ~\eqref{Prob:GNFP} provides the optimal cost of a shortest path in $X$. However, because of the reduction, the solutions of~\eqref{Prob:RSPP} or ~\eqref{Prob:GNFP} do no longer correspond itself to paths. Let $(\lambda, \theta)$ be an optimal solution to~\eqref{Prob:RSPP} (in case of solving~\eqref{Prob:GNFP}, we can obtain $(\lambda, \theta)$ from the flow variable $f$ by~\eqref{Eq:flowf}). Now, we define~$(x,y) \in \prod_{k =1}^m \mathbb{R}^{A^k} \times \prod_{k =0}^m \mathbb{R}^{D^k}$ as follows: 
\begin{align} \label{Eq:XandYsupport}
    x  := \left( \sum_{i \in A^k / G_X} \lambda^k_i\mathbbm{1}_{W^k_i} \right)_{k =1}^m \quad \text{and} \quad
    y  :=  \left( \sum_{i \in D^k / G_X} \theta^k_i\mathbbm{1}_{Z^k_i} \right)_{k =1}^m. 
\end{align}
It follows from the construction that the pair $(x,y)$ corresponds to an optimal solution of~\eqref{Prob:SPP}. Hence, it is a convex combination of characteristic vectors of $(s,t)$-paths in $X$. Let $X^{\sup}$ denote the subgraph of $X$ induced by the support of $(x,y)$. Then, $X^{\sup}$ is an acyclic graph. Namely, if there would exist a cycle in $X^{\sup}$, due to the orientation of the arcs in~$X$, it can only consist of arcs within one subgraph. Since these arcs all have a positive cost, the solution $(x,y)$ can be improved by excluding the cycle from it. By a similar argument, it follows that any $(s,t)$-path in $X^{\sup}$ must be optimal. Namely, if there is an $(s,t)$-path in the support with a strictly larger cost than the optimum, we can improve the solution $(x,y)$ by excluding this path from it.

These observations imply that any~$(s,t)$-path in $X^{\sup}$ is optimal. Finding such path can be done without actually constructing $X^{\sup}$. Starting from~$s$, we select an arbitrary arc from an orbital in~$D^0 / G_X$ that is in the support of $\theta^0$. This arc leads to a new vertex $\tau$. From the orbit where $\tau$ belongs to, we again select an orbital leaving this orbit that has a support in the optimal solution $(\lambda, \theta)$. Within this orbital, there is at least one arc starting from $\tau$ and we select such an arc arbitrarily if there are multiple. We continue doing this, which will eventually lead to the sink vertex $t$. It follows from the discussion above that this $(s,t)$-path provides an optimal qubit ordering for the NNCP. 

We end this section by giving an overview of the symmetry-reduced NNCP algorithm. The approach is given in pseudo-code in Algorithm~\ref{Alg:FinalAlgorithm}. We emphasize that it does not rely on the use
of algebraic software, nor does it require a construction of the full graph $X$.

\algrenewcommand\algorithmicrequire{\textbf{Input:}}
\algrenewcommand\algorithmicensure{\textbf{Output:}}
\begin{algorithm}[h!]
\scriptsize
\caption{Symmetry-reduced NNCP algorithm}\label{Alg:FinalAlgorithm}
\begin{algorithmic}[1]
\Require NNCP instance $\Gamma = (Q,C)$ and coupling graph $\Coup(E) = (L,E)$. 
\State Construct gate graph $(Q,U)$ and fixing pattern $\mathcal{F}$. 
\State Construct orbit representation $\mathcal{R}$ of $\mathbb{S}_n$ under the action $G_\sub$. 
\State Initialize quotient subgraph $\mathcal{H}^{\sub} = (\mathcal{V}^{\sub},~\mathcal{A}^{\sub})$ with $\mathcal{V}^{\sub}$ indexed by $\mathcal{R}$ and $\mathcal{A}^{\sub} = \emptyset$. 
\For{$\tau \in \mathcal{R}$}
\State Determine $B_\tau$ and $\mathcal{R}(E / B_\tau)$
\For{$\{i,j\} \in \mathcal{R}(E / B_\tau)$} 
\State Construct an arc in $\mathcal{H}$ between $\tau$ and the orbit to which $\tau(i~j)$ belongs, and add it to $\mathcal{A}^{\sub}$.
\State Compute the cardinality of $\Orb(\tau, \tau(i~j))$ and its out- and in-degree in $\Orb(\tau)$ and $\Orb(\tau(i~j))$, respectively. 
\EndFor
\EndFor
\State Initialize the quotient graph $\mathcal{X} = (\mathcal{V},\mathcal{A})$ where $\mathcal{V}$ consists of $m$ copies of $\mathcal{V}^{\sub}$, a source $s$ and a sink $t$, and $\mathcal{A}$ consists of all arcs within the subgraphs $\mathcal{A}^{\sub}$. Add to $\mathcal{A}$ all arcs between $s$ and the first subgraph.
\For{$g^k\in C$}
\State Determine $\mathcal{R}(D^k / G_X)$.
\If{$g$ is the $m$th quantum gate}
\State For all orbit representatives $\tau \in \mathcal{R}(D^k / G_X)$, add arcs from $\Orb(\tau)$ to sink vertex $t$.
\Else 
\State For all orbit representatives $\tau \in \mathcal{R}(D^k / G_X)$, add arcs from $\Orb(\tau)$ to same orbit in next subgraph.
\EndIf
\EndFor
\State Obtain optimal $(\lambda,\theta)$ pair via solving either~\eqref{Prob:RSPP}, \eqref{Prob:RSPP_scale} or the generalized network flow problem~\eqref{Prob:GNFP}.
\State Find an optimal sequence of qubit orders $\tau^k, k \in [m]$ by identifying any $(s,t)$-path in the support of $(x,y)$, where $(x,y)$ are defined as in~\eqref{Eq:XandYsupport}. 
\Ensure $\tau^k$, $k \in [m]$
\end{algorithmic}
\end{algorithm}

{
\color{black}
\subsection{Dynamic programming algorithm}
\label{Subsec:DP}

Suppose we additionally assume that $\mathbb{S}_n(\mathcal{F})$ is trivial, i.e., the fixing pattern only consists of singletons. Notice that this assumption holds for most quantum circuits, as the gate graph $(Q,U)$ is often connected even if the number of gates is small. In that case, the subgroup $B_\tau$ is trivial for all $\tau \in \mathbb{S}_n$, implying that the problem~\eqref{Prob:GNFP} can be solved more efficiently. Indeed, we then have $d^+(W^k_i) = d^-(W^k_i) = 1$ for all orbitals $W^k_i$, hence~$p_{j \ell} = 1$ for all~$W^k_i$. Now, for all $(j, \ell) \in \delta^-(t)$ we replace $p_{j \ell}f_{j \ell}$ by a new variable, say~$g_{j \ell}$, that is upper-bounded by~$d^-(Z^m_i)$. After these modifications, the resulting problem equals the LP formulation of a shortest path problem, for which strongly-polynomial time algorithms exist~\cite{FredmanTarjan}. 

In fact, due to the layered structure of the quotient graph $\mathcal{X} = (\mathcal{V}, \mathcal{A})$, the shortest path problem can be solved by a dynamic programming algorithm, which we briefly describe below. 

As $\mathbb{S}_n(\mathcal{F})$ is trivial, each orbit $\Orb(\tau)$ can be written as $\tau \Aut (\Coup (E))$, see~\eqref{Eq:OrbitCosets}. Hence, the orbits under the action of $G_X$ correspond to the left cosets of $\Aut(\Coup(E))$ in $\mathbb{S}_n$ and the orbit representation $\mathcal{R}$ is in fact a representation of left cosets of $\mathbb{S}_n$ in $\Aut(\Coup(E))$. Now, we define for each gate $k \in [m]$ and each representative $\tau \in \mathcal{R}$ the integer
\begin{align*}
    N(k,\tau ) := \quad & \text{the optimal number of SWAP gates to be inserted up to gate $g^k$} \\ & \text{if the qubit order prior to applying gate $g^k$ is $\tau$},
\end{align*}
where we set $N(k,\tau) = \infty$ if qubit order $\tau$ does not comply with gate $g^k$. Now, the values $N(k,\tau)$ are computed by the following dynamic programming iteration. Define $N(0,\tau) = 0$ for all $\tau \in \mathcal{R}$ and let
\begin{align*}
    N(k,\tau) = \begin{cases}
        \infty & \text{if $\tau^{-1}(g^k) \notin E$,} \\
        \min \big\{ N(k-1,\tau),~\min_{\hat{\tau} \in \mathcal{R} \setminus \{\tau\}} \left \{ N(k-1,\hat{\tau}) + \mathcal{J}_T(\hat{\tau}, \tau ) \right\} \big\} & \text{otherwise,}
    \end{cases}
\end{align*}
for all $k \in [m]$ and $\tau \in \mathcal{R}$, where
\begin{align*}
    \mathcal{J}_T(\hat{\tau}, \tau ) := \min\{ J_T(\hat{\tau}, \tilde{\tau}) \, : \, \, \tilde{\tau} \in \tau \Aut(\Coup(E))\}, 
\end{align*}
with the metric $J_T$ as defined in~\eqref{def:metricJ}. Thus, $\mathcal{J}_T(\hat{\tau},\tau)$ denotes the minimum distance $J_T$ between any pair of representatives of the cosets containing $\hat{\tau}$ and $\tau$. The dynamic programming algorithm relies on the fact that the shortest path from $s$ to qubit order $\tau$ in layer $k$ can be computed by considering all connections between qubit orders of layers $k-1$ and $k$. To transform qubit order $\hat{\tau}$ into $\tau$, a total number of $\mathcal{J}_T(\hat{\tau},\tau)$ SWAP gates need to be inserted, whereas no additional SWAP gates are needed if $\tau$ is already selected as the qubit order in the previous layer $k-1$. The optimal solution to the NNCP is given by $\min_{\tau \in \mathcal{R}} N(m,\tau)$.

The efficiency of the dynamic programming algorithm relies on the number of representatives in $\mathcal{R}$ and the complexity of computing the distance $\mathcal{J}_T$. For example, when the coupling graph is the star graph, we have $|\mathcal{R}| = n$ and $\mathcal{J}_T(\tau_1,\tau_2) = 1$ for all $\tau_1,\tau_2 \in \mathcal{R}$ with $\tau_1 \neq \tau_2$. The latter follows from the fact that each representative is fully characterized by the qubit that is placed on the center location of the star graph and that only one SWAP gate is necessary to place any other qubit on this center location. For the cycle graph, however, $|\mathcal{R}|$ is still exponentially large, although the metric $J_T$ on cyclic coupling graphs can be computed in polynomial time~\cite{Jerrum}. For general coupling graphs, even the computation of $J_T$, and thus of $\mathcal{J}_T$, is $\mathcal{NP}$-hard~\cite{Miltzow}, limiting the applicability of this dynamic programming algorithm. In such cases, the symmetry-reduced NNCP algorithm described in Algorithm~\ref{Alg:FinalAlgorithm} is favoured.

In case $\mathbb{S}_n(\mathcal{F})$ is not trivial, the above dynamic programming algorithm is still valid, although the approach does not exploit all symmetries in the underlying instance. In contrast, the method described in Algorithm~\ref{Alg:FinalAlgorithm} does make use of the full symmetry information of the problem.

}

\section{Special coupling graphs} \label{Sec:SpecialCoupling}
Of key importance in the formulation discussed in Section~\ref{Sec:SymmetryReduction} are the orbit and orbital representation of the subgraphs, which rely on the subgroups $B_\tau$. These objects heavily depend on the specific coupling graph. In this section we demonstrate how these objects are obtained for three specific structured coupling graphs: the cycle graph, the biclique graph and the star graph.

{\color{black}
As mentioned before, other graph types that are frequently used in the topology of quantum architectures are the linear array and two-dimensional grid graphs.  In~\cite{Mulderij}, the symmetry of the linear array (whose automorphism group has size 2) is exploited in a shortest path formulation of the NNCP. The symmetry-reduction of the NNCP on square grid graphs is described in~\cite{MeijerThesis}. Since these graphs have only small automorphism groups, symmetry-reduction techniques are less powerful on such coupling graphs. For that reason, we exclude them from the analysis in this paper, but we refer the interested reader to the above-mentioned papers for further details.
}

Table~\ref{Tab:summarycoupling} provides an overview of certain important characteristics of each of the considered coupling graphs. Details are provided in the subsections below.

\begin{table}[h!]
\centering
\footnotesize
\setlength{\tabcolsep}{5pt}
\begin{tabular}{@{}lccccl@{}}
\toprule
\textbf{Architecture} & \boldmath $n$ & \boldmath $|E|$ & \begin{tabular}{@{}c@{}} \textbf{Graph} \\ \textbf{structure} \end{tabular} & \boldmath $\Autbold(\Coupbold(E))$ & \hspace{0.8cm} \boldmath $|B_\tau|$ \\[0.5em] \midrule 

\textbf{Cycle} \boldmath $C_N$ & $N$ & $N$ & \begin{tabular}{@{}c@{}}
\vspace{0.1cm} \\
     \includegraphics[scale=0.2]{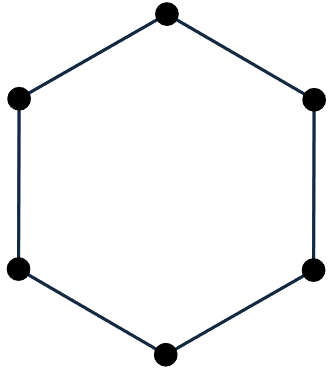} \\
     Example of $C_6$
    \end{tabular} & $\mathcal{D}_{2n}$ & 
    \begin{tabular}{@{}l@{}}1 (unless the NNCP instance \\ 
    \hspace{0.2cm} is trivial, see Theorem~\ref{Thm:trivialCycle}) \end{tabular} \\[5em]

\begin{tabular}{@{}l@{}} \textbf{Biclique} 
\boldmath $K_{M,N}$ \end{tabular} & $M + N$ & $MN$ & \begin{tabular}{@{}c@{}} \includegraphics[scale=0.17]{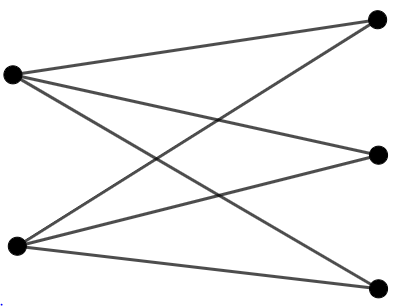} \\
    Example of $K_{2,3}$
    \end{tabular} & \begin{tabular}{@{}c@{}} $\mathbb{S}_M \times \mathbb{S}_{N}$ \\ (if $M \neq N$) \end{tabular} & \begin{tabular} {@{}l@{}}$2^{p - \hat{p}}f_1! f_2!$, where $\hat{p}, f_1$ and $f_2$ \\
    follow from Theorem~\ref{Thm:Biclique} \end{tabular}\\[4em]
    
\textbf{Star} \boldmath $K_{1,N}$ & $N + 1$ & $N$ & \begin{tabular}{@{}c@{}} \includegraphics[scale=0.2]{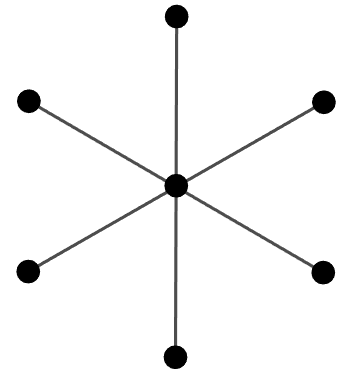} \\
    Example of $K_{1,6}$
    \end{tabular} & $\mathbb{S}_{n-1}$ & \begin{tabular} {@{}l@{}}$2^{p - \hat{p}}f_1! f_2!$, where $\hat{p}, f_1$ and $f_2$ \\
    follow from Theorem~\ref{Thm:Biclique} \end{tabular} \\[1em]    
    
\bottomrule
\end{tabular}
\caption{Summary of NNCP symmetry reduction characteristics for a set of special coupling graphs.\label{Tab:summarycoupling}}
\end{table}

\subsection[Cycle graph $C_N$]{Cycle graph \boldmath $C_N$} \label{Subsec:Cycle}
Let $C_N = (L,E)$ be the cycle on $N$ vertices, i.e., $L = [N]$ and ${E = \{ \{i, i+1\} \, : \, \, i \in [N-1]\} \cup \{N, 1\}}$. Then $n = |L| = N$. It is well-known that the automorphism group of~$C_N$ is given by~$\mathcal{D}_{2n}$, the dihedral group of order $2n$, see e.g., Godsil and Royle~\cite{GodsilRoyle}. This group consists of all reflections and rotations of the regular polygon of order $n$. It follows from Theorem~\ref{Thm:AllNormal} that $\Cay(\mathbb{S}_n, T)$ is normal when $N \geq 5$ and, as a consequence, its full automorphism group is isomorphic to $\mathbb{S}_n \times \mathcal{D}_{2n}$. The Cayley graph $\Aut(\Cay(\mathbb{S}_n, T))$ with $T$ induced by $C_{N}$ is in the literature known as the modified bubble-sort graph, see e.g., \cite{LakshmivarahanEtAl}.

The first step in studying the orbit and orbital structure of $X$ under $G_X$ is the identification of~$B_\tau$. It can be proven that $B_\tau$ is trivial under a very mild condition. Recall that $c$ is the number of qubits in a connected component of size at least three in $(Q,U)$, see Definition~\ref{Def:fixingpattern}.  
\begin{theorem} \label{Thm:trivialCycle}
Suppose $c \geq 3$. Then $B_\tau$ is trivial for all $\tau \in \mathbb{S}_n$. 
\end{theorem}
\begin{proof}
Let $\tau \in \mathbb{S}_n$. If the gate graph $(Q,U)$ contains a connected component of size at least three, then the fixing pattern $\mathcal{F}$ contains at least three single-element sets, say $\{i\}, \{j\}$ and~$\{\ell \}$. Since~$B_\tau$ is the subgroup of $\mathcal{D}_{2n}$ that setwise stabilizes the sets $\tau^{-1}(S_1), \ldots, \tau^{-1}(S_l)$, it follows that any~$b \in B_\tau$ must pointwise fix $\tau^{-1}(i)$, $\tau^{-1}(j)$ and $\tau^{-1}(\ell)$. However, the only element in $\mathcal{D}_{2n}$ that fixes more than two elements is the identity element. Thus, $B_\tau$ is trivial.  
\end{proof}

Observe that the condition of Theorem~\ref{Thm:trivialCycle} is not restrictive. Namely, when $c < 3$, the quantum circuit does not have overlapping quantum gates. This implies that a trivial qubit assignment is possible without the need of any inserted SWAP gates, making the NNCP instance trivial.

\subsection[Biclique graph $K_{M,N}$ and star graph $K_{1,N}$]{Biclique graph \boldmath $K_{M,N}$ and star graph \boldmath $K_{1,N}$} \label{Subsec:Biclique}
The biclique graph $K_{M,N}$ is given by ${L = [M] \sqcup [N]}$ and~${E = \{\{i,j\} \, : \, \, i \in [M], j \in [N]\}}$. The induced partition of the vertex set $L$ we denote by the sets $L_M$ and $L_N$. We assume here that~$M < N$. Any independent setwise permutation of vertices in $L_M$ and $L_N$ forms an automorphism of the graph, hence $\Aut(\Coup(E)) \cong \mathbb{S}_M \times \mathbb{S}_N$. The corresponding Cayley graph $\Cay(\mathbb{S}_n, T)$ is in the literature known as the generalized star graph, see e.g., \cite{GanesanOverview}. With respect to the structure of the subgroups~$B_\tau$, we prove the following result.
\begin{theorem} \label{Thm:Biclique}
    Let $\tau \in \mathbb{S}_n$. Let $\mathcal{F}'$ denote the fixing pattern obtained from $\mathcal{F}$ by replacing any~$S \in \mathcal{F}$ with $|S| \geq 2$ by  
    \begin{align*}
        S_1 = \{ i \in S \, : \, \, \tau^{-1}(i) \in L_M\} \quad \text{and} \quad S_2 = \{ i \in S \, : \, \, \tau^{-1}(i) \in L_N\}.
    \end{align*}
    Then $B_\tau \cong \mathbb{S}_n(\mathcal{F}')$. Moreover, let $\hat{p}$ denote the number of pairs $\{i,j\}$ for which $\tau^{-1}(i) \in L_M$ and~$\tau^{-1}(j) \in L_N$, let $f_1$ denote the number of elements in the free set that are mapped to~$L_M$ by~$\tau^{-1}$, and let $f_2 = f - f_1$. Then,
    \begin{align*}
        |B_\tau | = 2^{p - \hat{p}} f_1! f_2!.
    \end{align*}
\end{theorem}
\begin{proof}
    Let $\mathcal{G} := \{\tau^{-1}(S_1), \ldots, \tau^{-1}(S_l)\}$ be the partition of $[n]$ defined by $\mathcal{F}$ shifted over~$\tau^{-1}$. Then,~$B_\tau$ is the subgroup of $\mathbb{S}_n(\mathcal{G})$ which are also automorphisms of~$K_{M,N}$. Since any automorphism of $K_{M,N}$ setwise fix the vertices in $L_M$ and $L_N$, we obtain $B_\tau$ by splitting each set of $\mathcal{G}$ into its subset in $L_M$ and its subset in $L_N$, leading to the partition $\mathcal{G}'$. The partition $\mathcal{F}'$ is exactly~$\mathcal{G}'$ shifted over $\tau$, leading to $B_\tau \cong \mathbb{S}_n(\mathcal{F}')$. The second part of the statement follows from counting the number of elements in $\mathbb{S}_n(\mathcal{F}')$. 
\end{proof}

The special case where $M = 1$ is commonly known as the star graph $K_{1,N}$. Its induced Cayley graph is studied in~\cite{LakshmivarahanEtAl}. Since we consider this coupling graph extensively in the numerical results of Section~\ref{Sec:CompResults}, we add this case explicitly to Table~\ref{Tab:summarycoupling}.

\section{Computational results}
\label{Sec:CompResults}
In this section we test our symmetry-reduced NNCP formulation on a set of instances for the coupling graphs discussed in Section~\ref{Sec:SpecialCoupling}. We compare the result against the nonreduced shortest path formulation~\eqref{Prob:SPP}. 

We first describe the design of our numerical tests in Section~\ref{Subsec:DesignExperiments}, after which we discuss the results on real and random instances in Section~\ref{Subsec:ResultsRevlib} and~\ref{Subsec:ResultsRandom}, respectively. 

\subsection{Design of computational experiments} \label{Subsec:DesignExperiments}
For our experiments we consider both realistic as well as randomly generated quantum circuits on different coupling graphs. As described in Section~\ref{Sec:NNCP}, we are justified to make two assumptions on the quantum circuits under consideration, imposing a preprocessing strategy in case these assumptions are not met: 
\begin{enumerate}
    \item Single-qubit gates can be ignored for the NNCP, since these do always comply with the adjacent interaction constraints. Without loss of generality, we therefore remove the single-qubit gates from the circuits in the preprocessing phase. 
    \item All gates that act on more than two qubits are decomposed into gates that act on one or two qubits. Nielsen and Chuang~\cite{NielsenChuang} have shown that these gates are universal, and that any quantum gate can therefore be decomposed into one- or two-qubit gates.
    There exists a large number of different decomposition strategies, leading to possibly different quantum gates (with the same functionality, however). As the choice of the optimal decomposition strategy is outside the scope of our research, we always choose the same strategy, namely the method considered in \cite{Mulderij, MulderijEtAl}. 
\end{enumerate}
The quantum circuits that we consider in this paper consist of general one- or two-qubit gates, multiple-control Toffoli gates up to size five, Peres gates and multiple-control Fredkin gates up to size four. In Appendix~\ref{App:Gates} we consider the decomposition of these gates into one- or two qubit gates, following the approach from~\cite{Mulderij, MulderijEtAl}. After that, we remove all single-qubit gates from the circuit. The preprocessed circuit that remains, will be the quantum circuit $\Gamma = (Q,C)$ that we take as an input to our approach. 

We consider the following two instance classes: 
\begin{itemize}
    \item \textbf{Real data}: Realistic quantum circuits that we consider are obtained from the RevLib library~\cite{Revlib}. This dataset consists of quantum gates of (well-known) reversible functions considered in the quantum literature. Due to the assumptions of the preprocessing phase, we only consider instances consisting of the above-mentioned gates, see Appendix~\ref{App:Gates} for an overview. This leads to a set of 84 instances with $n \in \{5, \ldots, 17\}$ and $m \in \{7,\ldots, 112\}$. 
    
    \item \textbf{Random data}: We also consider synthetic quantum gates in order to also test our approach on circuits consisting of more qubits and gates. We apply two strategies: 
    \begin{itemize}
        \item \textit{Random Class I}: Given $n$ and $m$, we create a random circuit on $n$ qubits consisting of $m$ two-qubit gates. Each gate acts on two distinct qubits that are chosen uniformly at random from $[n]$, independently from the other gates. 

        For each combination of $n \in \{20, 30, \ldots, 100\}$ and $m = \{2n, 4n\}$, we consider 5 randomly generated instances of this type. This leads to a test set of 90 instances. 
        
        \item \textit{Random Class II}: Given $n$ and $m$, we first create a random circuit on $n$ qubits consisting of $m$ gates selected from: Toffoli gate (on 3, 4 or 5 qubits), Fredkin gate (on 3 or 4 qubits), Peres gate, or a general two-qubit gate. The latter class includes the CNOT, SWAP and controlled-$V$ or -$V^\dag$ gates. Each gate type is selected with equal probability, and the qubits on which each gate acts are chosen uniformly at random from $[n]$. After that, we apply the preprocessing approach explained above to convert each circuit to an equivalent circuit of two-qubit gates. This leads to quantum gates with possibly more realistic patterns than Random Class I. 

        For each combination of $n \in \{20, 30, \ldots, 100\}$ and $m \in \{n, 2n\}$, we consider 5 randomly generated instances of this type, leading to a test set of 90 instances. 
        After the preprocessing step, the values of $m$ increase and are within~${117 \leq m \leq 1872}$. 
    \end{itemize}
\end{itemize}
We solve the NNCP for each quantum circuit on the following coupling graphs: 
\begin{itemize}
    \item \textbf{Cycle graph}: The undirected cycle $C_N$ on $N = n$ qubits, see~Section~\ref{Subsec:Cycle}. 

    \item \textbf{Star graph}: The star graph $K_{1,N}$ with $N = n-1$, see Section~\ref{Subsec:Biclique}. 

    \item \textbf{Biclique graph}: The biclique graph $K_{M,N}$ with $M = 2$ and $N = n - 2$, see Section~\ref{Subsec:Biclique}. 
\end{itemize}

For each combination of quantum circuit and coupling graph, we solve the unreduced LP-formulation~\eqref{Prob:SPP} and the reduced scaled formulation~\eqref{Prob:RSPP_scale}. The unreduced formulation is implemented by a full construction of the graph $X = (V,A)$. For the reduced formulation, we follow the steps of Algorithm~\ref{Alg:FinalAlgorithm}, which is
based on the results from Sections~\ref{Sec:ExploitSym}-\ref{Sec:SpecialCoupling}. Preliminary experiments have shown that the performance between the nonscaled and scaled formulations, \eqref{Prob:RSPP} and \eqref{Prob:RSPP_scale}, respectively, is very similar. However, as the size of the coefficients in~\eqref{Prob:RSPP} grows with the order of the automorphism group of $\Coup(E)$, the LP formulation becomes unstable for the star and biclique graphs when $n \geq 11$ or $n \geq 12$, respectively. Therefore, we only use the more robust scaled version~\eqref{Prob:RSPP_scale} in our tests. 

Experiments are carried out on a PC with an Intel(R)
Core(TM) i7-8700 CPU, 3.20GHz and~8~GB RAM. Our methods are implemented in Julia 1.8.4 using JuMP v1.6.0~\cite{JuMP} to model the mathematical optimization problems. We use the LP solver of Mosek 10.0~\cite{Mosek100} to solve our models in the default configuration. The maximum computation time (including the construction time of the program) is set to 2 hours.

\subsection{Results on RevLib instances} \label{Subsec:ResultsRevlib}
Table~\ref{Tab:RevlibCycle},~\ref{Tab:RevlibStar} and~\ref{Tab:RevlibBiclique} show the results for the RevLib instances on the cycle, star and biclique graph, respectively. The columns `$n$' and `$m$' show the number of qubits and quantum gates in the preprocessed circuit. The column `OPT' shows the optimal value of the NNCP instance, i.e., the minimum number of inserted SWAP gates in order to make the quantum circuit compliant. The columns `time $(RSPP')$' and `time $(SPP)$' show the computation time (i.e., clocktimes) in seconds to  solve the reduced model~\eqref{Prob:RSPP_scale} and the base model~\eqref{Prob:SPP}, respectively. The values are rounded to three decimals. The columns `\#var $(RSPP')$' and `\#const $(RSPP')$' denote the total number of variables and constraints after the symmetry reduction. The column `reduction \#var (\%)' shows the relative reduction in the number of variables compared to the base model, i.e., $\frac{\text{\#var $(SPP)$} - \text{\#var $(RSPP')$}}{\text{\#var $(SPP)$}} \cdot 100\%$, rounded to two decimal places. The final column shows the same relative reduction for the number of constraints.
Whenever a given instance is not solvable (including construction) within the time limit of 2 hours, or whenever an instance leads to a shortage of memory, we report a~`-' in the tables. 

It turned out that the 62 instances with $n = 5$ are very easy to compute for both models~$(SPP)$ and~$(RSPP')$. For that reason, results on these instances are not depicted in Tables~\ref{Tab:RevlibCycle}, \ref{Tab:RevlibStar} and~\ref{Tab:RevlibBiclique}. The total relative reduction in the number of variables and constraints on the instances with $n = 5$ turns out to be at least 90\% and $89.8\%$, respectively.

For the cycle graph, one can clearly see that the bottleneck in the computational limit is the number of qubits $n$. It follows from Table~\ref{Tab:RevlibCycle} that our approach is able to solve instances up to roughly~8 qubits, while the base model can only solve instances up to 7 qubits. The total computation time of~\eqref{Prob:RSPP_scale} is often negligible and below 30 seconds for the instances that can be solved. For the base model the total computation times are significantly higher, with a maximum difference of about a factor 100. This can be explained by the large reduction in the total number of variables and constraints, which are both above~$91\%$ for all instances. 

For the star graph, we conclude from Table~\ref{Tab:RevlibStar} that the reduced model can easily handle the full set of RevLib instances. The computation times are negligible for almost all instances and always below~0.2 seconds. This can be explained by the order of $\Aut(\Coup(E))$ being factorial in~$n$, implying that the model~\eqref{Prob:RSPP_scale} scales linearly in both $m$ and $n$. The relative reductions with the base model are enormous, i.e., above 99\% in terms of the number of variables and constraints on all instances. 
 For the unreduced model, the largest instance we can solve has $n = 8$ and $m = 36$, which could not be solved on the cycle coupling graph. This can be explained by the fact that the star graph on $n$ vertices has one edge less than the cycle graph on $n$ vertices, resulting in the Cayley graph containing significantly fewer edges. The computational frontier, however, is reached already at the next instance, for which the base model runs into memory issues.

Finally, the results on the biclique coupling graph look very similar to the results on the star graph, see Table~\ref{Tab:RevlibBiclique}. The total relative reduction between the models is extremely large, leading to all instances to be solvable within 0.25 seconds using~\eqref{Prob:RSPP_scale}. The computation times are slightly larger than for the star graph, which can be explained by the smaller size of the automorphism group of the biclique. For the unreduced formulation we can only solve up to $n = 7$, while the reduction in computation time for the largest instance solvable by~\eqref{Prob:SPP} is about a factor 4700.

{\color{black} For all the RevLib instances, the gate graph $(Q,U)$ turns out to consist of a single connected component, and thus, the group $\mathbb{S}_n(\mathcal{F})$ is trivial. This implies that all of the symmetry reduction results from the symmetries of the coupling graph.   }

\begin{table}[h]
\setlength{\tabcolsep}{3pt}
\centering
\scriptsize
\begin{tabular}{@{}lrrrSSrrSS@{}}
\toprule
\textbf{Benchmark}  & \multicolumn{1}{c}{\textbf{\boldmath $n$}} & \multicolumn{1}{c}{\textbf{\boldmath $m$}} & \multicolumn{1}{c}{\textbf{OPT}} & \multicolumn{1}{c}{\textbf{\begin{tabular}[c]{@{}c@{}}time\\ \boldmath $(RSPP')$\end{tabular}}} & \multicolumn{1}{c}{\textbf{\begin{tabular}[c]{@{}c@{}}time\\ \boldmath$(SPP)$\end{tabular}}} & \multicolumn{1}{c}{\textbf{\begin{tabular}[c]{@{}c@{}}\#var \\ \boldmath$(RSPP')$\end{tabular}}} &  \multicolumn{1}{c}{\textbf{\begin{tabular}[c]{@{}c@{}}\#const \\ \boldmath $(RSPP')$\end{tabular}}}  & \multicolumn{1}{c}{\textbf{\begin{tabular}[c]{@{}c@{}}reduction \\ \#var (\%)\end{tabular}}} & \multicolumn{1}{c}{\textbf{\begin{tabular}[c]{@{}c@{}}reduction\\ \#const (\%)\end{tabular}}} \\ \midrule
graycode6\_47       & 6                              & 5                              & 0                                & 0.000                                                                                 & 0.172                                                                               & 1980                                                                                    & 3602                                                                                      & 91.67                                                                                        & 91.62                                                                                         \\
graycode6\_48       & 6                              & 5                              & 0                                & 0.016                                                                                 & 0.172                                                                               & 1980                                                                                    & 3602                                                                                      & 91.67                                                                                        & 91.62                                                                                         \\
decod24-enable\_124 & 6                              & 21                             & 5                                & 0.047                                                                                 & 0.937                                                                               & 8124                                                                                    & 15122                                                                                     & 91.67                                                                                        & 91.65                                                                                         \\
decod24-enable\_125 & 6                              & 21                             & 4                                & 0.047                                                                                 & 0.906                                                                               & 8124                                                                                    & 15122                                                                                     & 91.67                                                                                        & 91.65                                                                                         \\
decod24-bdd\_294    & 6                              & 24                             & 8                                & 0.062                                                                                 & 1.203                                                                               & 9276                                                                                    & 17282                                                                                     & 91.67                                                                                        & 91.66                                                                                         \\
mod5adder\_129      & 6                              & 71                             & 27                               & 0.157                                                                                 & 4.563                                                                               & 27324                                                                                   & 51122                                                                                     & 91.67                                                                                        & 91.66                                                                                         \\
mod5adder\_128      & 6                              & 77                             & 32                               & 0.172                                                                                 & 4.250                                                                               & 29628                                                                                   & 55442                                                                                     & 91.67                                                                                        & 91.66                                                                                         \\
decod24-enable\_126 & 6                              & 86                             & 34                               & 0.188                                                                                 & 5.500                                                                               & 33084                                                                                   & 61922                                                                                     & 91.67                                                                                        & 91.66                                                                                         \\
xor5\_254           & 6                              & 5                              & 3                                & 0.016                                                                                 & 0.188                                                                               & 1980                                                                                    & 3602                                                                                      & 91.67                                                                                        & 91.62                                                                                         \\
ex1\_226            & 6                              & 5                              & 3                                & 0.016                                                                                 & 0.187                                                                               & 1980                                                                                    & 3602                                                                                      & 91.67                                                                                        & 91.62                                                                                         \\
4mod5-bdd\_287      & 7                              & 23                             & 8                                & 0.469                                                                                 & 36.203                                                                              & 61080                                                                                   & 115922                                                                                    & 92.86                                                                                        & 92.86                                                                                         \\
alu-bdd\_288        & 7                              & 28                             & 7                                & 0.641                                                                                 & 51.031                                                                              & 74280                                                                                   & 141122                                                                                    & 92.86                                                                                        & 92.86                                                                                         \\
ham7\_106           & 7                              & 49                             & 20                               & 1.172                                                                                 & 91.672                                                                              & 129720                                                                                  & 246962                                                                                    & 92.86                                                                                        & 92.86                                                                                         \\
ham7\_105           & 7                              & 65                             & 32                               & 1.485                                                                                 & 135.625                                                                             & 171960                                                                                  & 327602                                                                                    & 92.86                                                                                        & 92.86                                                                                         \\
ham7\_104           & 7                              & 83                             & 38                               & 1.984                                                                                 & 181.734                                                                             & 219480                                                                                  & 418322                                                                                    & 92.86                                                                                        & 92.86                                                                                         \\
rd53\_137           & 7                              & 66                             & 33                               & 3.750                                                                                 & 146.811                                                                             & 174600                                                                                  & 23762                                                                                     & 92.24                                                                                        & 92.86                                                                                         \\
rd53\_139           & 8                              & 36                             & 14                               & 22.672                                                                                &   {-}                                                                                  & 754200                                                                                  & 90722                                                                                     & 93.75                                                                                        & 93.75                                                                                         \\
rd53\_138           & 8                              & 44                             & 20                               & 26.266                                                                                &    {-}                                                                                 & 921240                                                                                  & 110882                                                                                    & 93.75                                                                                        & 93.75                                                                                         \\
mini\_alu\_305      & 10                             & 57                             &                      {-}             &            {-}                                                                           &         {-}                                                                             &     {-}                                                                                     &   {-}                                                                                         &  {-}                                                                                             &    {-}                                                                                            \\
sys6-v0\_144        & 10                             & 62                             &                         {-}          &     {-}                                                                                   &    {-}                                                                                  &  {-}                                                                                        &  {-}                                                                                          &          {-}                                                                                     &    {-}                                                                                            \\
rd73\_141           & 10                             & 64                             &                    {-}               &   {-}                                                                                     &  {-}                                                                                    &  {-}                                                                                        &   {-}                                                                                         &         {-}                                                                                      &   {-}                                                                                             \\
parity\_247         & 17                             & 16                             &                        {-}           &   {-}                                                                                     &  {-}                                                                                    & {-}                                                                                         &  {-}                                                                                          &        {-}                                                                                       &      {-}                                                                                          \\ \bottomrule
\end{tabular}
\caption{Results on the `RevLib' instances on the cyclic coupling graph. We compare the performance of the base model $(SPP)$ with the reduced model $(RSPP)$. Times are clocktimes given in seconds. \label{Tab:RevlibCycle} }
\end{table}

\begin{table}[H]
\setlength{\tabcolsep}{3pt}
\centering
\scriptsize
\begin{tabular}{@{}lrrrSSrrSS@{}}
\toprule
\textbf{Benchmark}  & \multicolumn{1}{c}{\textbf{\boldmath $n$}} & \multicolumn{1}{c}{\textbf{\boldmath $m$}} & \multicolumn{1}{c}{\textbf{OPT}} & \multicolumn{1}{c}{\textbf{\begin{tabular}[c]{@{}c@{}}time\\ \boldmath $(RSPP')$\end{tabular}}} & \multicolumn{1}{c}{\textbf{\begin{tabular}[c]{@{}c@{}}time\\ \boldmath$(SPP)$\end{tabular}}} & \multicolumn{1}{c}{\textbf{\begin{tabular}[c]{@{}c@{}}\#var \\ \boldmath$(RSPP')$\end{tabular}}} & \multicolumn{1}{c}{\textbf{\begin{tabular}[c]{@{}c@{}}\#const \\ \boldmath $(RSPP')$\end{tabular}}} & \multicolumn{1}{c}{\textbf{\begin{tabular}[c]{@{}c@{}}reduction \\ \#var (\%)\end{tabular}}} & \multicolumn{1}{c}{\textbf{\begin{tabular}[c]{@{}c@{}}reduction\\ \#const (\%)\end{tabular}}} \\ \midrule
graycode6\_47       & 6                              & 5                              & 2                                & 0.000                                                                                 & 0.125                                                                               & 166                                                                                     & 32                                                                                        & 99.17                                                                                        & 99.11                                                                                         \\
graycode6\_48       & 6                              & 5                              & 2                                & 0.000                                                                                 & 0.125                                                                               & 166                                                                                     & 32                                                                                        & 99.17                                                                                        & 99.11                                                                                         \\
decod24-enable\_124 & 6                              & 21                             & 4                                & 0.016                                                                                 & 0.609                                                                               & 678                                                                                     & 128                                                                                       & 99.17                                                                                        & 99.15                                                                                         \\
decod24-enable\_125 & 6                              & 21                             & 5                                & 0.000                                                                                 & 0.594                                                                               & 678                                                                                     & 128                                                                                       & 99.17                                                                                        & 99.15                                                                                         \\
decod24-bdd\_294    & 6                              & 24                             & 8                                & 0.000                                                                                 & 0.672                                                                               & 774                                                                                     & 146                                                                                       & 99.17                                                                                        & 99.16                                                                                         \\
mod5adder\_129      & 6                              & 71                             & 19                               & 0.000                                                                                 & 2.343                                                                               & 2278                                                                                    & 428                                                                                       & 99.17                                                                                        & 99.16                                                                                         \\
mod5adder\_128      & 6                              & 77                             & 18                               & 0.000                                                                                 & 2.953                                                                               & 2470                                                                                    & 464                                                                                       & 99.17                                                                                        & 99.16                                                                                         \\
decod24-enable\_126 & 6                              & 86                             & 19                               & 0.016                                                                                 & 2.718                                                                               & 2758                                                                                    & 518                                                                                       & 99.17                                                                                        & 99.16                                                                                         \\
xor5\_254           & 6                              & 5                              & 0                                & 0.000                                                                                 & 0.109                                                                               & 166                                                                                     & 32                                                                                        & 99.17                                                                                        & 99.11                                                                                         \\
ex1\_226            & 6                              & 5                              & 0                                & 0.000                                                                                 & 0.125                                                                               & 166                                                                                     & 32                                                                                        & 99.17                                                                                        & 99.11                                                                                         \\
4mod5-bdd\_287      & 7                              & 23                             & 5                                & 0.000                                                                                 & 14.875                                                                              & 1019                                                                                    & 163                                                                                       & 99.86                                                                                        & 99.86                                                                                         \\
alu-bdd\_288        & 7                              & 28                             & 11                               & 0.000                                                                                 & 16.141                                                                              & 1239                                                                                    & 198                                                                                       & 99.86                                                                                        & 99.86                                                                                         \\
ham7\_106           & 7                              & 49                             & 20                               & 0.015                                                                                 & 28.328                                                                              & 2163                                                                                    & 345                                                                                       & 99.86                                                                                        & 99.86                                                                                         \\
ham7\_105           & 7                              & 65                             & 18                               & 0.000                                                                                 & 36.157                                                                              & 2867                                                                                    & 457                                                                                       & 99.86                                                                                        & 99.86                                                                                         \\
ham7\_104           & 7                              & 83                             & 18                               & 0.015                                                                                 & 57.516                                                                              & 3659                                                                                    & 583                                                                                       & 99.86                                                                                        & 99.86                                                                                         \\
rd53\_137           & 7                              & 66                             & 10                               & 0.000                                                                                 & 38.521                                                                              & 2911                                                                                    & 464                                                                                       & 99.86                                                                                        & 99.86                                                                                         \\
rd53\_139           & 8                              & 36                             & 15                               & 0.047                                                                                 & 7031.828                                                                            & 2096                                                                                    & 290                                                                                       & 99.98                                                                                        & 99.98                                                                                         \\
rd53\_138           & 8                              & 44                             & 12                               & 0.000                                                                                 & {-}                                                                                   & 2560                                                                                    & 354                                                                                       & 99.98                                                                                        & 99.98                                                                                         \\
mini\_alu\_305      & 10                             & 57                             & 16                               & 0.016                                                                                 & {-}                                                                                   & 5254                                                                                    & 572                                                                                       & 100.00                                                                                       & 100.00                                                                                        \\
sys6-v0\_144        & 10                             & 62                             & 26                               & 0.015                                                                                 & {-}                                                                                   & 5714                                                                                    & 622                                                                                       & 100.00                                                                                       & 100.00                                                                                        \\
rd73\_141           & 10                             & 64                             & 27                               & 0.000                                                                                 & {-}                                                                                   & 5898                                                                                    & 642                                                                                       & 100.00                                                                                       & 100.00                                                                                        \\
parity\_247         & 17                             & 16                             & 0                                & 0.000                                                                                 & {-}                                                                                   & 4401                                                                                    & 274                                                                                       & 100.00                                                                                       & 100.00                                                                                        \\ \bottomrule
\end{tabular}
\caption{Results on the `RevLib' instances on the star coupling graph. We compare the performance of the base model $(SPP)$ with the reduced model $(RSPP)$. Times are clocktimes given in seconds. \label{Tab:RevlibStar}}
\end{table}

\begin{table}[H]
\centering
\setlength{\tabcolsep}{3pt}
\scriptsize
\begin{tabular}{@{}lrrrSSrrSS@{}}
\toprule
\textbf{Benchmark}  & \multicolumn{1}{c}{\textbf{\boldmath $n$}} & \multicolumn{1}{c}{\textbf{\boldmath $m$}} & \multicolumn{1}{c}{\textbf{OPT}} & \multicolumn{1}{c}{\textbf{\begin{tabular}[c]{@{}c@{}}time\\ \boldmath $(RSPP')$\end{tabular}}} & \multicolumn{1}{c}{\textbf{\begin{tabular}[c]{@{}c@{}}time\\ \boldmath$(SPP)$\end{tabular}}} & \multicolumn{1}{c}{\textbf{\begin{tabular}[c]{@{}c@{}}\#var \\ \boldmath$(RSPP')$\end{tabular}}} & \multicolumn{1}{c}{\textbf{\begin{tabular}[c]{@{}c@{}}\#const \\ \boldmath $(RSPP')$\end{tabular}}} & \multicolumn{1}{c}{\textbf{\begin{tabular}[c]{@{}c@{}}reduction \\ \#var (\%)\end{tabular}}} & \multicolumn{1}{c}{\textbf{\begin{tabular}[c]{@{}c@{}}reduction\\ \#const (\%)\end{tabular}}} \\ \midrule
graycode6\_47       & 6                              & 5                              & 1                                & 0.015                                                                                 & 0.250                                                                               & 655                                                                                     & 77                                                                                        & 97.92                                                                                        & 97.86                                                                                         \\
graycode6\_48       & 6                              & 5                              & 1                                & 0.000                                                                                 & 0.235                                                                               & 655                                                                                     & 77                                                                                        & 97.92                                                                                        & 97.86                                                                                         \\
decod24-enable\_124 & 6                              & 21                             & 4                                & 0.016                                                                                 & 1.500                                                                               & 2703                                                                                    & 317                                                                                       & 97.92                                                                                        & 97.90                                                                                         \\
decod24-enable\_125 & 6                              & 21                             & 4                                & 0.000                                                                                 & 1.297                                                                               & 2703                                                                                    & 317                                                                                       & 97.92                                                                                        & 97.90                                                                                         \\
decod24-bdd\_294    & 6                              & 24                             & 5                                & 0.015                                                                                 & 1.485                                                                               & 3087                                                                                    & 362                                                                                       & 97.92                                                                                        & 97.91                                                                                         \\
mod5adder\_129      & 6                              & 71                             & 15                               & 0.032                                                                                 & 5.031                                                                               & 9103                                                                                    & 1067                                                                                      & 97.92                                                                                        & 97.91                                                                                         \\
mod5adder\_128      & 6                              & 77                             & 14                               & 0.031                                                                                 & 5.016                                                                               & 9871                                                                                    & 1157                                                                                      & 97.92                                                                                        & 97.91                                                                                         \\
decod24-enable\_126 & 6                              & 86                             & 16                               & 0.031                                                                                 & 5.844                                                                               & 11023                                                                                   & 1292                                                                                      & 97.92                                                                                        & 97.91                                                                                         \\
xor5\_254           & 6                              & 5                              & 1                                & 0.000                                                                                 & 0.266                                                                               & 655                                                                                     & 77                                                                                        & 97.92                                                                                        & 97.86                                                                                         \\
ex1\_226            & 6                              & 5                              & 1                                & 0.000                                                                                 & 0.265                                                                               & 655                                                                                     & 77                                                                                        & 97.92                                                                                        & 97.86                                                                                         \\
4mod5-bdd\_287      & 7                              & 23                             & 4                                & 0.016                                                                                 & 72.110                                                                              & 5081                                                                                    & 485                                                                                       & 99.58                                                                                        & 99.58                                                                                         \\
alu-bdd\_288        & 7                              & 28                             & 5                                & 0.016                                                                                 & 83.844                                                                              & 6181                                                                                    & 590                                                                                       & 99.58                                                                                        & 99.58                                                                                         \\
ham7\_106           & 7                              & 49                             & 8                                & 0.031                                                                                 & 143.265                                                                             & 10801                                                                                   & 1031                                                                                      & 99.58                                                                                        & 99.58                                                                                         \\
ham7\_105           & 7                              & 65                             & 14                               & 0.031                                                                                 & 217.359                                                                             & 14321                                                                                   & 1367                                                                                      & 99.58                                                                                        & 99.58                                                                                         \\
ham7\_104           & 7                              & 83                             & 8                                & 0.078                                                                                 & 282.453                                                                             & 18281                                                                                   & 1745                                                                                      & 99.58                                                                                        & 99.58                                                                                         \\
rd53\_137           & 7                              & 66                             & 10                               & 0.047                                                                                 & 223.981                                                                             & 14541                                                                                   & 1388                                                                                      & 98.14                                                                                        & 99.58                                                                                         \\
rd53\_139           & 8                              & 36                             & 8                                & 0.063                                                                                 & {-}                                                                                   & 12556                                                                                   & 1010                                                                                      & 99.93                                                                                        & 99.93                                                                                         \\
rd53\_138           & 8                              & 44                             & 10                               & 0.062                                                                                 & {-}                                                                                    & 15340                                                                                   & 1234                                                                                      & 99.94                                                                                        & 99.94                                                                                         \\
mini\_alu\_305      & 10                             & 57                             & 14                               & 0.218                                                                                 & {-}                                                                                    & 41997                                                                                   & 2567                                                                                      & 100.00                                                                                       & 100.00                                                                                        \\
sys6-v0\_144        & 10                             & 62                             & 13                               & 0.141                                                                                 & {-}                                                                                    & 45677                                                                                   & 2792                                                                                      & 100.00                                                                                       & 100.00                                                                                        \\
rd73\_141           & 10                             & 64                             & 14                               & 0.095                                                                                 & {-}                                                                                    & 47149                                                                                   & 2882                                                                                      & 100.00                                                                                       & 100.00                                                                                        \\
parity\_247         & 17                             & 16                             & 1                                & 0.108                                                                                 & {-}                                                                                    & 65896                                                                                   & 2178                                                                                      & 100.00                                                                                       & 100.00                                                                                        \\ \bottomrule
\end{tabular}
\caption{Results on the `RevLib' instances on the biclique coupling graph. We compare the performance of the base model $(SPP)$ with the reduced model $(RSPP)$. Times are clocktimes given in seconds. \label{Tab:RevlibBiclique}}
\end{table}

\subsection{Results on random instances} \label{Subsec:ResultsRandom} 
From Table~\ref{Tab:RevlibStar} and~\ref{Tab:RevlibBiclique} we observe that the RevLib instances can be easily solved by our symmetry reduced formulation. To test the performance on larger instances, we consider the random data set, consisting of quantum circuits with up to 100 qubits and 1837 quantum gates. For the cycle coupling graph, we have seen that we could only solve instances up to~$n = 8$. Therefore, we do not include the cycle coupling graph anymore for the random data set. For the same reason, we do no longer consider the base model~\eqref{Prob:SPP}. 

Table~\ref{Tab:RandomStar} and~\ref{Tab:RandomBiclique} show the performance of our symmetry-reduced NNCP formulation on Random Class I and Random Class II for both the star and biclique coupling graph. Next to the total {\color{black} average} solution time, which is given in the column `time $(RSPP)$', we show in the column `time constr.' the {\color{black} average} time that is required to construct the LP-instance. {\color{black}Also, we add the column `$|\mathbb{S}_n(\mathcal{F})|$', which shows the average order of the subgroup $\mathbb{S}_n(\mathcal{F})$ (recall that this quantity only depends on the quantum circuit and not on the coupling graph).}  Each row in the tables corresponds to the average value over 5 randomly generated instances of that type. In Figure~\ref{Fig:Random} we plot the averaged total computation time, i.e., construction and solution time, compared to $n$ and $m$ for both coupling graphs and random classes.

For the star coupling graph, we see that we can easily solve all instances from Random Class~I within on average 25 seconds, while at most 90 seconds are needed to construct the model. For Random Class II, we can solve up to $n = 100$, however, when $m$ is too large, the PC runs out of memory. For the biclique coupling graph on Random Class I, we can solve instances up to $n = 40$ within the time span of 2 hours, whereas for Random Class II the instances with large $m$ cannot be solved anymore.

The sum of solution and construction times on the biclique graphs is significantly higher than on the star graphs, see Figure~\ref{Fig:Random}. The tables reveal that the solution times on the former are an order of magnitude 2 higher. This can be explained by the difference in the order of $\Aut(\Coup(E))$, as explained in Section~\ref{Subsec:ResultsRevlib}. The construction times, however, heavily deviate among the instances on the star and the biclique coupling graph. Indeed, the smaller automorphism group increases the number of orbits. For each of these orbits, one needs to evaluate the orbitals of the group action of~$B_{\tau}$ on $E$. Hence, the negative effects of having a smaller number of symmetries and a larger number of edges, strengthen one another and result in large construction times when $n$ and $m$ increase.
 
\begin{figure}[H]
  \centering
  \includegraphics[scale=0.55]{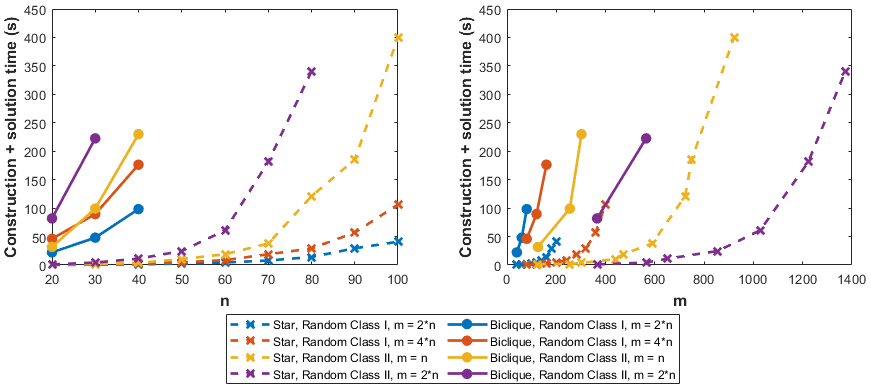}
  \caption{Overview of total average computation times (construction + solution time) of random instances with respect to $n$ and $m$. Each data point displays the average over 5 randomly generated instances of that type. \label{Fig:Random}}
\end{figure}

When comparing Random Class I and II, we do not observe significant structural differences. It seems to be primarily the magnitude of $n$ and $m$ that influences the complexity of the instance. Due to the construction, $m$ grows more rapidly with respect to $n$ for Random Class II than for Random Class I. This effect can be observed from Figure~\ref{Fig:Random}, where we observe that for fixed $n$, an instance from Random Class II on average requires more computation time. 

{\color{black}We observe for both Random Class I and II that the average order of the subgroup $\mathbb{S}_n(\mathcal{F})$ is quite small, although the subgroup is often not trivial. This implies that part of the symmetry reduction is due to the underlying relationship between the quantum gates in the circuit. However, when comparing these average orders to the order of $\Aut(\Coup(E))$ (which are $(n-1)!$ and $2(n-2)!$ for the star and biclique graph, respectively), still the vast majority of the reduction is due to the symmetries in the coupling graph. }

The largest quantum circuit that we can successfully solve contains 100 qubits and 1047 quantum gates. Observe that the unreduced model of this instance would embrace subgraphs of $100!$ vertices, hence solving this model is infeasible. 

\begin{table}[H]
\centering
\scriptsize
\setlength{\tabcolsep}{5pt}
\begin{tabular}{@{}rrSSSScrrSSSS@{}}
\cmidrule(r){1-6} \cmidrule(l){8-13}
\multicolumn{6}{c}{\textbf{Random Class I}}                                                                                                                                       & \textbf{} & \multicolumn{6}{c}{\textbf{Random Class II}}                                                                                                                                     \\
\cmidrule(r){1-6} \cmidrule(l){8-13} 
\textbf{\boldmath $n$} & \textbf{\boldmath $m$} & \textbf{OPT} & \textbf{\begin{tabular}[c]{@{}c@{}}time\\ \boldmath $(RSPP')$\end{tabular}} & \textbf{\begin{tabular}[c]{@{}c@{}}time\\ constr.\end{tabular}} & {\color{black} \boldmath $|\mathbb{S}_n (\mathcal{F})|$} &  \textbf{} & \textbf{\boldmath $n$} & \textbf{\boldmath $m$} & \textbf{OPT} & \textbf{\begin{tabular}[c]{@{}c@{}}time\\ \boldmath $(RSPP')$\end{tabular}} & \textbf{\begin{tabular}[c]{@{}c@{}}time\\ constr.\end{tabular}} & {\color{black} \boldmath $|\mathbb{S}_n (\mathcal{F})|$}  \\ \cmidrule(r){1-6} \cmidrule(l){8-13} 
20           & 40           & 29.6           & 0.031                                                             & 0.088   & {\color{black}1.2}                                                        &           & 20           & 125.6        & 35.0           & 0.119                                                             & 1.425       & {\color{black}1.2}                                                    \\
20           & 80           & 65.2           & 0.056                                                             & 0.134 & {\color{black}1}                                                              &           & 20           & 365.6        & 80.8           & 0.334                                                             & 0.712        & {\color{black}1}                                                   \\
30           & 60           & 52.4           & 0.106                                                             & 0.274      & {\color{black}1}                                                        &           & 30           & 255          & 59.6           & 0.544                                                             & 1.189             & {\color{black}1.4}                                              \\
30           & 120          & 101.2          & 0.243                                                             & 0.551  & {\color{black}1}                                                          &           & 30           & 564.6        & 126.4          & 1.350                                                             & 2.940              & {\color{black}1}                                             \\
40           & 80           & 72.4           & 0.282                                                             & 0.820      & {\color{black}1}                                                      &           & 40           & 302.2        & 76.6           & 1.150                                                             & 2.482                   &  {\color{black}30.6}                                        \\
40           & 160          & 144.4          & 0.631                                                             & 1.556     & {\color{black}1}                                                       &           & 40           & 652.4        & 159.8          & 3.150                                                             & 8.081                   &  {\color{black}1}                                       \\
50           & 100          & 91.2           & 0.569                                                             & 1.971     & {\color{black}5.6}                                                       &           & 50           & 441.4        & 104.6          & 3.272                                                             & 7.223       & {\color{black}2.6}                                                    \\
50           & 200          & 184.6          & 1.312                                                             & 3.336 & {\color{black}1}                                                           &           & 50           & 854.8        & 203.8          & 8.091                                                             & 16.474          & {\color{black}1}                                                \\
60           & 120          & 112.8          & 1.025                                                             & 3.349    & {\color{black}2.4}                                                        &           & 60           & 471          & 122.8          & 5.737                                                             & 13.121                & {\color{black}27}                                          \\
60           & 240          & 222.4          & 2.447                                                             & 6.162 & {\color{black}1}                                                           &           & 60           & 1027.4       & 247.6          & 16.903                                                            & 44.061                   & {\color{black}1}                                       \\
70           & 140          & 132.0          & 1.769                                                             & 6.135      & {\color{black}3}                                                      &           & 70           & 589.4        & 145.2          & 10.838                                                            & 26.888                    & {\color{black}6}                                      \\
70           & 280          & 264.6          & 4.662                                                             & 14.194           & {\color{black}1}                                                &           & 70           & 1223         & 292.8          & 31.875                                                            & 149.451                                  & {\color{black}1}                       \\
80           & 160          & 151.0          & 3.313                                                             & 10.704  & {\color{black}1.4}                                                         &           & 80           & 722.4        & 171.6          & 23.634                                                            & 97.156                       & {\color{black}6.28}                                   \\
80           & 320          & 304.2          & 7.775                                                             & 21.305      & {\color{black}1}                                                     &           & 80           & 1372.8       & 333.6          & 32.600                                                            & 307.844                 & {\color{black}1}                                        \\
90           & 180          & 172.0          & 4.809                                                             & 24.524    & {\color{black}2.2}                                                       &           & 90           & 750          & 184.8          & 22.312                                                            & 162.915                       & {\color{black}2016.6}                                  \\
90           & 360          & 343.8          & 14.681                                                            & 42.293  &{\color{black}1}                                                          &           & 90           & 1602.8       & {-}        & {-}                                                             & {-}                    & {\color{black}1}                                       \\
100          & 200          & 191.8          & 9.106                                                             & 31.802   & {\color{black}3.2}                                                        &           & 100          & 921.2        & 218.8          & 36.966                                                            & 363.073            & {\color{black}2021.4}                                             \\
100          & 400          & 385.0          & 21.066                                                            & 85.295   & {\color{black}1}                                                        &           & 100          & 1709.6       & {-}        & {-}                                                             & {-}            & {\color{black}1}                                               \\ \cmidrule(r){1-6} \cmidrule(l){8-13} 
\end{tabular}
\caption{Results on the random instances on the star coupling graph. Each row shows the average values over 5 randomly generated instances. Times are clocktimes given in seconds. \label{Tab:RandomStar}}
\end{table}
\vspace{-0.5cm}
\begin{table}[H]
\centering
\scriptsize
\setlength{\tabcolsep}{5pt}
\begin{tabular}{@{}rrSSSScrrSSSS@{}}
\cmidrule(r){1-6} \cmidrule(l){8-13}
\multicolumn{6}{c}{\textbf{Random Type I}}                                                                                                                                       & \textbf{} & \multicolumn{6}{c}{\textbf{Random Type II}}                                                                                                                                      \\ \cmidrule(r){1-6} \cmidrule(l){8-13} 
\textbf{ \boldmath $n$} & \textbf{\boldmath$m$} & \textbf{OPT} & \textbf{\begin{tabular}[c]{@{}c@{}}time\\ \boldmath$(RSPP')$\end{tabular}} & \textbf{\begin{tabular}[c]{@{}c@{}}time\\ constr.\end{tabular}} & {\color{black} \boldmath $|\mathbb{S}_n (\mathcal{F})|$} & \textbf{} & \textbf{\boldmath$n$} & \textbf{\boldmath$m$} & \textbf{OPT} & \textbf{\begin{tabular}[c]{@{}c@{}}time\\ \boldmath$(RSPP')$\end{tabular}} & \textbf{\begin{tabular}[c]{@{}c@{}}time\\ constr.\end{tabular}} & {\color{black} \boldmath $|\mathbb{S}_n (\mathcal{F})|$} \\ \cmidrule(r){1-6} \cmidrule(l){8-13} 
20           & 40           & 20.6         & 1.588                                                             & 12.029     & {\color{black}1.2}                                                      &           & 20           & 125.6        & 27.6         & 4.188                                                             & 17.270         & {\color{black}1.2}                                                  \\
20           & 80           & 43.4         & 2.590                                                             & 14.217 & {\color{black}1}                                                          &           & 20           & 365.6        & 66.8         & 15.113                                                            & 30.584     & {\color{black}1}                                                      \\
30           & 60           & 38.4         & 9.675                                                             & 374.035   & {\color{black}1}                                                       &           & 30           & 255          & 50.0         & 49.175                                                            & 413.975    & {\color{black}1.4}                                                      \\
30           & 120          & 71.2         & 18.322                                                            & 390.226   & {\color{black}1}                                                       &           & 30           & 564.6        & 107.2        & 115.350                                                           & 829.710           & {\color{black}1}                                               \\
40           & 80           & 54.2         & 44.053                                                            & 3872.272    & {\color{black}1}                                                     &           & 40           & 302.2        & 61.7         & 168.276                                                           & 2800.738               & {\color{black}30.6}                                          \\
40           & 160          & 109.4        & 66.884                                                            & 3989.450  & {\color{black}1}                                                       &           & 40           & 652.4        & {-}            & {-}                                                                 & {-}                       & {\color{black}1}                                         \\ \cmidrule(r){1-6} \cmidrule(l){8-13} 
\end{tabular}
\caption{Results on the random instances on the biclique coupling graph. Each row shows the average values over 5 randomly generated instances. Times are clocktimes given in seconds. \label{Tab:RandomBiclique}}
\end{table}

\section{Conclusions}
In this paper we study an exact method for solving the NNCP in the gated quantum computing model by exploiting symmetries in the underlying formulation.

Starting from the shortest path formulation introduced by~\cite{MatsuoYamashita}, see~\eqref{Prob:SPP}, we study the algebraic structure of the underlying graph in Section~\ref{Sec:ExploitSym}. This graph is composed of a series of Cayley graphs of the symmetric group $\mathbb{S}_n$ generated by the transpositions in the coupling graph of the quantum system. We show that $\mathbb{S}_n \times \Aut(\Coup(E))$ is the full automorphism group of such Cayley graphs
for almost all coupling graphs.
Although the automorphism groups of specific Cayley graphs generated by transpositions have been studied before in the literature, we do not make any assumption on the underlying coupling graph apart from being connected. Next, we show how these subgroups are merged into a subgroup $G_X$ of the automorphism group of the entire graph, see~\eqref{Eq:GXreformulated}. One component of this subgroup is determined by the algebraic structure of the coupling graph, while the other component relies on a so-called fixing pattern $\mathcal{F}$ following from the quantum gates in the circuit, see Definition~\ref{Def:fixingpattern}. The orbit and orbital structures of the action of this group on the graph are also studied, leading in particular to an overview of the cardinalities of the corresponding quotients, see Table~\ref{Table:quotients}. 

By exploiting the convexity of~\eqref{Prob:SPP}, we reduce the symmetries in the formulation by averaging over all symmetric solutions using the Reynolds operator, see~\eqref{Def:Reynolds}. This leads to a more compact equivalent formulation~\eqref{Prob:RSPP} and its scaled variant~\eqref{Prob:RSPP_scale}. We show that this formulation is equivalent to a generalized network flow problem~\eqref{Prob:GNFP}. Due to the in-depth analysis on the orbit and orbital structure, these formulations can be explicitly constructed from scratch without the need to first construct the exponentially large Cayley graphs. A direct theoretical implication of our approach are the complexity results of Theorem~\ref{Thm:Complexity} and Corollary~\ref{Cor:bicliquecomplexity}, which reveal a class of polynomial time solvable special cases of the NNCP. 

The gain of using our approach compared to the base model~\eqref{Prob:SPP} is most vibrant in case the fixing pattern is less restrictive and the coupling graph is (highly) symmetric. We test our approach on three types of coupling graphs, for which we explicitly derive the key ingredients of our formulation, see Table~\ref{Tab:summarycoupling}. Our numerical results show that the gain in efficiency due to the exploitation of symmetries is very large. For each of the~84 real and~180 random instances, the total reduction in the number of variables and constraints is at least 90\% and 89.8\%, respectively, and this number grows with $n$ and $m$. The computation times are significantly reduced compared to the unreduced model, resulting in solving NNCP instances that are much larger than the ones considered so far in the literature. The largest instance we can solve contains 100~qubits and $1047$~quantum gates.

\medskip

Given that we are only at the beginning of the quantum era, related optimization problems such as the NNCP are likely to remain important in the near future. Based on the successful implementation of our symmetry-reduced solution approach, it would be interesting to consider the NNCP on other quantum architectures having a large symmetry group. 

{\footnotesize
\bibliography{QC}
}
\newpage
\appendix 
\section{Quantum gates and their two-qubit decompositions} \label{App:Gates}
Since the NNCP is only well-defined when a quantum circuit consists solely of one- or two-qubit gates, we have to decompose all gates that act on more than two gates. As indicated in Section~\ref{Sec:CompResults}, this task can be completed in lots of ways and performing this decomposition optimally can be seen as a research problem in itself. In this paper we apply the decomposition method used in~\cite{MulderijEtAl}, although the authors of~\cite{MulderijEtAl} already indicated that this method might be open for improvement. 

The quantum circuits that we consider in our experiments consist of the following types of quantum gates: one-qubit gates, two-qubit gates, three-qubit Peres gates, three- and four-qubit Fredkin gates and three-, four- and five-qubit Toffoli gates. Commonly used one-qubit gates are the Hadamard gate and the Pauli-gates, e.g., the Pauli-$X$-gate. When applying the Hadamard gate to a qubit in any state, it brings the qubit in a superposition state where it has an equal probability to be 0 or 1 upon measurement. The Hadamard gate in a quantum circuit is depicted as \raisebox{0.3em}{\Qcircuit @C=0.7em @R=1.4em   {
                & \gate{H} & \qw
            }}. The Pauli-$X$-gate is also known as the NOT gate and can be seen as its quantum analog. The NOT-gate is depicted as \raisebox{0.3em}{\Qcircuit @C=0.7em @R=1.4em   {
                & \targ & \qw
            }}. 
            
The most commonly used two-qubit gates are depicted in Figure~\ref{Fig:2qubit}. The controlled-NOT gate, also known as CNOT or Feynman gate, acts on a control qubit and a target qubit. If the control qubit is in state $|1\rangle$, a NOT-gate is applied to the target qubit, otherwise nothing happens. The SWAP gate swaps the states of the two qubits where it acts on. The controlled-$V$ and controlled-$V^\dag$ act similarly to the controlled-NOT gate, with the only difference that the unitary operation $V$ or~$V^\dag$ is applied to the target qubit. The operation $V$ and $V^\dag$ are the square root of the NOT-gate and its Hermitian conjugate, respectively. That is, if two controlled-$V$ gates are placed in succession, the result is similar to a controlled-NOT gate, while the identity gate is obtained when applying a controlled-$V$ and a controlled-$V^\dag$ gate in succession. 

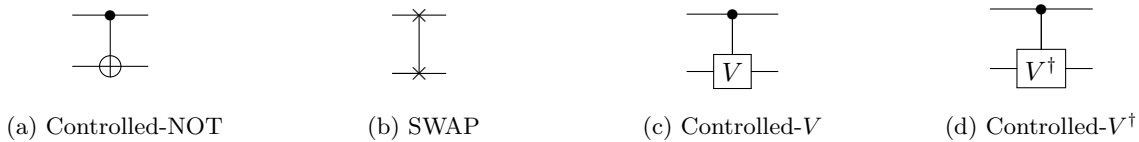
\begin{figure}[H]
\centering
    \begin{subfigure}[b]{0.2\textwidth}
         \centering
         \begin{tikzpicture}
            \node at (0,0) {$\begin{aligned}
             \Qcircuit @C=1em @R=1.4em   {
                & \ctrl{1} & \qw \\
                & \targ & \qw 
            }
            \end{aligned}$}; 
        \end{tikzpicture}
        \vspace{0.15cm}
        \caption{Controlled-NOT}
     \end{subfigure} \hfill
     \begin{subfigure}[b]{0.2\textwidth}
         \centering
         \begin{tikzpicture}
            \node at (0,0) {$\begin{aligned}
             \Qcircuit @C=1em @R=2.2em   {
                & \qswap & \qw \\
                & \qswap \qwx \ghost{$V^\dag$}& \qw 
            }
            \end{aligned}$};
        \end{tikzpicture}
        \vspace{0.2cm}
        \caption{SWAP}
     \end{subfigure}
     \hfill
     \begin{subfigure}[b]{0.2\textwidth}
         \centering
         \begin{tikzpicture}
            \node at (0,0) {$\begin{aligned}
             \Qcircuit @C=1em @R=1.4em   {
                & \ctrl{1} & \qw \\
                & \gate{V} \qwx & \qw 
            }
            \end{aligned}$}; 
        \end{tikzpicture}
        \caption{Controlled-$V$}
     \end{subfigure}
     \hfill
     \begin{subfigure}[b]{0.2\textwidth}
         \centering
         \begin{tikzpicture}
            \node at (0,0) {$\begin{aligned}
             \Qcircuit @C=1em @R=1.4em   {
                & \ctrl{1} & \qw \\
                & \gate{V^\dag} \qwx & \qw 
            }
            \end{aligned}$}; 
        \end{tikzpicture}
        \caption{Controlled-$V^\dag$}
     \end{subfigure}
    
    \caption{Overview of commonly used two-qubit quantum gates. \label{Fig:2qubit}}
\end{figure}

A Toffoli gate~\cite{Toffoli} is the multiple-control NOT gate. Acting on several control qubits and a single target qubit, a NOT gate is applied to the target qubit if all the control qubits are in state~$|1\rangle$. The three-qubit Toffoli gate is depicted in Figure~\ref{Fig:t3}, along with a possible decomposition into two-qubit gates, following the approach of~\cite{BarencoEtAl}.

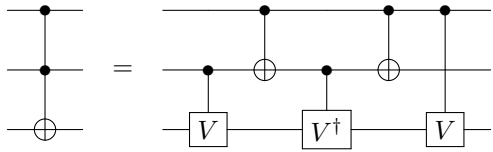
\begin{figure}[H]
\centering
    \begin{tikzpicture}
        \node at (0,0) {$\begin{aligned} \Qcircuit @C=1em @R=0.8em @!R {
& \ctrl{1} & \qw & & & & \qw & \ctrl{1} & \qw & \ctrl{1} & \ctrl{2} & \qw \\
& \ctrl{1} & \qw & &\quad = \qquad &  & \ctrl{1} & \targ & \ctrl{1} & \targ &  \qw & \qw\\
& \targ & \qw &  & & & \gate{V} & \qw & \gate{V^\dag} & \qw & \gate{V} & \qw
} \end{aligned}$}; 
    \end{tikzpicture}
    \caption{Decomposition of multiple-control Toffoli gate with two controls and a single target qubit. \label{Fig:t3}}
\end{figure} 
The Peres gate~\cite{Peres} is obtained from a combination of a two-qubit controlled-NOT gate and standard controlled-NOT gate. Following the approach from~\cite{HungEtAl}, the Peres gate can be decomposed into four two-qubit gates, as shown in Figure~\ref{Fig:p3}. 
\begin{figure}[H]
\centering
    \begin{tikzpicture}
        \node at (0,0) {$\begin{aligned}
        \Qcircuit @C=1em @R=1.4em @!R  {
        & \ctrl{1} & \ctrl{1} & \qw  \\ 
        & \ctrl{1} & \targ & \qw  \\
        & \targ & \qw & \qw 
    }
    \end{aligned} \quad = \quad  \begin{aligned}
        \Qcircuit @C=1em @R=0.8em @!R  {
        & \qw     & \ctrl{2} & \ctrl{1} & \qw & \qw  \\ 
        & \ctrl{1} & \qw & \targ  & \ctrl{1} & \qw \\
        & \gate{V^\dag} & \gate{V^\dag} & \qw & \gate{V} & \qw 
    }
    \end{aligned}$}; 
    \end{tikzpicture}
    \caption{Decomposition of Peres gate on three qubits. \label{Fig:p3}}
\end{figure}
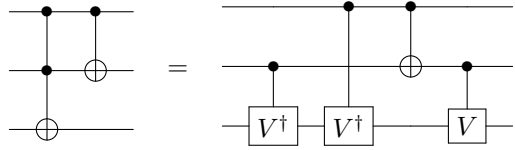

The Fredkin gate~\cite{FredkinToffoli} operates on three qubits as a controlled-SWAP gate. If the state of the control qubit is $|1\rangle$, then a SWAP gate on the two target qubits is performed. The decomposition into two-qubit gates that we adapt here is the same as the one considered in ~\cite{MulderijEtAl,Revlib}, see Figure~\ref{Fig:f3}

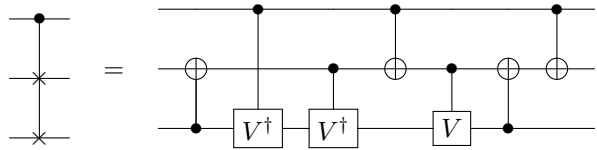
\begin{figure}[H]
\centering
    \begin{tikzpicture}
        \node at (0,0) {$\begin{aligned}
        \Qcircuit @C=1em @R=2em @!R  {
        & \ctrl{1} &  \qw  \\ 
        & \qswap &  \qw  \\
        & \qswap \qwx &  \qw 
    }
    \end{aligned} \quad = \quad \begin{aligned}
        \Qcircuit @C=1em @R=0.8em @!R  {
        & \qw &  \ctrl{2}  & \qw &  \ctrl{1} & \qw &  \qw & \ctrl{1} & \qw \\ 
        & \targ &  \qw & \ctrl{1} & \targ &  \ctrl{1} & \targ & \targ & \qw \\
        & \ctrl{-1} \qwx &  \gate{V^\dag} & \gate{V^\dag} & \qw & \gate{V} & \ctrl{-1} & \qw & \qw
    }
    \end{aligned}$}; 
    \end{tikzpicture}
    \caption{Decomposition of Fredkin gate (controlled swap gate) with one control qubit. \label{Fig:f3}}
\end{figure}
Finally, we consider the four- and five qubit variants of the Fredkin and Toffoli gate. The functionality of these gates is similar to their three-qubit implementation, only the number of control qubits is larger. The four-qubit Fredkin gate can be decomposed as shown by~\cite{Alhagi}, see Figure~\ref{Fig:f4}. Fredkin gates on a larger number of qubits do not appear in our experiments. 
\begin{figure}[H]
\centering
    \begin{tikzpicture}
        \node at (0,0) {$ \begin{aligned}
        \Qcircuit @C=1em @R=2em @!R  {
           & \ctrl{1} & \qw \\
            & \ctrl{1} & \qw \\
           &  \qswap & \qw \\
            & \qswap \qwx & \qw 
        }
    \end{aligned} \quad = \quad \begin{aligned}
    \Qcircuit @C=1em @R=0.8em @!R {
        & \qw & \ctrl{3} & \ctrl{1} & \qw & \ctrl{1} & \qw & \qw & \qw &  \ctrl{2} & \qw & \qw & \qw &  \ctrl{2} & \qw & \qw & \qw \\ 
        & \qw & \qw      & \targ    & \ctrl{2} & \targ & \ctrl{2} &  \ctrl{1} & \qw & \qw & \qw & \ctrl{1} & \qw & \qw & \qw & \qw  & \qw \\
        & \targ & \qw      & \qw  & \qw & \qw & \qw & \targ & \ctrl{1} & \targ & \ctrl{1} & \targ & \ctrl{1} & \targ & \ctrl{1} & \targ & \qw \\
       & \ctrl{-1} & \gate{V}    & \qw & \gate{V^\dag} & \qw & \gate{V} &  \qw & \gate{V^\dag} & \qw & \gate{V} &\qw &  \gate{V^\dag} & \qw & \gate{V} & \ctrl{-1} & \qw \\
    }
    \end{aligned}$}; 
    \end{tikzpicture}
    \caption{Decomposition of Fredkin gate (controlled swap gate) with two control qubits. \label{Fig:f4}}
\end{figure}
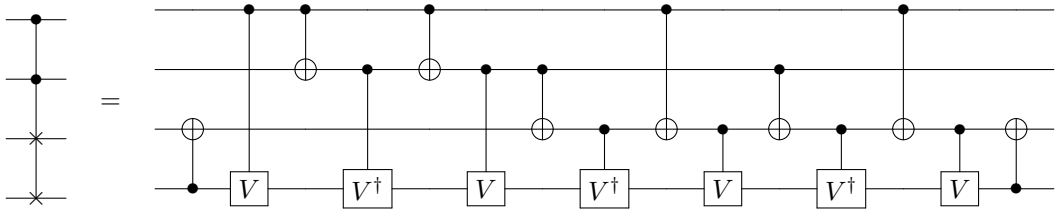

Finally, the four- and five-qubit Toffoli gates are shown in Figure~\ref{Fig:t4} and~\ref{Fig:t5}. The decompositions shown here follow from the construction derived in~\cite{BarencoEtAl}. Toffoli gates on more than five qubits do not appear in our experiments. 
\begin{figure}[H]
\centering
    \begin{tikzpicture}
        \node at (0,0) {$\begin{aligned}
    \Qcircuit @C=1em @R=1.4em @!R  {
        & \ctrl{1} & \qw  \\ 
        & \ctrl{1} & \qw  \\
        & \ctrl{1} & \qw  \\
        & \targ    & \qw  \\
    }
    \end{aligned} \quad = \quad \begin{aligned}
    \Qcircuit @C=1em @R=0.8em @!R {
        & \ctrl{3} & \ctrl{1} & \qw & \ctrl{1} & \qw & \qw & \qw &  \ctrl{2} & \qw & \qw & \qw &  \ctrl{2} & \qw & \qw \\ 
        & \qw      & \targ    & \ctrl{2} & \targ & \ctrl{2} &  \ctrl{1} & \qw & \qw & \qw & \ctrl{1} & \qw & \qw & \qw & \qw \\
        & \qw      & \qw  & \qw & \qw & \qw & \targ & \ctrl{1} & \targ & \ctrl{1} & \targ & \ctrl{1} & \targ & \ctrl{1} & \qw \\
        & \gate{V}    & \qw & \gate{V^\dag} & \qw & \gate{V} &  \qw & \gate{V^\dag} & \qw & \gate{V} &\qw &  \gate{V^\dag} & \qw & \gate{V} & \qw \\
    }
    \end{aligned}$}; 
    \end{tikzpicture}
    \caption{Decomposition of multiple-control Toffoli gate with three controls and a single target qubit. \label{Fig:t4}}
\end{figure}
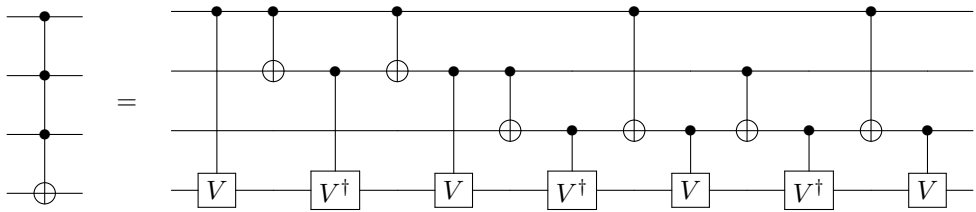

\begin{landscape}

\begin{figure}[H]
\centering
    \begin{tikzpicture}
        \node at (0,0) {$\begin{aligned}
    \Qcircuit @C=0.8em @R=1.4em @!R  {
        & \ctrl{1} & \qw  \\ 
        & \ctrl{1} & \qw  \\
        & \ctrl{1} & \qw  \\
        & \ctrl{1} & \qw  \\
        & \targ    & \qw  
    }
    \end{aligned} \quad = \quad \begin{aligned}
    \Qcircuit @C=0.8em @R=0.8em @!R  {
        & \ctrl{4} & \ctrl{1} & \qw & \ctrl{1} & \qw & \qw & \qw & \ctrl{2} & \qw & \qw & \qw & \ctrl{2} & \qw & \qw & \qw & \ctrl{3} & \qw & \qw & \qw & \ctrl{3} & \qw & \qw & \qw &\ctrl{3} & \qw & \qw & \qw & \ctrl{3} & \qw & \qw \\ 
        & \qw & \targ & \ctrl{3} & \targ & \ctrl{3} & \ctrl{1} & \qw & \qw & \qw &  \ctrl{1} & \qw & \qw & \qw & \qw & \qw & \qw & \qw & \ctrl{2} & \qw & \qw & \qw & \qw & \qw & \qw & \qw & \ctrl{2} & \qw & \qw & \qw & \qw \\
        & \qw & \qw & \qw & \qw & \qw & \targ & \ctrl{2} & \targ & \ctrl{2} & \targ & \ctrl{2} & \targ & \ctrl{2} &  \ctrl{1} & \qw & \qw & \qw & \qw & \qw & \qw & \qw & \ctrl{1} & \qw & \qw & \qw & \qw & \qw & \qw & \qw & \qw\\
        & \qw & \qw & \qw &\qw & \qw & \qw & \qw &\qw & \qw & \qw & \qw & \qw & \qw & \targ & \ctrl{1} & \targ & \ctrl{1} & \targ & \ctrl{1} & \targ & \ctrl{1} & \targ & \ctrl{1} & \targ & \ctrl{1} & \targ & \ctrl{1} & \targ & \ctrl{1} & \qw \\
        & \gate{V}    & \qw  & \gate{V^\dag} & \qw & \gate{V} &\qw & \gate{V^\dag} & \qw & \gate{V} & \qw & \gate{V^\dag} & \qw & \gate{V} & \qw & \gate{V^\dag} & \qw & \gate{V} & \qw & \gate{V^\dag} & \qw & \gate{V} & \qw & \gate{V^\dag} & \qw & \gate{V} & \qw & \gate{V^\dag} & \qw & \gate{V} & \qw
    }
    \end{aligned}$}; 
    \end{tikzpicture}
    \caption{Decomposition of multiple-control Toffoli gate with four controls and a single target qubit. \label{Fig:t5}}
\end{figure}
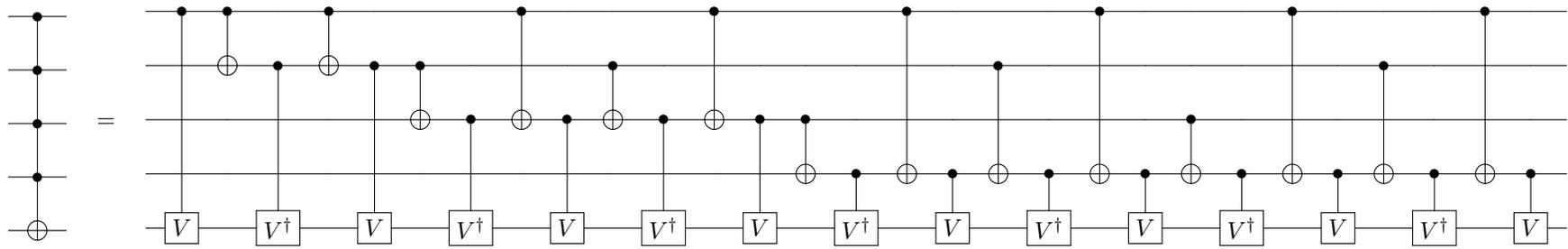

\end{landscape}

\end{document}